\pdfoutput=1
\documentclass[11pt,namelimits,sumlimits,a4paper]{amsart}
\usepackage{cite}
\usepackage{comment}

\usepackage{amssymb,amsmath}
\usepackage[mathscr]{eucal}
\usepackage{slashed}

\usepackage[latin1]{inputenc}
\usepackage[T1]{fontenc}
\usepackage{amsfonts}
\usepackage{fancyhdr}
\usepackage{graphicx}
\usepackage{amsmath}
\usepackage{amsthm}
\usepackage{amsmath,amscd}
\usepackage{latexsym}
\usepackage{cite}
\usepackage{microtype}
\usepackage{amssymb,amsmath}

\usepackage{comment}
\usepackage[mathscr]{eucal}
\usepackage{slashed}
\usepackage{times,psfrag}
\usepackage{mathtools}
\usepackage[pdftitle={Infinitely many new families of complete cohomogeneity one G2-manifolds},hidelinks]{hyperref}
\usepackage{multirow}
\usepackage{longtable}
\usepackage{bm}
\usepackage[all]{xypic}
\usepackage{fancyhdr}
\usepackage{float}
\usepackage{esint}
\usepackage{layout}
\usepackage{enumitem}
\usepackage{lmodern}
\usepackage[rgb]{xcolor}

\usepackage[abbrev,msc-links]{amsrefs}

\numberwithin{equation}{section}
\newtheorem{theorem}[equation]{Theorem}
\newtheorem{lemma}[equation]{Lemma}
\newtheorem{prop}[equation]{Proposition}

\theoremstyle{definition}
\newtheorem{definition}[equation]{Definition}
\newtheorem{example}[equation]{Example}

\theoremstyle{remark}
\newtheorem{remark}[equation]{Remark}
\newtheorem*{remark*}{Remark}

\newcommand{\abs}[1]{\lvert#1\rvert}

\newcommand{\ie}{\emph{i.e.} }
\newcommand{\eg}{\emph{e.g.} }
\newcommand{\cf}{\emph{cf.} }

\newcommand{\beq}{\begin{equation}}
\newcommand{\eeq}{\end{equation}}
\newcommand{\bea}{\begin{eqnarray}}
\newcommand{\eea}{\end{eqnarray}}

\newcommand{\C}{\mathbb{C}}

\newcommand{\R}{\mathbb{R}}
\newcommand{\Q}{\mathbb{Q}}
\newcommand{\Z}{\mathbb{Z}}
\newcommand{\N}{\mathbb{N}}

\newcommand{\HH}{\mathbb{H}}
\newcommand{\CP}{\mathbb{CP}}
\newcommand{\PP}{\mathbb{P}}
\newcommand{\Sph}{\mathbb{S}}
\newcommand{\ra}{\rightarrow}

\newcommand{\diag}{\operatorname{diag}}

\newcommand{\dvol}{\operatorname{dv}}
\newcommand{\Real}{\operatorname{Re}}
\newcommand{\Imag}{\operatorname{Im}}

\newcommand{\Lie}[1]{\mathfrak{#1}}

\newcommand{\Aut}{\text{Aut}}

\newcommand{\tu}[1]{\textup{#1}}

\newcommand{\gtwo}{\ensuremath{\textup{G}_2}}

\newcommand{\gtstr}{\gtwo--structure}

\newcommand{\gtmfd}{\gtwo--manifold}
\newcommand{\gtmetric}{\gtwo--metric}
\newcommand{\gthol}{\gtwo--holonomy\ }

\newcommand{\hk}{hyperK\"ahler}

\newcommand{\unitary}[1]{\textup{U$(#1)$}}

\newcommand{\sunitary}[1]{\textup{SU$(#1)$}}

\newcommand{\suthreestr}{\sunitary{3}--structure}

\newcommand{\sorth}[1]{\textup{SO$(#1)$}}

\newcommand{\orth}[1]{\textup{O$(#1)$}}

\newcommand{\taubnut}{\ensuremath{g_{\textup{tn}}^m}}
\newcommand{\csgt}{\ensuremath{\varphi_{\tu{cs}}^c}}
\newcommand{\triple}[1]{\boldsymbol{#1}}
\newcommand{{\isomgtc}}{\ensuremath{\sunitary{2}^3 \rtimes S_3}}

\newcommand{\gbs}{\ensuremath{\tu{G}^{\tu{bs}}}}

\def\co{\colon\thinspace}

\newcounter{mtheorem}

\newtheoremstyle{mystyle}%
  {}%
  {}%
  {\itshape}%
  {}%
  {\bfseries}%
  {.}%
  { }%
  {}%

\theoremstyle{mystyle}
\newtheorem{mtheorem}[mtheorem]{Theorem}

\begin{document}

\title[Infinitely many families of complete cohomogeneity one $\gtwo$--manifolds]{Infinitely many new families of complete cohomogeneity one $\gtwo$--manifolds: \gtwo~analogues of the Taub--NUT and Eguchi--Hanson spaces}
\author[L.~Foscolo]{Lorenzo~Foscolo}
\author[M.~Haskins]{Mark~Haskins}
\author[J.~Nordstr\"om]{Johannes~Nordstr\"om}

\begin{abstract}
We construct infinitely many new $1$-parameter families of simply connected complete noncompact  \gtmfd s with controlled geometry at infinity. The generic member of each family has so-called asymptotically locally conical (ALC) geometry. However, the nature of the asymptotic geometry changes at two special parameter values: at one special value we obtain a unique member of each family with asymptotically conical (AC) geometry;  
on approach to the other special parameter value the family of metrics collapses to an AC Calabi--Yau $3$-fold.
Our infinitely many new diffeomorphism types of AC \gtmfd s are particularly noteworthy: previously the three examples constructed by Bryant and Salamon in 1989 furnished the only known 
simply connected AC \gtmfd s.

We also construct a closely related conically singular \gthol space: away from a single isolated conical singularity, where the geometry becomes asymptotic to the \gtwo--cone over the standard nearly K\"ahler structure on 
the product of a pair of 3-spheres, the metric is smooth and it has ALC geometry at infinity. 
We argue that this conically singular ALC \gtwo--space is the natural \gtwo\ analogue of the Taub--NUT metric 
in 4-dimensional hyperK\"ahler geometry and that our new AC \gtmetric s are  all analogues of the Eguchi--Hanson metric, 
the simplest ALE hyperK\"ahler mani\-fold. Like the Taub--NUT and Eguchi--Hanson metrics, 
all our examples are cohomogeneity one, \ie  they admit an isometric Lie group action whose generic orbit has codimension one. 
\end{abstract}

\maketitle

\section{Introduction}
Over the past 40 years cohomogeneity one Riemannian metrics, \ie  metrics admitting an isometric Lie group action with generic orbit of codimension one, have played a distinguished role in the construction of complete Ricci-flat or Einstein metrics, 
particularly in the cases of metrics with special or exceptional holonomy. Throughout its history the subject has attracted considerable interest from both mathematicians and theoretical physicists, 
with key contributions and important new insights from both communities, \eg \cites{Eguchi:Hanson,BGPP:1978,Calabi:AC:CY,Atiyah:Hitchin,Bryant:Salamon,GPP:1990,Candelas:delaOssa,Stenzel,CGLP:Spin7:NP,BGGG,Bohm:Invent,Eschenburg:Wang}. 
Moreover, until very recently \cites{FHN:ALC:G2:from:AC:CY3}, the \emph{only} known complete noncompact \gtmfd s, 
\ie Riemannian $7$-manifolds admitting metrics with holonomy group the compact exceptional Lie group \gtwo,  were of cohomogeneity one.

The cohomogeneity one property affords a reduction of the system of nonlinear partial differential equations that characterises an Einstein metric or a holonomy reduction, to a system of nonlinear ordinary differential equations (ODEs). In some cases these ODEs can be integrated explicitly and completeness of such metrics can then be approached directly, 
\eg \cites{Eguchi:Hanson,Calabi:AC:CY,Atiyah:Hitchin,Bryant:Salamon,Stenzel,Candelas:delaOssa,Dancer:Wang:KEcoh1,BGPP:1978,CGLP:Spin7:JGP,CGLP:Spin7:NP,BGGG}. 
In more complicated cases 
qualitative methods from the theory of ODEs are needed to prove the existence and to establish
qualitative properties of solutions; it is usually then a significant challenge to understand when such (non explicit) solutions give rise to complete metrics, \eg 
\cites{Bohm:Invent,Bohm,Foscolo:Haskins,Bogoyavlenskaya,Bazaikin:Spin7}. 

Prior to this paper a very limited number of complete cohomogeneity one 
\gtmetric s had been constructed: the three rigid \emph{asymptotically conical} (AC) examples constructed by Bryant--Salamon in 1989 \cite{Bryant:Salamon} and
a $1$-parameter family 
of so-called \emph{asymptotically locally conical} (ALC) examples constructed in 2013 by Bogoyavlenskaya \cite{Bogoyavlenskaya}. 
The asymptotic geometry of ALC spaces will be described a little later in this introduction.
The existence of the latter family, denoted $\mathbb{B}_7 $ in the physics literature, 
was first predicted in 2001 by Brandhuber--Gomis--Gubser--Gukov based on an informal analysis of deforming away from 
a single explicit ALC \gtmetric\  that they constructed~\cite{BGGG}.
In addition to these rigorously constructed examples, numerical analysis of the relevant ODE systems by Brandhuber, Cveti\v{c}--Gibbons--L\"u--Pope and later Hori--Hosomichi--Page--Rabad\'an--Walcher, suggested 
the existence of a further three $1$-parameter families of ALC \gtmetric s, denoted $\mathbb{A}_7$ \cite{A7:ALC}, $\mathbb{C}_7$ \cites{CGLP:C7,Brandhuber} and $\mathbb{D}_7$ \cites{CGLP:M:Conifolds,Brandhuber}, in the physics literature. 
Up to discrete symmetries, in the first case the group $G$ acting is $\sunitary{2} \times \sunitary{2}$, while 
the latter two cases have the enhanced symmetry group $G=\sunitary{2} \times \sunitary{2} \times \unitary{1}$.

\subsection*{Main results} In the current paper we revisit the theory of noncompact \gtmfd s with a cohomogeneity one action of $G=\sunitary{2} \times \sunitary{2}$,
with a particular focus on the case where the symmetry enhances to  $G=\sunitary{2} \times \sunitary{2} \times \unitary{1}$.
Our main results are stated in Theoreoms \ref{mthm:CS}--\ref{mthm:Classification} later in this introduction. Theorem \ref{mthm:B7:D7} proves the existence of the previously predicted $1$-parameter family of cohomogeneity one \gtmetric s $\mathbb{D}_7$, and also gives a new proof of the existence of the $\mathbb{B}_7 $ family.
Theorems \ref{mthm:AC} and \ref{mthm:ALC} construct  \emph{infinitely} many new $1$-parameter families of complete simply connected cohomogeneity one \gtmfd s, all with $G=\sunitary{2} \times \sunitary{2} \times \unitary{1}$. Theorem \ref{mthm:AC} is particularly noteworthy because it constructs infinitely many new asymptotically conical 
\mbox{\gtmetric s}; previously only the three classical examples due to Bryant--Salamon, dating back to 1989, were known.
The metrics constructed in Theorems  \ref{mthm:AC} and \ref{mthm:ALC} include the previously predicted $1$-parameter family of \gtmetric s $\mathbb{C}_7$ as a special case. Theorem \ref{mthm:Classification} states that our existence results recover all complete simply connected $\sunitary{2} \times \sunitary{2} \times \unitary{1}$--invariant \gtmfd s. 

The general qualitative features of all these $1$-parameter 
families turn out to be the same.
We will give an explanation for this,
which relies on the existence of a new \emph{singular} cohomogeneity one \gtmetric~$\varphi_{\tu{cs}}$ on 
$(0,\infty) \times S^3 \times S^3$ which is forward complete with ALC geometry as $t \to \infty$ 
but which as $t \to 0$ has a conically singular (CS) end modelled on the \gtwo--cone $C$ over the standard homogeneous
nearly K\"ahler structure on $S^3 \times S^3$: see Theorem \ref{mthm:CS}.

\subsection*{Motivation from highly collapsed \texorpdfstring{\gtmetric s}{G2 metrics}}
Recently, in \cite{FHN:ALC:G2:from:AC:CY3}, we developed a new analytic method for 
the construction of complete noncompact \gtmetric s with ALC geometry. This method is very powerful:  
it gives the first constructions of complete noncompact \gtmetric s with very little 
symmetry; easily yields \gtmetric s on infinitely many different simply connected $7$-manifolds 
and also produces high-dimensional families of \gtmetric s. 
The method of construction necessarily produces \gtmetric s 
that are highly collapsed, that is these \gtmetric s are Gromov--Hausdorff close to a complete 
noncompact Calabi--Yau $3$-fold $B$. Since smooth complete \gtmetric s typically deform in a smooth 
finite-dimensional moduli space, 
it is natural to try to understand deformations of these highly collapsed \gtmetric s,
both in a local and in a global sense. 

The local deformation theory of \gtmetric s is already well established in other settings:
in the smooth compact case by Joyce \cite{Joyce:book}, 
in the asymptotically cylindrical case by the third author \cite{Nordstrom:thesis} 
and in the asymptotically conical and conically singular cases by Karigiannis--Lotay \cite{Karigiannis:Lotay}. 
This local deformation theory can be adapted to the ALC setting 
once a Fredholm theory for elliptic operators on suitable weighted spaces on ALC spaces is developed. 
We have developed such analytic tools, the details of which, together with the local deformation theory, 
will appear elsewhere. 

\enlargethispage{\baselineskip}
However, currently it seems (far) out of reach to hope 
to understand \emph{large} deformations of our highly collapsed 
\gtmetric s in any generality. In this paper we focus on some particular cases 
where this large deformation question turns out to be tractable.
Although the general \gtmetric\  we construct using the methods of \cite{FHN:ALC:G2:from:AC:CY3} 
admits only a circle symmetry, the symmetries are enhanced if the base Calabi--Yau $B$ 
admits symmetries. In particular,
if the limiting Calabi--Yau $3$-fold $B$ has cohomogeneity one, then so do our
highly collapsed \gtmetric s.
It is then not difficult to deduce that there are infinitely many topological 
types of simply connected $7$-manifold that admit complete noncompact \gtmetric s of cohomogeneity 
one.
The goal of this paper is to understand all these complete noncompact \gtmetric s including \emph{far from the highly collapsed regime}.
We achieve this by using cohomogeneity one methods. 
It turns out that the behaviour of these moduli spaces of complete cohomogeneity one \gtmetric s 
on each of the $7$-manifolds in question is qualitatively the same.
However, despite the qualitative similarities to the already-understood $\mathbb{B}_7$ family, 
constructing these new families of complete cohomogeneity one solutions will require various new ideas.  

\pagebreak[2]
\subsection*{Geometry of the  \texorpdfstring{$\mathbb{B}_7$}{B7} family}
We now turn to a description of the key properties of the $\mathbb{B}_7$ family, 
since these will be shared by all the cohomogeneity one families we construct and 
moreover motivate our approach to constructing these new families.
The $\mathbb{B}_7$ family is a $1$-parameter family of complete cohomogeneity one \gtmetric s $g_\alpha$
on $S^3 \times \R^4$, parametrised by a finite interval, say $(0,1]$. 
The group acting is $\sunitary{2} \times \sunitary{2} \times \unitary{1}$ with 
principal orbits diffeomorphic to $S^3 \times S^3$ and singular orbit $S^3 \times \{0\}$. 
$\unitary{1}$ acts trivially on $S^3$ and as the standard Hopf action on $\R^4$. 
The quotient space by the (not fixed point free) $\unitary{1}$ action is nevertheless a manifold,  
homeomorphic to $S^3 \times \R^3$.
For $\alpha \in (0,1)$, all the metrics $g_\alpha$ share the same basic asymptotic behaviour: they are complete with
(submaximal) volume growth of large geodesic balls of order $r^6$. 

\subsubsection*{ALF and ALC spaces}
To model finer asymptotic behaviour of $n$-dimensional Ricci-flat spaces with volume growth of order $r^{n-1}$ it is natural to consider a metric $g_\infty$ on a circle bundle $M$ over 
a Ricci-flat cone $(C,g_C)$ of dimension $n-1$ of the form 
\[g_\infty = g_C + \ell ^2 \theta^2,\] where $\ell >0$ is some constant 
and $\theta$ is some fixed connection on the circle bundle $M$. 
Thus the model metric $g_\infty$
is a Riemannian submersion over the cone $C$ with circle fibres of constant length $2\pi\ell$.
Under appropriate conditions on curvature decay, 
we might then expect that an $n$-dimensional Ricci-flat manifold with volume growth of order $r^{n-1}$ must,  
outside a compact set, become asymptotic to (an exterior domain in) such a model end.
If $n=4$, then (up to a possible $\Z_2$ quotient) necessarily $C=\R^3$.
Such metrics were therefore termed \emph{asymptotically locally flat} (ALF) spaces by physicists.
The higher-dimensional analogues of ALF spaces were subsequently termed \emph{asymptotically locally conical} (ALC) spaces \cite{CGLP:Spin7:NP}.
For $\alpha \in (0,1)$, any metric $g_\alpha$ in the $\mathbb{B}_7$ family is ALC. 
For any ALC \mbox{\gtmfd}~the cone $C$ should be a $3$-dimensional Calabi--Yau cone.
Many such cones are now known to exist, \eg see the discussion in \cite{FHN:ALC:G2:from:AC:CY3}*{\S 9} and references therein. Hence, unlike the very rigid situation for ALF hyperK\"ahler $4$-manifolds, there are many different asymptotic models for ALC \gtmetric s.
For all ALC metrics in the $\mathbb{B}_7$ family the Calabi--Yau cone $C$ is the conifold, \ie 
the cone over the standard homogeneous Sasaki--Einstein metric on $S^2 \times S^3$, 
viewed as the homogeneous space $\sunitary{2} \times \sunitary{2} / \Delta \unitary{1}$.

\subsubsection*{Transitions in the asymptotic geometry of the $\mathbb{B}_7$ family}
As we approach either endpoint of the interval $(0,1]$ the asymptotic geometry of the metrics $g_\alpha$
degenerates. Fixing the scale of $g_\alpha$ by requiring the size of the
singular orbit $S^3\times\{0\} \subset S^3 \times \R^4$ to be constant,
as $\alpha \ra 0$ we find that $\ell \ra 0$ and the \gtmetric s
collapse to the Stenzel metric \cites{Stenzel,Candelas:delaOssa} on the smoothing of the conifold, a well-known example of a cohomogeneity one AC Calabi--Yau $3$-fold.
Collapse occurs with bounded curvature, except close to the singular $S^3$, which is fixed by the $\unitary{1}$ action.

As $\alpha \ra 1$ instead $\ell \ra \infty$ and the metric $g_1$ has a different
asymptotic behaviour: it is still complete but the volume growth jumps from submaximal $r^6$ up to maximal growth $r^7$. 
The geometry at infinity of $g_1$ is \emph{asymptotically conical} (AC), \ie it is modelled by a \gthol cone 
over a smooth nearly K\"ahler $6$-manifold: in this case $S^3 \times S^3$ 
endowed with its homogeneous nearly K\"ahler structure. In fact, the metric $g_1$ is the classical Bryant--Salamon \gtmetric\  on $S^3 \times \R^4$ \cite{Bryant:Salamon}.

When the ALC/AC transition occurs at $\alpha =1$ nothing catastrophic happens to the \emph{local geometry}, 
in the sense that 
there is an extended $1$-parameter family $g_\alpha$, $\alpha \in (0,\infty)$, 
of local cohomogeneity one \gtmetric s that continues to 
close smoothly in the neighbourhood of the singular orbit  $S^3\times \{0\}$. However, for $\alpha >1$,  
these local solutions $g_\alpha$ are incomplete. 
So there are two ``phases'' of locally well-behaved solutions:
one consisting entirely of complete ALC metrics and the other entirely of incomplete metrics. 
The AC solution appears at the transition between these two phases.

\subsection*{A conically singular ALC \texorpdfstring{\gtwo--space}{G2-space}, its desingularisations and the \texorpdfstring{\gtwo--flop}{G2-flop}}
Given the freedom to rescale any \gtmetric, when describing the moduli space of \gtmetric s it is sometimes natural to impose a choice that breaks the scale invariance, and so describes solutions up to scale. Different ways to fix this scale invariance can however lead to different behaviour of solutions 
in various limiting regimes.
In the above description of the $\mathbb{B}_7$ family scale invariance was broken by fixing the size of the singular orbit $S^3$.
For a family of ALC metrics there is another geometrically natural way to break scale invariance: keep the length $\ell$ of the asymptotic circle of all the metrics fixed. 
Adopting this alternative scale fixing forces the size of the singular orbit to vary within the $\mathbb{B}_7$ family and the size  shrinks to zero as $\alpha$ approaches $1$, where the ALC/AC transition occurred. 

If we let $\tilde g_\alpha$ denote this rescaling of the $\mathbb{B}_7$ metrics, then there is an obvious guess for the behaviour of the limit of $\tilde g_\alpha$ as $\alpha \to 1$. 
There should exist a cohomogeneity one \gtmetric\  on $(0,\infty) \times S^3 \times S^3$ with the following properties:
at the end where $t \ra \infty$ the metric is forward complete with ALC asymptotics;  
at the end where $t \ra 0$ we have a conically singular (CS) end, \ie the metric completion over $t=0$ has an isolated conical singularity, modelled on the \gtwo--cone $C$ over the homogeneous nearly K\"ahler structure on $S^3 \times S^3$.

Our first main result confirms the existence of this CS ALC \gtwo--space:
its existence and significance does not seem to have been anticipated in the physics literature. 

\theoremstyle{mystyle}
\begin{mtheorem}
\label{mthm:CS}
Let $\tu{C}$ be the {\gtwo}--cone over the homogeneous nearly K\"ahler structure on $\tu{S}^3\times\tu{S}^3$.
\begin{enumerate}[leftmargin=0.75cm]
\item There exists a $1$-parameter family of $\sunitary{2}\times\sunitary{2}\times\unitary{1}$--invariant torsion-free {\gtstr s}
$\csgt$, $c\in\R$, defined on $(0,\epsilon)\times \tu{S}^3\times \tu{S}^3$ for some $\epsilon>0$ and such that the metric $g_{\csgt}$ has a conical singularity as $t\ra 0$ asymptotic to the cone $\tu{C}$. If $c=0$ the solution coincides with the cone $C$.
 
\item
\label{exist:CS:ALC}
If $c>0$ then the local conically singular solution from \textup{(i)}
is forward complete and extends to a
torsion-free \gtstr\  on $(0,\infty) \times \sunitary{2} \times \sunitary{2}$ 
with a conically singular end as $t \ra 0$ and an ALC end as $t \ra \infty$. 
All solutions with $c>0$ differ only by rescaling. 
\item
If $c<0$ then the local conically singular solution from \textup{(i)}
 is forward incomplete. 
\end{enumerate}
\end{mtheorem}

\subsubsection*{Desingularising CS \gtmetric s}
The existence of the CS ALC \gtmetric~constructed in Theorem \ref{mthm:CS} 
immediately suggests the possibility of obtaining families of smooth ALC \gtmetric s from it, by a desingularisation 
procedure.
Assuming the existence of a compact CS \gtwo--space $(X,g_{\tu{cs}})$, Karigiannis \cite{Karigiannis} used a gluing method to find a $1$-parameter family of \mbox{\gtmetric s} $g_t$ on a smooth compact $7$-manifold $M$ that degenerates to $(X,g_{\tu{cs}})$ as $t \to 0$. More specifically, he assumed that the \gtwo--cone $C$ that models the singularity 
admits a \gtwo--desingularisation, \ie a complete AC \mbox{\gtmfd} asymptotic at infinity to the given cone $C$.
The smooth compact manifold $M$ is obtained from $X$ by first replacing a neighbourhood of the conical singularity with
a rescaled copy of the AC manifold
and then correcting to \gthol by analytic methods. In the compact case there are two significant obstacles to turning Karigiannis' desingularisation method into a practical way to 
construct smooth \gtmetric s:
firstly,
it is still unknown how to produce \emph{any} compact CS \gtwo--spaces; secondly, topological obstructions to smoothing do occur in the compact setting.

One can also consider the same \gtwo--desingularisation procedure in the setting of noncompact CS \gtwo--spaces. 
In the specific case of the CS ALC \gtwo--space constructed in Theorem \ref{mthm:CS}
the classical Bryant--Salamon AC \gtmetric~on $S^3 \times \R^4$ provides a \gtwo--desingularisation 
of the cone over the standard nearly K\"ahler structure on $S^3 \times S^3$.
For instance, for $\delta > 0$ small and $\alpha \in (1-\delta, 1)$, the $\mathbb{B}_7$ metric $\tilde g_\alpha$ above should arise this way
as a \gtwo--desingularisation of the CS ALC \gtmetric~of Theorem \ref{mthm:CS}.

To adapt Karigiannis' approach to the noncompact ALC setting requires nontrivial
additional analytic work to be undertaken. 
However, the CS ALC setting has two notable advantages: 
(i) we have already proven the existence of at least one CS ALC \gtwo--space; 
(ii) the obstructions to smoothing present in the compact setting no longer arise, essentially because of the extra freedom to vary the asymptotic geometry.
We will give the details of this ALC \gtwo--desingularisation method elsewhere, since this also entails developing the requisite weighted analysis on ALC spaces.

In fact, there are three variants of the AC Bryant--Salamon \gtmetric~on $S^3 \times \R^4$, that 
are equivalent up to diffeomorphism, but not up to $\sunitary{2} \times \sunitary{2} \times \unitary{1}$--equivariant diffeomorphism. 
Depending on which variant is used to perform the desingularisation one obtains three different $1$-parameter families 
of smooth complete ALC \gtmetric s. Naturally, one family is the $\mathbb{B}_7$ family, whose degeneration behaviour
motivated us to seek the CS ALC solution in the first place.
Both the other families are versions of the so-called $\mathbb{D}_7$ family:
families of \gtmetric s that in the collapsed limit arise from the two small resolutions of the conifold
(whereas the $\mathbb{B}_7$ family collapses to the smoothing of the conifold).
The transition between the $\mathbb{B}_7$ and $\mathbb{D}_7$ family through the CS ALC space of Theorem A 
gives a metric version of the so-called \gtwo--flop, whose physical significance 
was emphasised by Acharya and Atiyah--Maldacena--Vafa in the context of large $N$ duality \cites{Acharya:SYM,Atiyah:Maldacena:Vafa}.

Since both the CS ALC \gtmetric~of Theorem A and the Bryant--Salamon AC metric 
have cohomogeneity one, then one would hope
to be able to prove the existence of these particular \gtwo--desingularisations more directly by ODE methods.
This is indeed the case. 
\begin{mtheorem}
\hfill
\label{mthm:B7:D7}
\begin{enumerate}[leftmargin=0.75cm,topsep=0cm,parsep=0cm]
\item
\label{exist:B7}
There exists a $1$-parameter family (up to scale) of \,$\sunitary{2} \times \sunitary{2} \times \unitary{1}$--invariant 
complete \gtmetric s on $S^3 \times \R^4$ with ALC geometry where
$\unitary{1}$ acts trivially on $S^3$ and via the standard Hopf action on $\R^4$. 
\item
\label{exist:D7}
There exists a $1$-parameter family  (up to scale) of \,$\sunitary{2} \times \sunitary{2} \times \unitary{1}$--invariant complete \gtmetric s on $S^3 \times \R^4$ with ALC geometry
where $\unitary{1}$ acts via the standard Hopf action on $S^3$ and trivially on $\R^4$.
\end{enumerate}
\end{mtheorem}
The family constructed in \ref{exist:B7} is the $\mathbb{B}_7$ family whose existence was first proven in 2013 by Bogoyavlenskaya \cite{Bogoyavlenskaya}, 
following earlier work by Brandhuber--Gomis--Gubser--Gukov \cite{BGGG} and Cveti\v{c}--Gibbons--L\"u--Pope \cite{CGLP:Spin(7)&G2}.
The family constructed in \ref{exist:D7} is the $\mathbb{D}_7$ family.
Numerical evidence for its existence was given by Cveti\v{c}--Gibbons--L\"u--Pope \cite{CGLP:M:Conifolds} and Brandhuber \cite{Brandhuber}*{\S 3.2}
We proved the existence of sufficiently collapsed members of this family in \cite{FHN:ALC:G2:from:AC:CY3}.
The advantage of the cohomogeneity one methods of this paper is that they allow us to construct the full $1$-parameter space of solutions, not only those that are sufficiently collapsed or close to the CS ALC solution.

\subsubsection*{Highly collapsed compact CS \gtwo--spaces}
Singular exceptional holonomy spaces play a crucial role in physics: M theory compactified on a smooth \gtmfd~ has
low energy behaviour that is too simple to model known features of the Standard Model of Particle Physics, but M theory on singular \gtwo--spaces with a combination of 
codimension 4 orbifold singularities and codimension 7 singularities can correct this problem
\cites{Acharya:SYM,Acharya:Witten,Witten:anomaly,Berglund:Brandhuber,Acharya:Gukov}.
Understanding compact singular spaces with \gthol is therefore an important, but currently open, problem.
The existence of our CS ALC \gtwo--space suggests a possible approach (that we are currently pursuing) to construct highly collapsed compact CS \gtwo--spaces 
that are close to a compact CS Calabi--Yau $3$-fold.

\subsection*{Finite quotients and infinitely many new AC \texorpdfstring{\gtmfd s}{G2-manifolds}} We have just seen the importance of AC \gtmetric s for constructing smooth \gtmfd s by desingularising CS $\gtwo$--spaces. AC \gtmfd s have also been studied from a physics perspective, \eg
the local physics associated with the three Bryant--Salamon AC \gtmfd s was studied in detail by Atiyah--Witten \cite{Atiyah:Witten}. However, 
these three classical examples remained the only known AC \mbox{\gtmetric s}.
As we now explain, our previous work on highly collapsed \gtmetric s also suggests the existence of infinitely many new cohomogeneity one AC \gtmetric s.

Earlier, we explained that specialising our analytic construction of highly collapsed \gtmetric s 
to the case where the collapsed AC Calabi--Yau limit is also of cohomogeneity one,
yielded the existence of infinitely many simply connected cohomogeneity one \gtmfd s.
In each of these cases (up to scale) there is a $1$-parameter family of ALC \gtmetric s
close to the highly collapsed limit; it is then natural to try to understand this $1$-parameter family 
away from this highly collapsed regime.
Motivated by the geometry of the $\mathbb{B}_7$ family of \gtmetric s,
it is natural to conjecture that, as we deform 
\emph{any} of these infinitely many $1$-parameter families away from the collapsed limit, 
then eventually an isolated conical singularity develops while maintaining ALC asymptotics at infinity. 

CS ALC \gtwo--spaces consistent with this conjecture exist: they
are quotients of the CS ALC space constructed in Theorem \ref{mthm:CS}
by particular finite cyclic subgroups of its group of \gtwo--isometries.
However, in order for such a CS ALC \gtwo --space to be a limit of smooth ALC \gtmetric s we must have previously observed an AC \mbox{\gtmetric}~with prescribed topology and asymptotic geometry bubbling off.
Thus continuing these $1$-parameter families of highly collapsed cohomogeneity one \mbox{\gtmetric s} 
far from the collapsed regime predicts the existence of infinitely 
many new cohomogeneity one AC \gtmetric s, asymptotic to particular finite quotients 
of the standard \gtwo--cone over $S^3 \times S^3$. 

\subsubsection*{Infinitely many new cohomogeneity one AC \gtmetric s}
For each $m, n \in \Z$, let $M_{m,n}$ denote the total space of the circle bundle over $K_{\CP^1 \times \CP^1}$ 
whose restriction to the zero section 
has first Chern class  $c_1=m [\omega] - n [\omega']$, where $\omega$ and $\omega'$ denote the standard K\"ahler--Einstein metrics on the two factors.
We describe $M_{m,n}$ as a cohomogeneity one space as follows. 
Denote by  $K_{m,n}\subset T^2 \subset \sunitary{2} \times \sunitary{2}$ the kernel of the homomorphism $\rho_{m,n}: T^2 \to \unitary{1}$ 
defined by $(e^{i\theta_1},e^{i\theta_2}) \mapsto e^{i(m \theta_1 + n\theta_2)}$.
$K_{m,n}$ is isomorphic to $\unitary{1} \times \mathbb{Z}_d$ where $d=\gcd(m,n)$.
Provided $m$ and $n$ are coprime then $M_{m,n}$ is a simply connected cohomogeneity one $7$-manifold
with $G=\sunitary{2} \times \sunitary{2}$: it has singular orbit type $G/K_{m,n} \simeq S^2 \times S^3$  and principal orbit type $G/\Gamma_{m+n}$
where $\Gamma_{m+n}:= K_{m,n} \cap K_{2,-2}$ is isomorphic to a cyclic group of
order $2\abs{m+n}$ whose generator $\zeta$ is embedded in
$T^2 \subset \sunitary{2} \times \sunitary{2}$ 
via $\zeta \mapsto (\zeta^n,\zeta^{-m})$ (up to the action of the outer automorphism the image in fact depends
only on $m+n$). In fact, each $M_{m,n}$ also admits a  cohomogeneity one action of $\sunitary{2} \times \sunitary{2} \times \unitary{1}$.

\begin{mtheorem} 
\label{mthm:AC}Suppose that $m$ and $n$ are coprime positive integers.
\begin{enumerate}[leftmargin=0.75cm]
\item There exists a $1$-parameter family $\varphi_\alpha$, $\alpha>0$, of 
smooth $\sunitary{2} \times \sunitary{2} \times \unitary{1}$--invariant  torsion-free \gtstr s
defined in a tubular neighbourhood of the singular orbit in $M_{m,n}$.
\item
There exists a unique $\alpha_\tu{ac}>0$ such that 
$\varphi_{\alpha_\tu{ac}}$ extends to a complete torsion-free AC \gtstr\ on $M_{m,n}$ asymptotic 
to the  $\Gamma_{m+n} $--quotient of the homogeneous nearly K\"ahler structure on $S^3 \times S^3$. 
\end{enumerate}
\end{mtheorem}
These new AC \gtmetric s are rigid up to scaling. This is consistent with our conjecture that they arise as limits of $1$-parameter families of ALC metrics: otherwise the desingularisation construction would yield a larger parameter family of smooth ALC \gtmetric s close to the CS ALC limit.

Note also that the tangent cone at infinity of our AC \gtmetric s depends only on the sum $m+n$, not on $m$ and $n$ separately, 
whereas we prove (see Remark \ref{rmk:Transitions:AC}) that metrics with different coprime pairs $(m,n)$ satisfying $0<m\le n$ are not isometric. 
Hence by considering all such pairs with fixed sum we obtain finitely many different AC \gtmetric s asymptotic to the same 
\gtwo--cone. This gives rise to infinitely many new geometric transitions in $\gtwo$ geometry.

\subsubsection*{Cohomogeneity one ALC metrics from AC ones}
The existence of these new AC \gtmetric s asymptotic to the cone over $(S^3 \times S^3)/\Gamma_{m+n}$
allows us to use the ALC version of the \gtwo--desingularisation technique to construct new  smooth ALC \gtmetric s 
close to CS ALC limits (and therefore far from the collapsed limit).
However, the cohomogeneity one methods we developed to prove 
the existence of the new  AC \gtmetric s also enable us to construct such smooth ALC \gtmetric s more directly.
Moreover, these methods allow us 
to construct the whole $1$-parameter family of \gtmetric s interpolating between the highly collapsed regime 
and the CS ALC/AC limit, whereas 
analytic methods can at present produce only solutions close to one of those two limiting regimes.
Under the assumptions of Theorem \ref{mthm:AC} and adopting its notation we have the following.
\begin{mtheorem}\hfill
\label{mthm:ALC}
\begin{enumerate}[leftmargin=0.75cm]
\item
If $0<\alpha < \alpha_\tu{ac}$ then the local solution $\varphi_\alpha$ constructed in Theorem \ref{mthm:AC} extends to a complete torsion-free ALC  \gtstr\ on $M_{m,n}$, 
asymptotic to a circle bundle over a $\Z_2$--quotient of the conifold.
\item
If $\alpha > \alpha_\tu{ac}$ then the local solution $\varphi_\alpha$ constructed in Theorem \ref{mthm:AC} cannot be extended to a complete 
invariant \gtmetric.
\end{enumerate}
\end{mtheorem}
\noindent
Theorems \ref{mthm:AC} and \ref{mthm:ALC} specialised to the case $m=n=1$ 
establish the existence of the whole $\mathbb{C}_7$ family of metrics previously conjectured to exist in \cite{CGLP:C7}.

\subsubsection*{Dihedral ALC \gtwo--spaces} 
When $m=n=1$, the desingularisation procedure also suggests the existence of another family of complete cohomogeneity one ALC \gtmetric s. Indeed we can quotient the CS ALC space of Theorem  \ref{mthm:CS} by a ``dihedral'' group $\Z_4$ that does not preserve the circle action, in which case the quotient inherits only a
cohomogeneity one action of $\sunitary{2}^2$. However, since the tangent cone at the conical singularity has enhanced symmetry $\sunitary{2}^3$, the local singularity model is in fact isomorphic to the cone over $S^3 \times S^3 /\Gamma_2$. Hence the new AC \gtmetric~ given by Theorem~\ref{mthm:AC} can be used to desingularise this singularity. The resulting smooth complete cohomogeneity one ALC \gtmfd s have only $\sunitary{2}^2$ symmetry and belong to the $\mathbb{A}_7$ family predicted by Hori--Hosomichi--Page--Rabad\'an--Walcher \cite{A7:ALC}: because of its smaller symmetry group at present we are unable to recover the $\mathbb{A}_7$ family by cohomogeneity one methods. We defer the analytic construction of the $\mathbb{A}_7$ family close to either the CS ALC  or the highly collapsed limits to elsewhere.

\subsubsection*{Classification of $\sunitary{2}^2\times\unitary{1}$--invariant complete \gtmetric s} 
While we therefore expect the existence of at least another family of complete cohomogeneity one $\sunitary{2}^2$--invariant \gtmfd s,  we show that our constructions recover all complete \gtmfd s with enhanced $\sunitary{2}^2\times\unitary{1}$ symmetry.
\begin{mtheorem}
\label{mthm:Classification}
Any complete simply connected $\sunitary{2}^2\times\unitary{1}$--invariant {\gtmfd} is isometric to one of the complete metrics of Theorems \ref{mthm:B7:D7}, \ref{mthm:AC} and \ref{mthm:ALC}.
\end{mtheorem}

\subsection*{Analogies with 4-dimensional \texorpdfstring{\hk~geometry}{hyperkahler geometry}}
There are various analogies between the geometry of \hk~$4$-manifolds and 
the geometry of \gtmfd s. 
Here we pursue the very strong analogy between the CS ALC \gtmetric~constructed in Theorem \ref{mthm:CS} 
and the Taub--NUT metric on $\C^2$. 
The reader is encouraged to skip to the plan of the paper if analogies with 
\hk~geometry seem unlikely to enlighten.

\subsubsection*{The Taub--NUT family} 
Recall that the Gibbons--Hawking ansatz \cite{Gibbons:Hawking} provides a local description of any $4$-dimensional 
\hk~metric admitting a triholomorphic circle action in terms of a positive harmonic function $h$ 
on some domain in $\R^3$. 
We will refer to the family of \hk~metrics obtained by 
choosing $h_m(x)=m+\tfrac{1}{2\abs{x}}$ for any choice of $m \in \R$, as the Taub--NUT family of metrics~$\taubnut$. 
By scaling we can reduce to the three cases: $m=0$, $m=+1$ or $m=-1$.
For $m=0$ we obtain the flat metric on $\R^4$. 
For $m=1$ we obtain the (Euclidean) Taub--NUT metric on $\R^4$:
a complete cohomogeneity one ALF space discovered by Hawking \cite{Hawking}. Many other 
ALF hyperK\"ahler $4$-manifolds can be derived from Taub--NUT as we describe below.
For $m=-1$, $h$ is no longer everywhere positive, and the resulting metric is incomplete. 
However, even this incomplete cohomogeneity one metric is not without interest. 
For $r$ sufficiently large, $h$ is strictly negative and one therefore obtains a (negative) definite hyperK\"ahler ALF end
from it. The asymptotic geometry of the Atiyah--Hitchin metric \cite{Atiyah:Hitchin}, 
another more complicated ALF hyperK\"ahler $4$-manifold that does not arise from the Gibbons--Hawking ansatz, nevertheless, turns out to be exponentially well approximated by (a finite quotient of) such a negative mass Taub--NUT metric \taubnut.

\subsubsection*{The conically singular family of \gtmetric s as the \gtwo~analogue of the Taub--NUT family}
It is natural to view the family of CS torsion-free \gtstr s \csgt~we constructed in Theorem \ref{mthm:CS}  as the $\gtwo$ analogue of the Taub--NUT family of metrics \taubnut.
Certain similarities are immediately apparent:
we have a $1$-parameter family of cohomogeneity one special holonomy metrics for 
which the singular orbit is a point;  up to scale there 
are only three different solution types $c=0$, $c>0$ or $c<0$; 
when $c=0$ we obtain a Ricci-flat cone $C$ with special holonomy 
which is rigid and that has enhanced symmetry 
compared to the solutions with $c \neq 0$; 
for $c>0$ the CS solution \csgt~ is forward complete and has an ALC end as $t \to \infty$; 
for $c<0$ the CS solution \csgt~is forward incomplete.

\subsubsection*{ALF \hk~metrics from quotients of Taub--NUT}
Further similarities between the CS family of \gtmetric s \csgt~and the Taub--NUT family of metrics \taubnut~
are connected with the infinitely many new families of cohomogeneity one \gtmetric s we constructed in Theorems \ref{mthm:AC} and \ref{mthm:ALC}.
First we need to recall how to obtain other ALF hyperK\"ahler 
$4$-manifolds by combining the  (positive mass) Taub--NUT metric with
ALE  hyperK\"ahler $4$-manifolds, following Biquard--Minerbe \cite{Biquard:Minerbe}.
 
If we identify $\R^4$ with the quaternions $\mathbb{H}$ then 
left multiplication by any unit quaternion and right multiplication by unit quaternions
in the normaliser $\unitary{1}\rtimes\Z_2$ of the maximal torus $\unitary{1}$ of $\sunitary{2}$
all give isometries of the Taub--NUT metric; 
the left $\sunitary{2}$ action rotates the $2$-sphere of compatible complex structures, while the
$\unitary{1}\rtimes\Z_2$ action acts triholomorphically.
On flat $\R^4$ the group of triholomorphic isometries fixing the origin enhances to $\sunitary{2}$.
For any ADE subgroup $\Gamma \subset \sunitary{2}$ the minimal resolution of $\C^2/\Gamma$ admits a family of complete ALE \hk~metrics asymptotic to the Euclidean 
orbifold metric on $\C^2/\Gamma$ \cite{Kronheimer:HK:construct}. 
In the simplest case, $\Gamma=\Z_2$
is the centre of $\sunitary{2}$ and the minimal resolution
of $\C^2/\Z_2$, biholomorphic to $T^*\CP^1$, admits a unique (up to scale) cohomogeneity one $\sunitary{2}$--invariant ALE \hk~metric, the Eguchi--Hanson metric \cite{Eguchi:Hanson}.
For any cyclic or dihedral subgroup $\Gamma \subset \sunitary{2}$, but \emph{not} any of the exceptional ones, since these are not contained in $\unitary{1}\rtimes\Z_2$, we may
consider the ALF \hk~orbifold,  Taub--NUT quotiented by $\Gamma$.
We may resolve its isolated orbifold singularity by gluing 
in an ALE \hk~metric on the minimal resolution of $\C^2/\Gamma$. 
Biquard--Minerbe \cite{Biquard:Minerbe} showed that one can glue in such ALE \hk~spaces 
without destroying the ALF geometry at infinity. 
In the simplest case $\Gamma = \Z_2$, one obtains
a $1$-parameter family of cohomogeneity one ALF \hk~metrics on $T^*\CP^1$.
In fact, whenever $\Gamma$ is cyclic,
both the ALE metric on the minimal resolution and the Taub--NUT metric admit triholomorphic 
circle actions, and the result of gluing arises more explicitly from the Gibbons--Hawking ansatz.
Gluing in a dihedral ALE space will however destroy the triholomorphic circle symmetry 
of Taub--NUT and the main point of \cite{Biquard:Minerbe} was to construct dihedral ALF spaces.

In the analogy with \hk~geometry, quotients of the CS ALC \mbox{\gtmetric} constructed 
in Theorem \ref{mthm:CS} by finite subgroups of its group of {\gtstr}--preserving isometries acting freely on $S^3 \times S^3$,
are the \gtwo--analogues of orbifold quotients of Taub--NUT. 
The rigid cohomogeneity one AC \gtmetric s
constructed in Theorem \ref{mthm:AC}
can all be thought of as like the Eguchi--Hanson metric, which is also cohomogeneity one and rigid (up to scale). 
Finally the ALC version of \gtwo--desingularisation, applied to finite quotients of the CS ALC \gtmetric~of Theorem 
\ref{mthm:CS}, and desingularised by appropriate AC \gtmetric s, \eg using the new AC \gtmetric s we construct, is the \gtwo--analogue of the Biquard--Minerbe construction. 

\subsection*{Plan of the paper}
In the rest of the Introduction we give the plan of the paper, which also serves as a detailed outline of the proof of our four main theorems.
 
In Section \ref{sec:ALC:from:AC:CY} we describe an infinite family of simply connected noncompact $7$-manifolds admitting cohomogeneity one actions of
$\sunitary{2} \times \sunitary{2}\times\unitary{1}$
which, except in one case,
arise as the total space of a non-trivial circle bundle over a cohomogeneity one AC Calabi--Yau $3$-fold $B$. 
There are three main examples of the latter: the small resolutions and the smoothing of the conifold and the canonical bundle of $\CP^1 \times \CP^1$. 
The small resolutions and the smoothing of the conifold give rise to a single simply connected cohomogeneity 
one $7$-manifold each.  Both are diffeomorphic to $S^3 \times \R^4$, with principal orbit $S^3 \times S^3$, 
but they have inequivalent singular orbit types and are therefore not equivariantly diffeomorphic. 
$K_{\CP^1 \times \CP^1}$, because it has $2$-dimensional second cohomology, provides by far the richest source of simply connected cohomogeneity one $7$-manifolds:
there are infinitely many such circle bundles $M_{m,n}$ over $K_{\CP^1 \times \CP^1}$, parameterised as described above by a pair of coprime integers $m$ and $n$.
Provided $m$ and $n$ have the same sign, the circle bundles $M_{m,n}$ all satisfy the hypotheses in our recent 
construction of complete highly collapsed  \gtmfd s with ALC geometry from circle bundles over an AC Calabi--Yau $3$-fold $B$ \cite{FHN:ALC:G2:from:AC:CY3}. 
Hence we know that each $M_{m,n}$ admits a $1$-parameter family of highly collapsed ALC
\gtmetric s, which
must be
invariant under the action of $\sunitary{2}\times\sunitary{2}\times\unitary{1}$. The principal orbits in $M_{m,n}$ are diffeomorphic to quotients of $S^3 \times S^3$ by a cyclic group $\Gamma_{m+n}$ of order $2\abs{m+n}$.

The goal of the rest of the paper is to use cohomogeneity one techniques to understand 
the whole $1$-parameter family of complete cohomogeneity one $\sunitary{2}\times\sunitary{2}\times\unitary{1}$--invariant \gtmetric s---not only the highly collapsed regime accessible via our 
analytic methods---for each of the $7$-manifolds described above.

In Section \ref{sec:half:flat} we describe the system of nonlinear ODEs that governs 
the local behaviour of $\sunitary{2} \times \sunitary{2}\times\unitary{1}$--invariant torsion-free \gtstr s $\varphi$ on $I \times (S^3 \times S^3)/\Gamma$,  where 
$I$ is some interval and $\Gamma$ is either trivial or one of the finite cyclic subgroups $\Gamma_k \simeq \Z_{2k}$. 
The case where $\Gamma$ is non-trivial is essential to this paper. 
Both Hamiltonian and Lagrangian formulations of the ODE system 
arise from Hitchin's work on stable forms \cite{Hitchin:Stable:forms}:
our paper makes use of both. Two real parameters $p$ and $q$ determine the cohomology class of the fundamental $3$-form $\varphi$. For each choice of parameters $p$ and $q$, local $\sunitary{2}\times\sunitary{2}\times\unitary{1}$--invariant \gtmetric s on $I \times (S^3 \times S^3)$
depend on $2$ parameters. The main task is to understand which of these local solutions extend to complete metrics, 
or in the case of the CS ALC solution of Theorem \ref{mthm:CS}, to a metric that is forward complete with ALC geometry as $t \to \infty$ and has prescribed singular behaviour as $t \to 0$.

In Sections 4 and 5 
we understand solutions that extend smoothly over a singular orbit and solutions that have either 
conically singular or asymptotically conical
end behaviour. Proposition \ref{prop:Solutions:Singular:Orbit}, the main result of Section \ref{sec:sing:extension}, 
characterises which members of the $2$-parameter family of local $\sunitary{2} \times \sunitary{2} \times \unitary{1}$-invariant 
torsion-free \gtstr s extend smoothly over the different classes of singular orbit
that we need to consider. In each case there are strong restrictions on the values of the constants $p$ and $q$ compatible with 
smooth extension over a given singular orbit type
and in all cases  there is (up to scale) a $1$-parameter family of smooth solutions defined in a neighbourhood of each type of singular orbit. 
To prove these results we adapt to our first-order ODE systems, the representation-theoretic 
approach to singular initial value problems for cohomogeneity one Einstein metrics developed by  Eschenburg--Wang \cite{Eschenburg:Wang}.
The first two authors used the same framework in the course of proving the existence of complete cohomogeneity one nearly K\"ahler $6$-manifolds \cite{Foscolo:Haskins}*{\S 4}.
The main result in Section \ref{sec:cs:ac:ends} is Proposition \ref{prop:CS:AC:ends}, which proves the existence of 1-parameter families of $\sunitary{2} \times \sunitary{2} \times \unitary{1}$--invariant torsion-free \gtstr s on either a CS or AC end:
in the CS case necessarily $p = q = 0$, while in the AC case we get a 1-parameter family
of AC ends $\varphi_{\tu{ac}}^c$ for each fixed $p, q$. Describing solutions with either type of end behaviour 
leads to a class of singular initial value problems not widely studied in the previous extensive
work on cohomogeneity one Einstein metrics. 
We state a general existence result, Theorem \ref{thm:Singular:IVP:Extended}, which yields convergent \emph{generalised} power series solutions to a wide class of first-order singular initial value problems,
including ours,
and for which solutions depend real analytically on a finite number of real parameters.

The most novel arguments in the paper appear beginning  in Section \ref{sec:ALC} where 
we develop criteria that guarantee that a locally-defined cohomogeneity one 
torsion-free \gtstr~extends to a forward-complete solution and then establish additional conditions under which 
we can bootstrap from forward-completeness to finer asymptotic metric behaviour---in our case ALC asymptotics. 

Proposition \ref{prop:Mean:Curvature:Blow:Up} gives a necessary and sufficient 
condition for forward completeness in terms of the positivity of the mean curvature of all principal orbits. 
The fact that (non-flat) principal orbits cannot be minimal in \emph{complete} cohomogeneity one Ricci-flat manifolds 
was observed previously by B\"ohm \cite{Bohm}. The sufficiency of the positivity of mean curvature of all principal orbits for forward completeness, however, uses the fact that we have a \emph{first-order} ODE system for the metric $g$. It is a pleasure to thank Burkhard Wilking for suggesting this idea.
Close to a singular orbit, or to an isolated conical singularity, 
the mean curvature of any principal orbit is necessarily large and positive. The next step is therefore to determine conditions under which
the positivity of the mean curvature of the principal orbits persists for all time. 

Most of the results of Sections 2 to 5 are stated more generally for $\sunitary{2}\times\sunitary{2}$--invariant torsion-free \gtstr s that do not necessarily enjoy an enhanced symmetry group $\sunitary{2}\times\sunitary{2}\times\unitary{1}$. At this stage, to make further progress, it is crucial that we
restrict to the case where there is an additional 
\unitary{1} symmetry. 
Once the values of $p$ and $q$ have been fixed, two coefficient functions $a$ and $b$ 
determine any closed $\sunitary{2} \times \sunitary{2} \times \unitary{1}$--invariant \gtstr~$\varphi$. 
The rest of our analysis of the global behaviour of solutions is based on the Lagrangian formulation
of the problem.  Under the additional symmetry assumption this leads to the single 
second-order nonlinear ODE, equation \eqref{eq:Fundamental:ODE:Brandhuber:U(1)}, 
which we write schematically as $G_{p,q}(a,b,a',b',a'',b'')=0$.
In Lemma \ref{lem:ALC:chamber} we prove that solutions 
to equation \eqref{eq:Fundamental:ODE:Brandhuber:U(1)} that begin in a certain open subset $\mathcal{O}_\tu{fc}$ 
of the set of principal orbits must remain in $\mathcal{O}_\tu{fc}$, at least while they continue to exist
with the fundamental $3$-form $\varphi$ remaining well defined.
Moreover, it turns out that the mean curvature of any principal orbit  in $\mathcal{O}_\tu{fc}$ 
is strictly positive. Therefore, by Proposition \ref{prop:Mean:Curvature:Blow:Up}, any solution 
that ever enters $\mathcal{O}_\tu{fc}$  cannot blow up in finite time and is therefore forward complete.

Significantly more work is needed to pass from forward completeness to a statement about 
the complete end necessarily having ALC geometry. The first step is to clarify what asymptotic behaviour 
the coefficients $a$ and $b$ should exhibit along an ALC end. 
On the total space of the circle bundle $\R^+ \times \sunitary{2} \times \sunitary{2} \to \R^+ \times (\sunitary{2}\times\sunitary{2})/\Delta \unitary{1}=C$ over the conifold $C$
there is an explicit $1$-parameter family of \emph{closed} $\sunitary{2} \times \sunitary{2} \times \unitary{1}$--invariant 
\gtstr s $\varphi_{\infty}^{\ell}$, where the parameter $\ell>0$ is the asymptotic length of the circle fibre. 
These closed \gtstr s $\varphi_{\infty}^{\ell}$ serve  as models for the possible asymptotic behaviour 
of ALC ends of invariant \gtmetric s. Convergence of $\varphi$ to the asymptotic model $\varphi_{\infty}^{\ell}$
implies that,
in appropriate parametrisations,
$a$ and $b$ behave asymptotically like $s^3$ and $s^2$
respectively.

For a solution to \eqref{eq:Fundamental:ODE:Brandhuber:U(1)} that enters the open set 
$\mathcal{O}_\tu{fc}$ above,
forward completeness already forces both $a$ and $b$ to go to infinity.
However, without further assumptions 
it does not force the asymptotic behaviour needed for an ALC end. To guarantee this
we need to restrict to a smaller open subset $\mathcal{O}_\tu{alc} \subset \mathcal{O}_\tu{fc}$ of the space of principal orbits. 
Proposition \ref{prop:ALC:growth:a:b} establishes that if a solution of \eqref{eq:Fundamental:ODE:Brandhuber:U(1)} enters $\mathcal{O}_\tu{alc}$  then it stays 
in $\mathcal{O}_\tu{alc}$ for all future times and has an ALC end modelled on $\varphi_{\infty}^{\ell}$ for some $\ell>0$.
Proposition \ref{prop:ALC:growth:a:b} is our main tool for establishing the existence of forward-complete
invariant torsion-free \gtstr s with an ALC end. On the other hand, Proposition \ref{prop:death:quadrant} establishes the existence of an open subset $\mathcal{O}_\tu{in}$ of the space of principal orbits with the property that any solution of \eqref{eq:Fundamental:ODE:Brandhuber:U(1)} that enters $\mathcal{O}_\tu{in}$ is forward \emph{incomplete}.

We can now easily prove 
the existence of the conically singular ALC \gtmetric~described in Theorem \ref{mthm:CS} 
and also the two $1$-parameter families $\mathbb{B}_7$ and $\mathbb{D}_7$ of ALC \gtmetric s 
described in Theorem \ref{mthm:B7:D7}, 
by combining Proposition \ref{prop:ALC:growth:a:b} with information about the local solutions extending smoothly on singular orbits
constructed in Section \ref{sec:sing:extension}, or the local solutions with CS ends constructed in Section \ref{sec:cs:ac:ends}.
The point is that for all these local solutions we have expansions for the coefficients $a$ and $b$ of the $3$-form $\varphi$ in a neighbourhood of
the singular orbit. Using these we can identify for exactly which parameter values local solutions
enter $\mathcal{O}_\tu{alc}$ or $\mathcal{O}_\tu{in}$. For instance, the local CS solutions $\varphi^c_{\tu{cs}}$ enter into $\mathcal{O}_\tu{alc}$ (respectively, $\mathcal{O}_\tu{in}$)
precisely when $c>0$ ($c<0$).

One also expects a similar behaviour for the infinitely many $1$-parameter  (up to scale) families of local solutions defined in a neighbourhood of the singular orbit in $M_{m,n}$: there are two phases, one corresponding to solutions that enter $\mathcal{O}_\tu{alc}$ and one corresponding to solutions that enter $\mathcal{O}_\tu{in}$; the two phases are separated by a unique (up to scale) solution corresponding to a smooth AC metric on~$M_{m,n}$. However, unlike the previous cases, it is impossible to determine which phase each solution belongs to simply by looking at its behaviour near the singular orbit. The key to proving Theorems \ref{mthm:AC} and \ref{mthm:ALC} turns out to be first to focus attention on the single member of each $1$-parameter family that 
has AC geometry. Once we prove the existence of the unique AC member of each $1$-parameter family, we will be able to compare the behaviour of our other local solutions with it.  We will use this to 
conclude that, for all parameter values on one side of the AC solution, local solutions are eventually forced to enter $\mathcal{O}_\tu{alc}$, while on the other side of the AC solution,  local solutions that exist for a sufficiently long time must eventually enter $\mathcal{O}_\tu{in}$; see Sections \ref{sec:ALC:mn} and \ref{sec:ALC:mn:in}.

It remains therefore to describe the approach we take in Section \ref{sec:AC:back} 
to the construction of a new rigid complete AC \gtmetric~on each of the infinitely many spaces $M_{m,n}$. To illustrate the somewhat delicate nature of the problem consider the following facts. From Proposition \ref{prop:Solutions:Singular:Orbit} 
(iii) and (iv), for each choice of $m$ and $n$ there is (up to scale) a $1$-parameter family 
of local $\sunitary{2} \times \sunitary{2} \times \unitary{1}$--invariant solutions that extend smoothly 
over the singular orbit $\sunitary{2} \times \sunitary{2}/K_{m,n}$. 
Moreover, the constants $p$ and $q$ that specify the cohomology class of the 
\gtstr~$\varphi$ are determined (up to scale) by the choice of $m$ and $n$. 
Also from Proposition \ref{prop:CS:AC:ends} (ii), for each choice of $p$ and $q$, there 
exists a $1$-parameter family of $\sunitary{2} \times \sunitary{2} \times \unitary{1}$--invariant 
solutions with AC ends. Finally, recall also that the analysis of the ODEs on the principal orbits showed
that, for each fixed $p$ and $q$, there is a $2$-parameter family of local 
$\sunitary{2} \times \sunitary{2} \times \unitary{1}$--invariant solutions. 
Putting this together, a complete AC {\gtmetric} corresponds to an intersection point of two curves in a $2$-dimensional manifold. For 
dimensional reasons we might hope that such matching happens a (nonzero!) finite number of times.

We consider the problem of ``shooting from infinity'', \ie we consider the $\Gamma_{m+n}$ quotient of
the $1$-parameter family $\varphi_{\tu{ac}}^c$, $c\in\R$, of invariant AC end solutions $\varphi_{\tu{ac}}^c$ we constructed on $(T,\infty) \times S^3 \times S^3$ 
and investigate 
extending these solutions backwards. 
Our aim is to show that as we vary the parameter $c \in \R$ then, for exactly one value of $c$, 
the maximal backward extension of the AC end satisfies the conditions to close smoothly on 
the singular orbit $\sunitary{2} \times \sunitary{2}/K_{m,n}$ at its backwards extinction time. Proposition \ref{prop:evolve_back} provides our main 
criterion for backward extension of invariant AC ends:
if a principal orbit belongs to a certain open subset $\mathcal{O}^+_{\tu{ac}}$ of the set of principal orbits
then the solution can be continued backwards in time remaining in $\mathcal{O}^+_{\tu{ac}}$,  
provided that the fundamental $3$-form $\varphi$ remains well defined. The invariant AC ends $\varphi_{\tu{ac}}^c$ constructed in Proposition \ref{prop:CS:AC:ends} (ii) belong to $\mathcal{O}^+_{\tu{ac}}$ precisely when the parameter $c \in \R$ is positive.
The fundamental $3$-form $\varphi$ fails to be well defined whenever the curve $(a,b)$ hits the zero-locus of an explicit quartic polynomial $F$ 
(depending on $m$ and $n$). The level set $F=0$ contains a unique singular point. In Proposition \ref{prop:uniquehit} we show that there exists a unique parameter value $c_\tu{ac}>0$ such that the solution $\varphi^c_\tu{ac}$ extends backward until it hits this distinguished point. Finally, in Proposition \ref{prop:even} we show that this unique solution in fact defines a smooth torsion-free \gtstr~on $M_{m,n}$ which by construction is AC.

The classification result Theorem \ref{mthm:Classification} is also proved in Section \ref{sec:AC} by showing that, up to a finite quotient, any $\sunitary{2}^2\times\unitary{1}$--invariant {\gtmfd} closing smoothly on a singular orbit must be equivariantly diffeomorphic to one of the manifolds we consider in the paper, \cf Theorem \ref{thm:Classification:U(1)}.

\subsection*{Acknowledgments}
LF would like to thank the Royal Society for the support of his research under a University Research Fellowship and NSF for the partial support of his work under grant DMS-1608143. MH and JN would like to thank the Simons Foundation for its support of their research 
under the Simons Collaboration on Special Holonomy in Geometry, Analysis and Physics (grants \#488620, Mark Haskins, and \#488631, Johannes Nordstr\"om).

All three authors would like to thank MSRI and the program organisers for providing an extremely productive environment
during the Differential Geometry Program at MSRI in Spring 2016. Research at MSRI was partially supported by NSF grant DMS-1440140.

LF and MH would like to thank the Riemann Centre for Geometry and Physics at Leibniz Universit\"at Hannover for supporting LF under a Riemann Research Fellowship and for hosting a research visit of MH in Autumn 2015. They would also like to thank the Newton Institute in Cambridge for hosting them during the Metric and Analytic Aspects of Moduli Spaces workshop in Summer 2015. Discussions about this work took place at both institutions. 

MH would like to thank Christoph B\"ohm and Burkhard Wilking for very useful discussions during his visit to M\"unster in November 2015, 
that led to Proposition \ref{prop:Mean:Curvature:Blow:Up}. MH would also like to thank Robert Bryant for several discussions about 
singular initial value problems for ODEs and PDEs, and for suggesting classical references that eventually led  us to consult Picard's book \cite{Picard}.

The authors would also like to thank Harvey Tollick for help with producing Figure \ref{fig:m1n2}.

\section{Cohomogeneity one ALC \texorpdfstring{\gtmfd s}{G2-manifolds}}
\label{sec:ALC:from:AC:CY}
In this section we describe the $7$-manifolds with a cohomogeneity one action of $\sunitary{2}\times\sunitary{2}$ known to admit complete \gtmetric s. We start by describing the $\mathbb{B}_7$ family \cites{BGGG,Bogoyavlenskaya} and then discuss our analytic construction of highly collapsed ALC \gtmetric s \cite{FHN:ALC:G2:from:AC:CY3}, specialised to the cohomogeneity one setting.

A cohomogeneity one manifold $X$ is a Riemannian manifold acted upon isometrically by a Lie group $G$ with generic orbits of codimension one. The cohomogeneity one manifolds we consider in this paper are complete noncompact irreducible Ricci-flat manifolds
and $G$ will be compact. The Splitting Theorem implies that they have only one end and therefore the orbit space $X/G$ is a half line $[0,\infty)$. 
The cohomogeneity one structure of $X$ is encoded by a pair of closed subgroups $K_0 \subset K \subset G$ 
called its \emph{group diagram}.
Orbits corresponding to points in $(0,\infty)$ are hypersurfaces all diffeomorphic to $G/K_0$ and are called \emph{principal orbits}. The orbit over $0$ is the only lower-dimensional orbit $G/K$ and it is called the \emph{singular orbit}. $X$ has the topology of a disc bundle over the singular orbit. The complement of the zero section of this disc bundle is foliated by the principal orbits $G/K_0$, which themselves are sphere bundles over $G/K$. In particular, $K/K_0$ must be diffeomorphic to a sphere. All the manifolds we consider in this paper admit a cohomogeneity one action of $G=\sunitary{2}\times\sunitary{2}$ or of $G=\sunitary{2}\times\sunitary{2} \times \unitary{1}$.

\subsection{The \texorpdfstring{$\mathbb{B}_7$}{B7} family and its symmetries}

The $\mathbb{B}_7$ family is a $1$-parameter family (up to scale) of \gtmetric s on $S^3\times\R^4$ admitting a cohomogeneity one action of $\sunitary{2}\times\sunitary{2}$, in fact of $\sunitary{2}\times\sunitary{2}\times \unitary{1}$. As explained in the Introduction, the generic member of the family has ALC asymptotic geometry, while at a special parameter value we have the Bryant--Salamon AC metric \cite{Bryant:Salamon}. One special explicit ALC member of the family was found by Brandhuber--Gomis--Gubser--Gukov \cite{BGGG}, while the existence of the full $1$-parameter family of ALC metrics was later established by Bogoyavlenskaya \cite{Bogoyavlenskaya}. We now describe carefully the cohomogeneity one structure of these metrics and their full group of symmetries.

We identify $S^3=\sunitary{2}$ with the unit quaternions. $S^3\times\R^4$ admits a cohomogeneity one action of $\sunitary{2}\times\sunitary{2}$ described by the group diagram
\begin{equation}\label{eq:group:diagram:B7}
\{ 1 \} \subset \triangle \sunitary{2}\subset \sunitary{2}\times \sunitary{2}.
\end{equation}
More explicitly, we can identify $[0,\infty)\times S^3\times S^3/\sim$ with $S^3\times\R^4$ via $(t,q_1,q_2)\mapsto (q_1\overline{q}_2,tq_1)$, where $\sunitary{2}\times\sunitary{2}$ acts by left multiplication on $S^3\times S^3$ and 
$(t,q_1,q_2)\sim (t',q'_1,q'_2)$ if and only if $t=0=t'$ and $q'_1=q_1q$, $q'_2=q_2 q\,$ for some $q\in\sunitary{2}$.

The full symmetry group of any \gtmetric~in the $\mathbb{B}_7$ family is however larger than $\sunitary{2}\times\sunitary{2}$:  we describe it first for the most symmetric member of the family, the AC Bryant--Salamon metric. In this case the group of continuous \gtwo--isometries   
is $\gbs=\sunitary{2}^2\times\sunitary{2}/\triangle\Z_2$, where the first factor $\sunitary{2}^2$ acts by left multiplication, the third $\sunitary{2}$ acts diagonally by right multiplication and the quotient by
$\Delta \Z_2$ appears because to get an effective action we must quotient by the center $\Z_2$ of $\sunitary{2}$ embedded diagonally in the three factors. 
The asymptotic cone $C$ of the AC Bryant--Salamon metric is the cone over the standard homogeneous nearly K\"ahler structure on $S^3 \times S^3$. 
The induced metric on $S^3 \times S^3$ has additional discrete isometries arising from outer automorphisms of $\sunitary{2}^3$. More concretely
there is an action of the symmetric group $\tu{S}_3$ on $S^3 \times S^3$ generated by the following pair of involutions
\[
(q_1, q_2)\mapsto (\overline{q}_1, q_2\overline{q}_1), \qquad (q_1, q_2)\mapsto (q_1\overline{q}_2, \overline{q}_2),
\]
whose composition has order $3$.
However this action of $\tu{S}_3$ does not extend to the AC Bryant--Salamon metric. 
Instead the Bryant--Salamon metric can be realised in three inequivalent ways as a cohomogeneity one manifold with an action of $\sunitary{2}\times\sunitary{2}$, corresponding to the group diagrams \eqref{eq:group:diagram:B7}, \eqref{eq:group:diagram:D7} and its image under the outer automorphism of $\sunitary{2}^2$. 

The symmetry group of the ALC members of the $\mathbb{B}_7$ family is smaller than $\gbs$. The standard Hopf action of $S^1$ on $\R^4$ induces a circle action on $S^3\times\R^4$. The action fixes $S^3\times\{ 0\}$, but the quotient $S^3\times\R^3$ is still a manifold. Under our identification of $S^3\times \R^4$ with $[0,\infty) \times S^3\times S^3/\sim$, the projection  
$\pi\co  [0,\infty) \times S^3\times S^3/\sim\, \longrightarrow S^3\times \R^3$ is given by 
\[(t,q_1,q_2)\mapsto ( q_1\overline{q}_2, t^2 q_1 i \overline{q}_1),
\]
where we identify $\R^3$ with $\Imag\HH$. The group of continuous isometries that preserves the generic \mbox{\gtstr} in the $\mathbb{B}_7$ family is the subgroup $\textup{G}_\pi$ of $\gbs$ that preserves this projection $\pi$. Note that since, away from $t=0$, $\pi$ is a principal circle bundle endowed with a natural connection $\theta$ (the left-invariant $1$-form on $\sunitary{2}^2$ dual to the vector field generating the diagonal $\unitary{1}$ subgroup), there is an exact sequence
\[
1\ra \textup{G}^+_\pi\ra \textup{G}_\pi\ra \Z_2\ra 1,
\]  
where $\textup{G}^+_\pi$ is the subgroup that preserves the connection $\theta$ (whereas arbitrary elements of $\textup{G}_\pi$ are allowed to send $\theta$ to $\pm\theta$). We have $\textup{G}^+_\pi=\sunitary{2}^2\times\unitary{1}/\triangle\Z_2$, where $\sunitary{2}^2$ is the left-acting group, $\unitary{1}$ acts diagonally on the right and $\triangle\Z_2$ is the center of $\sunitary{2}$ embedded diagonally in each factor, and $\textup{G}_\pi=\sunitary{2}^2\times\tu{N}/\triangle\Z_2$, where $\tu{N} \simeq \unitary{1} \rtimes \Z_2$ is the normaliser of $\unitary{1}$ in the right-acting $\sunitary{2}$. There is a further $\Z_2$ action generated by the outer automorphism of $\sunitary{2}^2$, 
\ie  the involution of $S^3\times S^3$ that exchanges the two factors. This automorphism is an isometry of all members of the $\mathbb{B}_7$ family but it always acts non-trivially on the \gtstr. 

By considering quotients of the $\mathbb{B}_7$ family by finite subgroups $\Gamma \subset \textup{G}_\pi$ that act freely on $S^3\times S^3$ we can obtain either non-simply-connected smooth ALC \gtmfd s or ALC \gtwo--orbifolds with singularities contained in a compact set. Finite subgroups of $\textup{G}_\pi$~acting freely on $S^3 \times S^3$ exist in abundance: 
since the subgroup $\sunitary{2}^2$ acts by left translations on $S^3 \times S^3$ any finite subgroup $\Gamma \subset \sunitary{2}^2 \subset \textup{G}_\pi$ 
necessarily acts freely. (More generally, a partial classification of subgroups 
of $\sunitary{2}^3$ acting freely on $S^3\times S^3$ was obtained recently in \cite{Cortes:NK}, though the complete classification remains open.) 
In general the resulting quotient space $(S^3 \times \R^4)/\Gamma$ will no longer admit a cohomogeneity one action. 
However, if one chooses $\Gamma$ to be contained in the subgroup $\Z_2^2\times \tu{N}/\triangle\Z_2$ of $\textup{G}_\pi$ that commutes with the left action of $\sunitary{2}^2$  then the quotient space does still admit a cohomogeneity one action of $\sunitary{2}\times\sunitary{2}$. Note that $\Z_2^2\times \tu{N}/\triangle\Z_2\simeq \Z_2\times\tu{N}$ and the action on the two factors of $S^3\times\R^4$ is given precisely by this isomorphism. There are three possibilities for $\Gamma$:
\begin{enumerate}[leftmargin=*]
\item
$\Gamma=\Gamma_0$ where $\Gamma_0$ is a cyclic or binary dihedral subgroup of $\sunitary{2}$;
\item
$\Gamma =\Z_2 \times  \Gamma_0$ with $\Gamma_0$ as above;
\item
$\Gamma$ is a cyclic group of even order or a binary dihedral group embedded in $\Z_2\times\tu{N}$ via the homomorphism $\phi\times\iota$, where $\phi\co \Gamma\ra \Z_2$ is a non-trivial homomorphism and $\iota$ is the standard embedding of $\Gamma$ in $\tu{N}\subset\sunitary{2}$. 
\end{enumerate}
The resulting orbifolds are, respectively, $S^3\times(\R^4/\Gamma_0)$,\, $\R\PP^3\times(\R^4/\Gamma_0)$ and $(S^3\times(\R^4/\Gamma_0))/\Z_2$, where in the latter case $\Gamma_0:=\phi^{-1}(1)$.
In this paper the case where $\Gamma$ is a cyclic group of even order embedded in $\Z_2\times\unitary{1}\subset\Z_2\times\tu{N}$ as in (iii) plays a distinguished role. 

We introduce the following notation that we use throughout the rest of the paper: for each $m,n \in \Z$ let $K_{m,n}$ be the subgroup of the maximal torus $T^2\subset \sunitary{2}\times\sunitary{2}$ defined by
\begin{equation}\label{eq:K:m:n}
K_{m,n} = \{ (e^{i\theta_1},e^{i\theta_2})\in T^2\, |\, e^{i(m\theta_1+n\theta_2)}=1 \},
\end{equation}
\ie $K_{m,n}$ is the kernel of the group homomorphism $(e^{i\theta_1},e^{i\theta_2})\mapsto e^{i(m\theta_1+n\theta_2)}$ from $T^2$ to $\unitary{1}$. In particular, $K_{1,-1}$ is the diagonally embedded circle $\triangle \unitary{1}$ in $\sunitary{2}\times\sunitary{2}$ and the subgroup $\Z_2\times\unitary{1}$ of elements in $G_\pi^+$ that commute with the left $\sunitary{2}^2$ action on $S^3\times S^3$ can be identified with the right-acting $K_{2,-2}$. Indeed, note that $K_{m,n}$ is isomorphic to the direct product $\unitary{1}\times \Z_{\gcd{ (m,n)}}$ via the map $(e^{i\theta},\zeta) \mapsto (e^{ik\theta}\zeta^{r}, e^{-ih\theta}\zeta^s)$, where $m=\gcd{(m,n)h}$, $n=\gcd{(m,n)}k$ and $rh+sk=1$. We also note that conjugation by $j\in \sunitary{2}$ (or any other unit quaternion orthogonal to $1$ and $i$) on the first (second) factor identifies $K_{m,n}$ with $K_{-m,n}$ ($K_{m,-n}$) since
$-je^{i\theta}j = e^{-i\theta}$.

Now, if $\Gamma_k\simeq\Z_{2k}$ is the subgroup of $K_{2,-2}=\Z_2\times\unitary{1}$ defined by the embedding $\zeta\mapsto (\zeta^k,\zeta)$, then the \emph{orbifold} $(S^3\times\R^4)/\Gamma_k$ admits a $1$-parameter family of cohomogeneity one 
\gtmetric s (the quotients of the $\mathbb{B}_7$ family of metrics) with group diagram
\[
\Gamma_k \subset \triangle\sunitary{2}\subset \sunitary{2}\times\sunitary{2}.
\] 
We can think of this orbifold as a \emph{partial} resolution of the cone over $(S^3\times S^3)/\Gamma_k$. In the next section we will see that there are \emph{smooth} $7$-manifolds that topologically resolve the same cone and that are known to carry complete cohomogeneity one ALC \gtmetric s.

\begin{remark}\label{rmk:Gamma:k}
When $k$ is odd, the isomorphism $\Z_{2k}\ra \Z_2\times\Z_{k}$ defined by $\zeta\mapsto (\zeta^k,\zeta^2)$ and the outer automorphism of $\sunitary{2}\times\sunitary{2}$ that exchanges the two factors yield an isomorphism between $\Gamma_k$ and the subgroup $\Z_2\times\Z_{2k}\subset\Z_2\times\unitary{1}$ of type (ii). 
\end{remark}
   
\subsection{Highly collapsed ALC \texorpdfstring{\gtmetric s}{G2-metrics}}

The starting point for this paper is our construction of complete \gtmfd s from AC Calabi--Yau $3$-folds \cite{FHN:ALC:G2:from:AC:CY3} and its specialisation to the cohomogeneity one case. Our analytic construction yields the existence of infinitely many new families of cohomogeneity one \gtmetric s, but recovers only metrics close to a certain degenerate (collapsed) limit. In this section we describe these cohomogeneity one manifolds. The main goal of the paper is to describe the full moduli space of invariant \gtmetric s on each of the $7$-manifolds in question.

\subsubsection*{Cohomogeneity one AC Calabi--Yau 3-folds}
The cross-sections of Calabi--Yau cones are called Sasaki--Einstein manifolds. The only homogeneous Sasaki--Einstein $5$-manifolds are the round $5$-sphere and $S^2\times S^3 = \sunitary{2}\times\sunitary{2}/\triangle\unitary{1}$, or finite quotients thereof \cite{Sparks:SE}*{\S 2.2}. The case of the round $5$-sphere and its quotients is not relevant for this paper and therefore we consider only AC Calabi--Yau $3$-folds $B$ asymptotic to (quotients of) the \emph{conifold}, \ie the Calabi--Yau cone $\tu{C}(\Sigma)$ over $\Sigma = S^2\times S^3$ endowed with its $\sunitary{2}\times\sunitary{2}$--invariant Sasaki--Einstein structure. If we insist that the Calabi--Yau $3$-fold $B$ has a cohomogeneity one action then there are three known possibilities. In the following we use the notation $K_{m,n}$ introduced in \eqref{eq:K:m:n}.

\begin{enumerate}[leftmargin=*]
\item $B$ is one of the two small resolutions of the conifold \cite{Candelas:delaOssa}. They have group diagrams
\[
K_{1,-1} \subset \unitary{1}\times \sunitary{2} \subset \sunitary{2}\times\sunitary{2}, \ \  K_{1,-1} \subset \sunitary{2}\times \unitary{1} \subset \sunitary{2}\times\sunitary{2}.
\]
The two resolutions are in fact isomorphic under the outer automorphism of $\sunitary{2}\times\sunitary{2}$ that exchanges the two factors. Without loss of generality we can therefore concentrate only on the first case, which from now on we will refer to as \emph{the} small resolution of the conifold.
\item $B$ is the canonical line bundle of $\CP^1\times \CP^1$. Its group diagram is
\[
K_{2,-2}\subset T^2 \subset \sunitary{2}\times\sunitary{2}.
\] 
In this case the tangent cone at infinity is a (free) $\Z_2$--quotient of the conifold. Moreover, while the AC Calabi--Yau metric on the small resolution of the conifold is unique (up to scale), in this case there is a $2$-parameter family of AC Calabi--Yau metrics on $B$, parametrised by the K\"ahler class of the corresponding K\"ahler form. The metrics corresponding to the unique (up to scale) compactly supported K\"ahler class  are due to Calabi \cite{Calabi:AC:CY}; the full $2$-parameter family of AC Calabi--Yau metrics was first considered in \cite{Zayas:et:al}.
\item $B=T^\ast S^3$ is the smoothing of the conifold \cites{Candelas:delaOssa,Stenzel}. Its group diagram is
\[
K_{1,-1}\subset  \triangle \sunitary{2} \subset \sunitary{2}\times\sunitary{2}.
\]
\end{enumerate}

\begin{remark*}
In (ii), it would perhaps be more natural to define $S^2\times S^3/\Z_2$ as $\sunitary{2}\times\sunitary{2}/K_{2,2}$. Our choice is motivated by the desire to use the same conventions adopted in the physics literature, in particular those of \cite{Brandhuber}. The main confusing consequence of our choice is that the complex structures on the two factors in $\CP^1\times \CP^1 = \sunitary{2}\times \sunitary{2}/T^2$ are conjugate to each other. In particular, with our choices K\"ahler classes on $\CP^1\times \CP^1$ are parametrised by pairs $(\alpha,-\beta)$ with $\alpha,\beta>0$.
\end{remark*}

\subsubsection*{Highly collapsed \texorpdfstring{\gtmetric s}{G2-metrics} on circle bundles over AC Calabi--Yau 3-folds}
Let $(B,g_0,\omega_0,\Omega_0)$ be one of the three simply connected cohomogeneity one AC Calabi--Yau $3$-folds just described. Here $\omega_0$ is the K\"ahler form of the AC Calabi--Yau metric $g_0$ and $\Omega_0$ is a holomorphic complex volume form normalised so that $\omega_0^3=\tfrac{3}{2}\Real\Omega_0\wedge\Imag\Omega_0$. We now consider a \emph{non-trivial} circle bundle $M$ over~$B$: therefore we consider only cases  \tu{(i)} and \tu{(ii)}, since the second cohomology vanishes in case \tu{(iii)}.
\begin{enumerate}[leftmargin=*]
\item When $B$ is the small resolution of the conifold, up to finite quotients and a change of orientation, there is only one possible choice of line bundle. The total space $M$ is $S^3\times\R^4$, endowed with an $\sunitary{2}\times\sunitary{2}$ action with group diagram
\begin{equation}\label{eq:group:diagram:D7}
\{ 1 \} \subset \{ 1\} \times \sunitary{2} \subset \sunitary{2}\times \sunitary{2}.
\end{equation}
\item When $B=K_{\CP^1\times \CP^1}$ we have an infinite family of examples: for a pair $(m,n)$ of coprime integers we consider the simply connected manifold $M_{m,n}$ described by the group diagram
\begin{equation}\label{eq:group:diagram:C7(m,n)}
K_{m,n}\cap K_{2,-2} \subset K_{m,n}\subset \sunitary{2}\times \sunitary{2}.
\end{equation}
The subgroup $K_{m,n}\cap K_{2,-2}$ is isomorphic to $\Z_{2|n+m|}$ embedded in $T^2\subset\sunitary{2}\times\sunitary{2}$ via $\zeta \mapsto (\zeta^n,\zeta^{-m})=(\zeta^n,\zeta^{n+m}\zeta^n)$. If $n$ is odd then $\gcd{(n,2|n+m|)}=1$ since $n$ and $m$ are coprime and therefore $\zeta\mapsto\zeta^n$ is an isomorphism: $K_{m,n}\cap K_{2,-2}$ is then isomorphic to the group $\Gamma_{|n+m|}$ defined in the previous section. If $n=2n'$ is even, then $\gcd (n',|n+m|)=1$ and $K_{m,n}\cap K_{2,-2}$ is isomorphic to the subgroup $\{ (\xi,\epsilon\,\xi)\in T^2\, |\, \xi\in\Z_{|n+m|}, \epsilon\in\Z_2\}$ via the composition of the isomorphism $\Z_{2|n+m|}\simeq \Z_{|n+m|}\times\Z_2$ given by $\zeta\mapsto (\zeta^2,\zeta^{|n+m|})$ and the automorphism of $\Z_{|n+m|}$ defined by $\xi\mapsto \xi^{n'}$. As explained in Remark \ref{rmk:Gamma:k} the outer automorphism that exchanges the two factors of $\sunitary{2}\times\sunitary{2}$ identifies this group with $\Gamma_{|n+m|}$ (in fact, note that if $n$ is even then $m$ must be odd).
\end{enumerate}
In all cases there exists a unique $\sunitary{2}\times\sunitary{2}$--invariant connection $\theta$ on the circle bundle $M\ra B$ (the left-invariant $1$-form dual to the vector field that generates the diagonal $\unitary{1}$ subgroup) and therefore a $1$-parameter family of $\sunitary{2}^2\times\unitary{1}$--invariant metrics
\[
\overline{g}_\epsilon = g_0 + \epsilon^2 \theta^2.
\]
Note that the model metric $\overline{g}_\epsilon$ has ALC asymptotic geometry. Furthermore, for $\epsilon>0$ sufficiently small, $\overline{g}_\epsilon$ is approximately $\gtwo$. In \cite{FHN:ALC:G2:from:AC:CY3} we use analytic methods to show that, under the necessary topological condition $c_1(M)\cup [\omega_0]=0\in H^4(B)$, for all $\epsilon>0$ sufficiently small, we can perturb $\overline{g}_\epsilon$ to an $\sunitary{2}^2\times\unitary{1}$--invariant ALC \gtmetric~$g_\epsilon$. The metric $g_\epsilon$ is \emph{highly collapsed}: $g_\epsilon$ is arbitrarily close (in certain weighted H\"older spaces) to the model metric $\overline{g}_\epsilon$ and therefore $(M,g_\epsilon)$ collapses to $(B,g_0)$ as $\epsilon\ra 0$ (with globally bounded curvature). The specialisation of the main result of \cite{FHN:ALC:G2:from:AC:CY3} to the cohomogeneity one setting therefore yields the following theorem. 

\begin{theorem}\hfill
\begin{enumerate}[leftmargin=0.75cm]
\item The manifold $M=S^3\times\R^4$ described by the group diagram \eqref{eq:group:diagram:D7} carries a $1$-parameter family (up to scale) of highly collapsed $\sunitary{2}^2\times\unitary{1}$--invariant ALC \gtmetric s.
\item For every pair of coprime positive integers $m,n$ the manifold $M_{m,n}$ described by the group diagram \eqref{eq:group:diagram:C7(m,n)} carries a $1$-parameter family (up to scale) of highly collapsed $\sunitary{2}^2\times\unitary{1}$--invariant ALC \gtmetric s.
\end{enumerate}
\end{theorem}
\begin{proof}
Given the main result of \cite{FHN:ALC:G2:from:AC:CY3}, we need only explain why the necessary topological constraint \mbox{$c_1(M)\cup [\omega_0]=0\in H^4(B)$} is satisfied. 
The uniqueness (modulo diffeomorphism) of metrics in the construction of~\cite{FHN:ALC:G2:from:AC:CY3} implies that the continuous isometries $\sunitary{2} \times \sunitary{2}$  of the Calabi--Yau base $B$ extend to $M$, 
and rotation in the circle fibres provides the additional $\unitary{1}$. 
In case (i) the constraint is automatically satisfied since $H^4(B)=0$. In case (ii) note that $B=K_{\C\PP^1\times\C\PP^1}$ retracts onto its singular orbit $D=\CP^1\times \CP^1$. Hence the $7$-manifold $M_{m,n}$ retracts onto its singular orbit $M_{m,n}|_D$, a principal circle bundle over~$D$. We can then understand the constraint $c_1(M_{m,n})\cup [\omega_0]=0$ by restriction to this singular orbit.

Fix a basis of left-invariant $1$-forms $e_1,e_2, e_3,e'_1,e'_2,e'_3$ (the dual vector fields will be denoted by $E_i, E'_i$) on $\sunitary{2}\times \sunitary{2}$ with the property that
\begin{equation}\label{eq:Maurer:Cartan}
de_i=-e_j\wedge e_k, \qquad de'_i=-e'_j\wedge e'_k
\end{equation}
for $(ijk)$ any cyclic permutation of $(123)$. When identifying $\Lie{su}_2$ with $\Imag{\HH}$ we adopt the convention that $E_1=\frac{i}{2}$, $E_2=\frac{j}{2}$, $E_3=\frac{k}{2}$: indeed by the first Maurer--Cartan structure equation, we must have $[E_i,E_j]=E_k$ in order for \eqref{eq:Maurer:Cartan} to be satisfied. In particular, note that $E_i, E'_i$ have period $4\pi$. Without loss of generality we can assume that the maximal torus $T^2$ in $\sunitary{2}\times\sunitary{2}$ is generated by $E_3$ and $E'_3$. The area forms of the two factors of $\CP^1\times \CP^1 = \sunitary{2}\times\sunitary{2}/T^2$ are
\[
\omega _1 = -\tfrac{1}{2}de_3 = \tfrac{1}{2}e_1\wedge e_2, \qquad \omega_2 = -\tfrac{1}{2}de'_3=\tfrac{1}{2}e'_1\wedge e'_2.
\]
Now,
$c_1(M_{m,n}|_D)=m[\omega_1]+n[\omega_2]$, $[\omega_0|_D]=\alpha [\omega_1]-\beta[\omega_2]$ for some $\alpha,\beta>0$ and we must have $m\beta - n\alpha=0$, \ie $(\alpha,\beta)=a(m,n)$ for some $a\neq 0$. In particular, $m$ and $n$ must have the same sign and so up to changing the circle bundle to its dual we can assume that both are positive.
\end{proof}

\section{Invariant half-flat structures and Hitchin's flow}
\label{sec:half:flat}

Let $M$ be a cohomogeneity one $7$-manifold acted upon by $\sunitary{2}\times\sunitary{2}$ and described by one of the group diagrams  \eqref{eq:group:diagram:B7}, \eqref{eq:group:diagram:D7} or \eqref{eq:group:diagram:C7(m,n)}. The first step of our analysis is to describe torsion-free \gtstr s on the open dense subset of principal orbits, \ie on a cyclinder of the form $(0,\infty)\times\sunitary{2}\times\sunitary{2}/K_0$, where $K_0$ is either trivial or the finite cyclic subgroup $\Z_{2|m+n|}=K_{m,n}\cap K_{2,-2}$ for coprime integers $m,n$. Any cohomogeneity one {\gtstr} on $(0,\infty)\times\sunitary{2}\times\sunitary{2}/K_0$ can be thought of as a $1$-parameter family of invariant $\sunitary{3}$--structures on the principal orbit $\sunitary{2}\times\sunitary{2}/K_0$. The condition that the {\gtstr} be torsion-free can then be written as ``static'' and ``evolution'' equations for the corresponding family of $\sunitary{3}$--structures. The ``static'' equations constrain the torsion of the $\sunitary{3}$--structures: solutions are called \emph{half-flat} structures. The ``evolution'' equations form a system of first-order ODEs known as \emph{Hitchin's flow}. The notions of stable forms and volume functionals introduced by Hitchin \cite{Hitchin:Stable:forms} allow one to interpret the ODE system as a Hamiltonian system on the space of invariant half-flat structures. We also give an alternative description of cohomogeneity one torsion-free \gtstr s based on Hitchin's description of torsion-free \gtstr s as critical points of a volume functional on the space of closed stable $3$-forms on a $7$-manifold in a fixed cohomology class. Palais' Principle of Symmetric Criticality (which holds in our context since the symmetry group is compact \cite{Palais}*{Theorem 5.4}) allows us to give an alternative Lagrangian formulation of Hitchin's flow.       

\subsection{Invariant half-flat structures on \texorpdfstring{$\sunitary{2}\times\sunitary{2}$}{SU(2) times SU(2)}}

The holonomy reduction of a Riemannian $7$-manifold to $\gtwo$ is conveniently expressed as the existence of a closed and coclosed (in fact, parallel) $3$-form $\varphi$ with special algebraic properties at each point. The natural action of $\tu{GL}(7,\R)$ on $\Lambda^3(\R^7)^\ast$ has two open orbits; one of these is isomorphic to $\tu{GL}(7,\R)/\gtwo$ and we say that a $3$-form $\varphi$ on a $7$-manifold $M$ is positive if $\varphi_x$ lies in this orbit for every $x\in M$. Since the stabiliser of a positive $3$-form is conjugate to $\gtwo$, the existence of $\varphi$ is equivalent to the reduction of the frame bundle of $M$ to $\gtwo$. Moreover, since $\gtwo$ is a subgroup of $\sorth{7}$ every positive $3$-form $\varphi$ defines a Riemannian metric $g_\varphi$ and volume form $\dvol_\varphi$ on $M$.

A {\gtstr} on a family of parallel hypersurfaces such as the principal orbits in a cohomogeneity one manifold is described by a $1$-parameter family of \emph{half-flat} $\sunitary{3}$--structures.

\begin{definition}\label{dfn:Hald:Flat}
An {\suthreestr} on a $6$-manifold is a pair of smooth differential forms $(\omega,\Omega)$, where $\omega$ is a non-degenerate $2$-form and $\Omega$ is a complex volume form, satisfying the algebraic constraints 
\begin{equation}\label{eq:SU(3):structure:Constraints}
\omega\wedge\Real\Omega=0, \qquad \tfrac{1}{6}\omega^3 = \tfrac{1}{4}\Real\Omega\wedge\Imag\Omega.
\end{equation}
A \emph{half-flat structure} is an $\sunitary{3}$--structure $(\omega,\Omega)$ such that
\begin{equation}\label{eq:Half:Flat}
d\omega\wedge\omega=0=d\Real\Omega.
\end{equation}
\end{definition}

Invariant half-flat structures on $\sunitary{2}\times\sunitary{2}$ have been studied by Schulte-Hengesbach \cite{Schulte-Hengesbach}*{Chapter 5} and Madsen--Salamon \cite{Madsen:Salamon}. We now briefly summarise their results. It is useful to recall first the formal geometric set-up introduced by Hitchin in \cite{Hitchin:Stable:forms}*{\S 6}. Let $N$ be a compact $6$-manifold. Let $\mathcal{U}$ and $\mathcal{V}$ be the space of closed \emph{stable} $4$-forms and $3$-forms representing a fixed pair of cohomology classes on $N$ of degree $4$ and degree $3$ respectively. Here a $4$-form on $N$ is stable if it can be written as $\frac{1}{2}\omega^2$ for a non-degenerate $2$-form $\omega$ (uniquely determined up to sign) and a $3$-form is stable if it is the real part of a holomorphic volume form $\Omega$ (in this case the imaginary part of $\Omega$ is uniquely determined by its real part). The tangent space of $\mathcal{U}\times\mathcal{V}$ at any point $(\frac{1}{2}\omega^2, \Real\Omega)$ is the product of affine spaces $\Omega^4_{exact}\oplus\Omega^3_{exact}$. There is a non-degenerate pairing 
\[
\langle \sigma, \rho \rangle = \int_N {\alpha\wedge \rho}= - \int_N{\sigma\wedge\beta},
\] 
where $\sigma=d\alpha$ is an exact $4$-form and $\rho=d\beta$ an exact $3$-form, which can be used to define a symplectic form $\Psi$ on $\mathcal{U}\times\mathcal{V}$.

The diffeomorphism group of $N$ acts naturally on $\mathcal{U}\times\mathcal{V}$ preserving $\Psi$. Given $(\frac{1}{2}\omega^2, \Real\Omega)\in\mathcal{U}\times\mathcal{V}$ and a vector field $X$ on $N$, the infinitesimal action $v_X$ of $X$ on $\mathcal{U}\times\mathcal{V}$ is by Lie derivative and therefore
\[
\left( v_X\lrcorner\Psi \right) (\sigma,\rho)=\int_{N}{ \left( X\lrcorner\tfrac{1}{2}\omega^2 \right) \wedge\rho  - \sigma\wedge\left( X\lrcorner\Real\Omega\right)}.
\]
We claim that $v_X$ is a Hamiltonian vector field on $\mathcal{U}\times\mathcal{V}$. Indeed, it is clear that the function $\mu_X$ on $\mathcal{U}\times\mathcal{V}$ defined by
\[
\mu_X (\tfrac{1}{2}\omega^2,\Real\Omega) = \int_N { \left( X\lrcorner\tfrac{1}{2}\omega^2 \right) \wedge\Real\Omega }= - \int_N {\tfrac{1}{2}\omega^2\wedge\left( X\lrcorner\Real\Omega\right)}
\]
satisfies $d\mu _X=v_X\lrcorner\Psi$. Moreover, since $\omega$ is non-degenerate $X\mapsto X\lrcorner\,\omega$ is an isomorphism: rewriting\[
\mu_X \left(\tfrac{1}{2}\omega^2,\Real\Omega\right) = \int_N { (X\lrcorner\,\omega)\wedge\omega\wedge\Real\Omega}
\]
we see that the vanishing of the moment map $\mu_X$ for all $X$ is equivalent to $\omega\wedge\Real\Omega=0$.

Hitchin defines volume functionals $V(\tfrac{1}{2}\omega^2)$ and $V(\Real\Omega)$ and a diffeomorphism-invariant functional $H$ by taking a certain linear combination of these:
\[
H(\tfrac{1}{2}\omega^2,\Real\Omega)=\int_{N}{ 2\omega^3-3\Real\Omega\wedge\Imag\Omega }.
\]
By diffeomorphism invariance $H$ descends to the symplectic quotient 
\[\mathcal{M}=\mu^{-1}(0)/\text{Diff}_0(N).\] The zero-level set of $H$ in $\mathcal{M}$ can almost be identified with the moduli space of half-flat structures on $N$ in the given cohomology classes: the second constraint in \eqref{eq:SU(3):structure:Constraints} is only satisfied in an integral sense. However, if $N$ is a homogeneous space and we restrict to invariant forms (and invariant diffeomorphisms) then the zero-level set of $H$ in $\mathcal{M}$ does indeed parametrise invariant half-flat structures on $N$ with fixed cohomology classes.

We now specialise this general framework to the case where $N=\sunitary{2}\times\sunitary{2}$ (or, later, a finite free quotient of this). The group of Lie algebra inner automorphisms of $\Lie{su}_2\oplus\Lie{su}_2$ (\ie the group of invariant diffeomorphisms of $N$ isotopic to the identity) is $\Aut=\sorth{3}\times\sorth{3}$. Fix $(p,q)\in\R^2$. We consider the space $\mathcal{U}$ of invariant non-degenerate $2$-forms $\omega$ such that $d\omega\wedge\omega=0$ (since $H^4(N)=0$ the closed $4$-form $\omega^2$ is necessarily exact) and the space $\mathcal{V}$ of invariant closed stable $3$-forms $\Real\Omega$ of the form
\[
\Real\Omega = p\, e_1 \wedge e_2\wedge e_3 + q\, e'_1 \wedge e'_2\wedge e'_3 + d\eta
\]  
for some invariant $2$-form $\eta$. Here $e_i$, $e'_i$ denote the left-invariant $1$-forms defined 
in \eqref{eq:Maurer:Cartan}. By \cite{Schulte-Hengesbach}*{Chapter 5, Lemmas 1.1 and 1.3} $\mathcal{U}$ and $\mathcal{V}$ are each identified with open subsets of the space $M_{3\times 3}$ of real $3\times 3$ matrices via
\[
\omega = \sum_{i,j}A_{ij}\, e_i\wedge e'_j, \qquad \eta =\sum_{i,j} B_{ij}\, e_i\wedge e'_j.
\]
Via the double-cover $\sorth{4}\ra \sorth{3}\times\sorth{3}$, we identify the $\sorth{3}\times\sorth{3}$--representation $M_{3\times 3}$ with the $\sorth{4}$--representation $\tu{Sym}^2_0(\R^4)$ \cite{Madsen:Salamon}*{Lemma 1}. Then $(\mathcal{U}\times\mathcal{V},\Psi)$ can be identified with $T^\ast \tu{Sym}^2_0(\R^4)$ endowed with its canonical symplectic form. The vanishing of the moment map $\mu$ for the action of $\sorth{4}$ guarantees that the two matrices $A$ and $B$, thought of as traceless symmetric $4\times 4$ matrices, commute \cite{Madsen:Salamon}*{Theorem 1}. By singular symplectic reduction $\mu^{-1}(0)/\sorth{4}$ is the cotangent space of $\tu{Sym}^2_0(\R^4)/\sorth{4}$ and is identified with $\R^3 \times \R^3/W$, where $W$ is the symmetric group on $3$ elements acting diagonally on $\R^3\times \R^3$ \cite{Madsen:Salamon}*{Corollary 1}.

Concretely \cite{Schulte-Hengesbach}*{Chapter 5, Theorem 1.4} up to the action of $\Aut$ we can assume that
\begin{equation}\label{eq:Invariant:Half:Flat}
\begin{gathered}
\omega = \alpha_1\, e_1\wedge e'_1 + \alpha_2\, e_2\wedge e'_2 +\alpha_3\, e_3\wedge e'_3,\\
\Real\Omega = p\, e_1 \wedge e_2\wedge e_3 + q\, e'_1 \wedge e'_2\wedge e'_3 + d\left( a_1\, e_1\wedge e'_1 + a_2\, e_2\wedge e'_2 +a_3\, e_3\wedge e'_3\right).
\end{gathered}
\end{equation}
The $2$-form $\omega$ is non-degenerate and $\Real\Omega$ is stable if and only if, respectively, $\alpha_1, \alpha_2,\alpha_3>0$ and
\begin{multline}\label{eq:Stable:3:Form}
\Lambda (a_1,a_2,a_3)=a_1^4+a_2^4+a_3^4 - 2a_1^2 a_2^2 - 2a_2^2 a_3^2 - 2a_3^2 a_1^2 + \\4(p-q) a_1 a_2 a_3 + 2pq (a_1^2 + a_2^2 + a_3^2) +p^2 q^2 <0.
\end{multline}
Moreover, the second constraint in \eqref{eq:SU(3):structure:Constraints} forces
\[
2\alpha_1\alpha_2\alpha_3 = \sqrt{-\Lambda (a_1,a_2,a_3)}\,.
\]
Note that there is a residual ambiguity in the parametrisation \eqref{eq:Invariant:Half:Flat}: the Weyl group $W$ acts permuting $(\alpha_1,\alpha_2,\alpha_3)$ and $(a_1,a_2,a_3)$ simultaneously.

\begin{remark*}
Due to different choices of basis of $\Lie{su}_2$, the definition of $\Lambda$ in \eqref{eq:Stable:3:Form} is different from the one given by Schulte-Hengesbach in \cite{Schulte-Hengesbach}*{Chapter 5, Equation (1.7)}. The two formulas are related by $a_i\mapsto -a_i$ for all $i=1,2,3$.
\end{remark*}

\begin{remark}\label{rmk:Extra:U(1):SU(2):symmetry} In addition to the (left) $\sunitary{2}\times\sunitary{2}$--invariance, when $\alpha_1=\alpha_2$ and $a_1=a_2$, the half-flat structure \eqref{eq:Invariant:Half:Flat} is invariant under the right action of $\triangle\unitary{1}$ (in fact, of the normaliser $\tu{N}$ of $\unitary{1}$ in $\sunitary{2}$ if we don't restrict to continuous symmetries) on $\sunitary{2}\times\sunitary{2}$. Similarly, if $\alpha_1=\alpha_2=\alpha_3$ and $a_1=a_2=a_3$ then the half-flat structure is invariant under the right action of $\triangle\sunitary{2}$. The Bryant--Salamon AC {\gtmetric} on the spinor bundle of $S^3$ \cite{Bryant:Salamon} has the additional $\triangle \sunitary{2}$ symmetry, \cf Example \ref{ex:Bryant:Salamon}. All the global examples we construct in this paper will have the additional $\triangle\unitary{1}$ symmetry. 
\end{remark}

\smallskip
\noindent \emph{The case of non-trivial stabiliser of the principal orbits.}\label{sec:Half:Flat:Stabiliser}
We now extend the previous discussion to the case where the stabiliser $K_0$ of the principal orbits is non-trivial. In the situation of interest, the isotropy representation of $K_0=K_{m,n}\cap K_{2,-2}\simeq \Z_{2|m+n|}$ on $\Lie{su}_2\oplus\Lie{su}_2$ is generated by the automorphism
\begin{equation}\label{eq:Isotropy:Representation}
T=\diag (\zeta^{2|n|},1,\zeta^{2|n|},1),
\end{equation}
where we identify $\tu{span}(E_1,E_2)$ and $\tu{span}(E'_1,E'_2)$ with $\C$ and $\zeta$ is a generator of $\Z_{2|m+n|}$. The subgroup $\Aut_T$ of $\Aut = \sorth{3}\times\sorth{3}$ of inner automorphisms that commute with $T$ is $\Aut_T =\orth{2}\times\orth{2}$, where $\orth{2}$ is the subgroup of $\sorth{3}$ that fixes $E_3$. Similarly, an invariant $2$-form
\[
\sum_{i,j}C_{ij}\, e_i\wedge e'_j
\]
on $\sunitary{2}\times\sunitary{2}$ determined by a $3\times 3$ matrix $C$ descends to $\sunitary{2}\times\sunitary{2}/K_0$ if and only if $T^t\left( \begin{array}{cc} 0 & C \\ -C^t & 0 \end{array}\right) T= \left( \begin{array}{cc} 0 & C \\ -C^t & 0 \end{array}\right)$. A computation shows that
\[
C=\left( \begin{array}{ccc} c_{11} & c_{12} & 0\\ c_{21} & c_{22} & 0 \\ 0 & 0 & c_{33}\end{array}\right),
\]
with the upper-left $2\times 2$ block commuting with the rotation of $\R^2$ of angle $\frac{2\pi|n|}{|n+m|}$. The latter condition forces $c_{11}=c_{22}$ and $c_{12}=-c_{21}$ unless $e^{\frac{2\pi i|n|}{|n+m|}}=\pm 1$. This can only happen if there exists $d\in\Z$ such that $(d+1)m + (d-1)n=0$.

The constraints that $K_0$--invariant forms must satisfy and the smaller group of automorphisms $\Aut_T$ play off against each other and, as in the case of trivial stabiliser of the principal orbit, we deduce the following proposition.

\begin{prop}\label{prop:Invariant:Half:Flat} 
Up to the action of $\Aut_T$ any invariant half-flat structure $(\omega,\Omega)$ on $\sunitary{2}\times\sunitary{2}/K_{m,n}\cap K_{2,-2}$ can be put in the normal form 
\[
\begin{gathered}
\omega = \alpha_1\, e_1\wedge e'_1 + \alpha_2\, e_2\wedge e'_2 +\alpha_3\, e_3\wedge e'_3,\\
\Real\Omega = p\, e_1 \wedge e_2\wedge e_3 + q\, e'_1 \wedge e'_2\wedge e'_3 + d\left( a_1\, e_1\wedge e'_1 + a_2\, e_2\wedge e'_2 +a_3\, e_3\wedge e'_3\right) 
\end{gathered}
\]
with $\alpha_1, \alpha_2,\alpha_3,-\Lambda (a_1,a_2,a_3)>0$ and $2\alpha_1\alpha_2\alpha_3 = \sqrt{-\Lambda (a_1,a_2,a_3)}$. Furthermore, $\alpha_1=\alpha_2$ and $a_1=a_2$ unless there exists $d\in\Z$ such that $(d+1)m + (d-1)n=0$.
\end{prop}

\subsection{The fundamental ODE system}

We now introduce the evolution equations for a $1$-parameter family of invariant half-flat structures on $\sunitary{2}\times\sunitary{2}/K_0$ to define a torsion-free \gtstr. 

\subsubsection*{The Hamiltonian formulation: Hitchin's flow}
Consider a {\gtstr} $\varphi = dt\wedge\omega + \Real\Omega$ on a cylinder $(0,t_0)\times N$. Here $(\omega,\Omega)$ is a $1$-parameter family of $\sunitary{3}$--structures on the $6$-manifold $N$. Then we have $\ast\varphi = -dt\wedge\Imag\Omega + \frac{1}{2}\omega^2$ and the condition that $\varphi$ be closed and coclosed is equivalent to the half-flat equations \eqref{eq:Half:Flat} for $(\omega,\Omega)$ together with the evolution equations
\begin{equation}\label{eq:Half:Flat:Evolution}
\partial_t\Real\Omega = d\omega, \qquad \partial_t(\omega^2)=-2d\Imag\Omega.
\end{equation}

We specialise now to the case of a cohomogeneity one {\gtstr} on $M=(0,t_0)\times N$ with $N=\sunitary{2}\times \sunitary{2}/K_0$. Fix $p,q\in \R$ and consider a closed $\sunitary{2}\times \sunitary{2}$--invariant $3$-form
\begin{equation}\label{eq:Invariant:Closed:G2}
\varphi = p\, e_1\wedge e_2\wedge e_3 + q\, e'_1\wedge e'_2\wedge e'_3 + d\left( a_1 \, e_1\wedge e'_1 + a_2 \, e_2\wedge e'_2 + a_3 \, e_3\wedge e'_3\right),  
\end{equation}
where $d$ denotes the differential in $7$ dimensions. We rewrite $\varphi$ as $\varphi= dt\wedge\omega+\Real\Omega$ for a $1$-parameter family of pairs of differential forms
\[
\begin{gathered}
\omega = \dot{a}_1 \, e_1\wedge e'_1 + \dot{a}_2 \, e_2\wedge e'_2 + \dot{a}_3 \, e_3\wedge e'_3,\\
\Real\Omega = p\, e_1\wedge e_2\wedge e_3 + q\, e'_1\wedge e'_2\wedge e'_3 + d\left( a_1 \, e_1\wedge e'_1 + a_2 \, e_2\wedge e'_2 + a_3 \, e_3\wedge e'_3\right),
\end{gathered}
\]
where now $d$ is the differential on the $6$-manifold $\sunitary{2}\times\sunitary{2}/K_0$ and $\dot{a}_i=\frac{da_i}{dt}$. Assuming that $t$ is the arc-length parameter along a geodesic meeting all principal orbits orthogonally, the $3$-form $\varphi$ is a {\gtstr} if and only if the pair $(\omega,\Omega)$ defines an $\sunitary{3}$--structure for all $t$, \ie
\[
\dot{a}_i>0, \qquad \Lambda (a_1,a_2,a_3)<0, \qquad 2\dot{a}_1\dot{a}_2\dot{a}_3 = \sqrt{-\Lambda (a_1,a_2,a_3)}.
\]

\begin{remark}\label{rmk:arc:length}
The last condition is equivalent to the requirement that $t$ be the arc-length parameter along a geodesic meeting all principal orbits orthogonally. In the following it will sometimes be convenient to drop this constraint. In that case, we will only require $\dot{a}_i>0$ and $\Lambda (a_1,a_2,a_3)<0$.
\end{remark}

Set $y_i = a_i$ and $x_i = \dot{a}_j\dot{a}_k$. Then we have
\[
\begin{gathered}
\tfrac{1}{2}\omega^2 = x_1\, e_2\wedge e'_2\wedge e_3\wedge e'_3 + x_2\, e_3\wedge e'_3\wedge e_1\wedge e'_1 +x_3\, e_1\wedge e'_1\wedge e_2\wedge e'_2,\\
\Real\Omega = p\, e_1\wedge e_2\wedge e_3 + q\, e'_1\wedge e'_2\wedge e'_3 + d\left( y_1\, e_1\wedge e'_1 + y_2\, e_2\wedge e'_2 + y_3\, e_3 \wedge e'_3\right).
\end{gathered}
\] 
We are going to rewrite \eqref{eq:Half:Flat:Evolution} as an ODE system for the pair $(x,y)\in \R^3\times \R^3$. In Hitchin's formalism, the evolution equations \eqref{eq:Half:Flat:Evolution} are interpreted as the Hamiltonian flow of the Hamiltonian function $H$ on the space $\mathcal{U}\times\mathcal{V}$ of half-flat structures in fixed cohomology classes \cite{Hitchin:Stable:forms}*{Theorem 8}. Specialising to the $\sunitary{2}\times\sunitary{2}$--invariant setting, the Hamiltonian $H$ is given by
\begin{equation}\label{eq:Hamiltonian}
H(x,y)=\sqrt{-\Lambda(y_1,y_2,y_3)}-2\sqrt{x_1 x_2 x_3},
\end{equation}
where $\Lambda$ was defined in \eqref{eq:Stable:3:Form}. Then \eqref{eq:Half:Flat:Evolution} is equivalent to the Hamiltonian system
\begin{equation}
\begin{aligned}
\label{eq:Fundamental:ODE}
\dot{x_i}&=\frac{\partial H}{\partial y_i}= \frac{2}{\sqrt{-\Lambda (y_1,y_2,y_3)}}\left( y_i (-y_i^2 + y_j^2 + y_k^2-pq) - (p-q) y_j y_k\right),\\
\dot{y}_i &=-\frac{\partial H}{\partial x_i}=\frac{x_j x_k}{\sqrt{x_1 x_2 x_3}}.
\end{aligned}
\end{equation}
In particular, $H$ is constant along the flow. In fact $H=0$ along the flow since we require the normalisation $2\,\omega^3=3\Real\Omega\wedge\Imag\Omega$.

\begin{remark*}
Alternatively, one can use \cite{Schulte-Hengesbach}*{Corollary 1.5} (with the usual change of sign due to our different choice of basis of $\Lie{su}_2$) to write
\begin{align*}
\sqrt{-\Lambda}\Imag\Omega &= \left( 2y_1 y_2 y_3 -p(y_1^2 + y_2 ^2 + y_3 ^2 + pq)\right) e_1\wedge e_2 \wedge e_3\\
&+ \left( 2y_1 y_2 y_3 +q(y_1^2 + y_2 ^2 + y_3 ^2 + pq)\right) e'_1\wedge e'_2 \wedge e'_3\\
&+ \left( y_i (y_i^2 - y_j ^2 - y_k ^2 + pq) -2qy_j y_k\right) e_i\wedge e'_j \wedge e'_k\\
&+ \left( y_i (y_i^2 - y_j ^2 - y_k ^2 + pq) +2py_j y_k\right) e'_i\wedge e_j \wedge e_k,
\end{align*}
with the convention that we sum over cyclic permutations $(ijk)$ of $(123)$. The equivalence between \eqref{eq:Half:Flat:Evolution} and \eqref{eq:Fundamental:ODE} is then immediate.
\end{remark*}

\begin{example}\label{ex:Bryant:Salamon} By Remark \ref{rmk:Extra:U(1):SU(2):symmetry}, solutions of \eqref{eq:Fundamental:ODE} satisfying $x_1=x_2=x_3=x$ and $y_1=y_2=y_3=y$ correspond to $\sunitary{2}^3$--invariant torsion-free \gtstr s. In this highly symmetric setting, \eqref{eq:Fundamental:ODE} reduces to algebraic equations. Indeed, the vanishing of the Hamiltonian function $H$ describes a curve $4x^3=3y^4-4(p-q)y^3-6pqy^2-p^2 q^2$ in the $(x,y)$--plane. The solution with $p=0=q$ induces a conical metric: the $\gtwo$--cone $\tu{C}$ over the homogeneous nearly K\"ahler structure on $S^3 \times S^3$. Since the asymptotic behaviour of the curve $H(x,y)=0$ for large $y$ is independent of $p$ and $q$, solutions for arbitrary $p$ and $q$ have one AC end asymptotic to $\tu{C}$. As observed in \cite{Brandhuber}*{\S 2.1}, only when $(p,q)=(r_0^3,-r_0^3), (-r_0^3,0)$ or $(0,r_0^3)$ for some $r_0>0$, do solutions close smoothly on a singular orbit. Up to scale and (not necessarily $\sunitary{2}\times\sunitary{2}$--equivariant) diffeomorphisms there exists a unique complete AC metric asymptotic to the cone $\tu{C}$: the Bryant--Salamon metric on the spinor bundle of $S^3$ 
\cite{Bryant:Salamon}*{\S 3}.
\end{example}

\subsubsection*{The Lagrangian formulation}
Later in the paper it will be useful to have a different formulation of \eqref{eq:Fundamental:ODE}, first introduced by Brandhuber in \cite{Brandhuber}.

Let $\varphi$ be a closed {\gtstr} on a $7$-manifold $M$ and denote by $g_\varphi$ the induced Riemannian metric. Hitchin \cite{Hitchin:Stable:forms}*{Theorem 1} showed that the equation $d^\ast\varphi=0$ is the Euler--Lagrange equation for the volume functional $\varphi \mapsto \tu{Vol}(M,g_\varphi)$ restricted to variations of $\varphi$ amongst closed $3$-forms with fixed cohomology class. Assume that $M=(0,s_0)\times\sunitary{2}\times\sunitary{2}/K_0$, with coordinate $s$ on the first factor, and let $\varphi$ be a closed $\sunitary{2}\times \sunitary{2}$--invariant {\gtstr} of the form \eqref{eq:Invariant:Closed:G2}, \ie
\[
\varphi = p\, e_1\wedge e_2\wedge e_3 + q\, e'_1\wedge e'_2\wedge e'_3 + d\left( a_1\, e_1\wedge e'_1 + a_2\, e_2\wedge e'_2 + a_3\, e_3 \wedge e'_3\right).
\]
Then up to a constant we have
\begin{equation}\label{eq:Volume:Functional}
\tu{Vol}(M,g_\varphi) = \int_0^{s_0}{ L(a,a')\, ds},\qquad L(a,a') = \left( - a'_1 a'_2 a'_3 \Lambda (a_1, a_2, a_3) \right)^{\frac{1}{3}}.
\end{equation}
Here $a'_i = \frac{da_i}{ds}$. In order to prove \eqref{eq:Volume:Functional} observe that the fixed parameter $s$ and the arc-length parameter $t$ are related by
\[
2\left( \frac{ds}{dt}\right) ^3 a'_1 a'_2 a'_3 = \sqrt{-\Lambda (a_1, a_2, a_3)}
\]
and that the volume form of $g_\varphi$ is
\[
\dvol_{g_\varphi} = \tfrac{1}{6}dt\wedge\omega^3= \dot{a}_1 \dot{a}_2\dot{a}_3\, dt\wedge e_1\wedge e'_1\wedge e_2\wedge e'_2\wedge e_3\wedge e'_3.
\]
By the Principle of Symmetric Criticality, we conclude that $\varphi$ is coclosed if and only if the triple $a=(a_1, a_2, a_3)$ satisfies the second-order Lagrangian system
\[
\left( \partial_{a'_i}L(a,a')\right)' - \partial_{a_i}L(a,a')=0.
\] 
However, since Hitchin's volume functional is invariant under diffeomorphisms, we know \emph{a priori} that $\left( \partial_{a'_i}L(a,a')\right)'-\partial_{a_i}L(a,a')$ is orthogonal to the vector field $(a'_1, a'_2, a'_3)$. We therefore reduce to an ODE system of two second-order equations in three variables. This reformulation makes sense since we have the freedom to change the parametrisation $s$.

We carry out the relevant calculations explicitly in the case where $a_1=a_2$, \ie when there is an additional $\unitary{1}$ symmetry, \cf Remark \ref{rmk:Extra:U(1):SU(2):symmetry}. All the global solutions we find in this paper will admit this additional $\unitary{1}$ symmetry. Set $a:=a_1=a_2$, $b:=a_3$ and
\[
F(a,b):=-\Lambda (a,a,b) = 4a^2 (b-p)(b+q) - (b^2 +pq)^2.
\]
Let $F_a$ and $F_b$ denote the partial derivatives of $F$. Then we calculate:
\[
\begin{gathered}
9L^5\left(\frac{d}{ds}\left( \frac{\partial L}{\partial a'}\right) - \frac{\partial L}{\partial a} \right)= (a')^2 b' F \left( 2F (a' b'' - b' a'') - a' b' (a' F_a - 2b' F_b)\right),\\
9L^5\left(\frac{d}{ds}\left( \frac{\partial L}{\partial b'}\right) - \frac{\partial L}{\partial b} \right)= (a')^3 F \left( -2F (a' b'' - b' a'') + a' b' (a' F_a - 2 b' F_b)\right).
\end{gathered}
\]
Thus the pair of functions $(a,b)$ yields a torsion-free $\sunitary{2}\times\sunitary{2}\times\unitary{1}$--invariant {\gtstr} if and only if
\begin{equation}\label{eq:Fundamental:ODE:Brandhuber:U(1)}
2F (a'b''-b'a'') - a'b' \left( a' F_a - 2b' F_b  \right)=0.
\end{equation}

As a sanity check, note that \eqref{eq:Fundamental:ODE:Brandhuber:U(1)} can be immediately derived from \eqref{eq:Fundamental:ODE}. Indeed, in the presence of the additional $\unitary{1}$--symmetry, \eqref{eq:Fundamental:ODE} becomes the ODE system
\begin{equation}\label{eq:Fundamental:ODE:U(1)}
\begin{aligned}
\dot{x}_1 &= \frac{F_a(y_1,y_2)}{4\sqrt{F(y_1,y_2)}}, \qquad & \dot{x}_2 &= \frac{F_b(y_1,y_2)}{2\sqrt{F(y_1,y_2)}},\\
\dot{y}_1 &= \frac{x_1 x_2}{\sqrt{x_1^2 x_2}}, \qquad & 
\dot{y}_2 &= \frac{x_1^2}{\sqrt{x_1^2 x_2}},
\end{aligned}
\end{equation}
for the four functions $x_1 = \dot{a}\dot{b}$, $x_2 = \dot{a}^2$, $y_1=a$, $y_2=b$. Then \eqref{eq:Fundamental:ODE:Brandhuber:U(1)} is an immediate consequence of \eqref{eq:Fundamental:ODE:U(1)}. Moreover, the variable $t$ in \eqref{eq:Fundamental:ODE:U(1)} is the arc-length parameter along a geodesic meeting all principal orbits orthogonally and therefore we have the further normalisation $2\dot{a}^2\dot{b}=\sqrt{F(a,b)}$, \cf Remark \ref{rmk:arc:length}. 

\subsection{The induced metric}

For later use, we briefly discuss properties of the map $\varphi\mapsto g_\varphi$ in the $\sunitary{2}\times\sunitary{2}$-invariant setting. Lemma \ref{lem:Metric} below shows that Hitchin's flow \eqref{eq:Fundamental:ODE} can be regarded as an evolution equation for the family of Riemannian metrics induced by the $1$-parameter family of half-flat structures. In fact, cohomogeneity one $\sunitary{2}\times\sunitary{2}$--invariant \gtmfd s are described by a first-order ODE system for the metric coefficients in the work of Brandhuber--Gomis--Gubser--Gukov \cite{BGGG}, Cveti\v{c}--Gibbons--L\"u--Pope \cites{CGLP:Spin(7)&G2,CGLP:C7,CGLP:C7:tilde,CGLP:M:Conifolds}, Hori--Hosomichi--Page--Rabad\'an--Walcher \cite{A7:ALC} and Bazaikin--Bogoyavlenskaya \cites{Bazaikin:Bogoyavlenskaya,Bogoyavlenskaya}.

Consider the $\sunitary{2}\times\sunitary{2}$--invariant closed \gtstr
\[
\varphi = p\, e_1\wedge e_2\wedge e_3 + q\, e'_1 \wedge e'_2 \wedge e'_3 + d\left( a_1\, e_1 \wedge e'_1 + a_2\, e_2 \wedge e'_2 + a_3 \, e_3 \wedge e'_3\right).
\]
The induced metric $g_\varphi$ takes the form $g_\varphi = dt^2 + g_t$, where $t$ is the arc-length parameter along a geodesic meeting all $\sunitary{2}\times\sunitary{2}$--orbits orthogonally and $g_t$ is a $1$-parameter family of $\sunitary{2}\times\sunitary{2}$--invariant metrics on the principal orbits. By \cite{Schulte-Hengesbach}*{Corollary 1.5} (with the usual change of signs)
\begin{equation}\label{eq:Metric}
\tfrac{1}{2}\sqrt{-\Lambda}\, g_t = \dot{a}_i \left( a_j a_k - p a_i\right) e_i\otimes e_i + \dot{a}_i \left( a_j a_k + q a_i\right) e'_i\otimes e'_i + \dot{a}_i \left( a_i^2 - a_j ^2 - a_k ^2 -pq\right) e_i\otimes e'_i,
\end{equation}
where $(ijk)$ runs over cyclic permutations of $(123)$.

\begin{remark*}
When $p+q=0$ the metric $g_t$ is invariant under the  involution generated by the outer automorphism of $\sunitary{2}\times\sunitary{2}$ that exchanges the two factors. This additional symmetry however does not preserve the {\gtstr} $\varphi$.
\end{remark*}

\begin{lemma}\label{lem:Metric}
Fix $p,q\in\R$ and assume that $\dot{a}_1,\dot{a}_2,\dot{a}_3$ and $2 \dot{a}_1\dot{a}_2\dot{a}_3=\sqrt{-\Lambda (a_1,a_2,a_3)}$ are all positive. Then $(\dot{a}_1,\dot{a}_2,\dot{a}_3,a_1,a_2,a_3)$ is uniquely determined by the metric $g_t$ (up to discrete symmetries when $p+q=0$). In other words, the map that associates to each invariant half-flat structure in a given cohomology class $(p,q)$ its induced metric is a (local) diffeomorphism.
\proof
By \eqref{eq:Metric}, up to the action of the automorphism group of $\Lie{su}_2\oplus\Lie{su}_2$, the $\sunitary{2}\times\sunitary{2}$--invariant metrics $g$ that could possibly be induced by an invariant half-flat structure must be of the form
\[
g= A_i\, e_i\otimes e_i + B_i\, e'_i\otimes e'_i + C_i\, e_i\otimes e'_i.
\]
Furthermore, $A_i=B_i$ if $p+q=0$. If the coefficients $A_i, B_i, C_i$ are given by \eqref{eq:Metric} then a straightforward computation shows that
\[
4\dot{a}_j\dot{a}_k = \sqrt{\left( 4A_j B_j -C_j^2  \right)\left( 4A_k B_k - C_k^2\right)}.
\]
Thus $\dot{a}_1, \dot{a}_2, \dot{a}_3$ are uniquely determined by the metric $g$. Furthermore, $(p+q)a_i = \dot{a}_j\dot{a}_k (B_i-A_i)$ and therefore $a_1, a_2, a_3$ are also uniquely determined whenever $p+q\neq 0$.

If $p+q=0$ we calculate instead
\[
\begin{aligned}
\dot{a}_j\dot{a}_k\left( A_i + B_i + C_i \right) &= (a_j-a_k-a_i+p)(a_k-a_j-a_i +p), \\
-\dot{a}_j\dot{a}_k\left( A_i + B_i - C_i \right) &= (a_1+a_2+a_3+p)(a_i-a_j-a_k +p),
\end{aligned}
\]
from which it follows, using $A_i=B_i$, that
\[
\Lambda (a_1, a_2, a_3) = (a_1+a_2+a_3+p)(a_1-a_2-a_3 +p)(a_2-a_3-a_1+p)(a_3 - a_1 - a_2 +p).
\]
Together with $2 \dot{a}_1\dot{a}_2\dot{a}_3=\sqrt{-\Lambda (a_1,a_2,a_3)}$, these relations allow us to determine $a_i-a_j-a_k +p$ up to sign:
\[
4(a_i-a_j-a_k +p)^2 = (A_i + B_i -C_i)(A_j + B_j +C_j)(A_k + B_k +C_k).\qedhere
\]
\end{lemma}

A Riemannian quantity that will play an important role later in the paper is the mean curvature of the principal orbits in an $\sunitary{2}\times\sunitary{2}$--invariant \gtmfd. Consider a $1$-parameter family of invariant half-flat structures satisfying Hitchin's flow \eqref{eq:Fundamental:ODE}. The shape operator $L$ of the hypersurface $\{ t=\text{const}\}$ is $L=\frac{1}{2}g_t^{-1}\dot{g}_t$, where the dot stands for time differentiation. The mean curvature is then $l=\text{tr}(L)$ or alternatively the time derivative of the logarithm of the orbital volume function $2\dot{a}_1 \dot{a}_2 \dot{a}_3=\sqrt{-\Lambda (a_1, a_2, a_3)}$. We easily calculate
\begin{equation}\label{eq:Mean:Curvature}
 l=\frac{1}{2(\dot{a}_1\dot{a}_2\dot{a}_3)^2}\sum_{i=1}^3{\dot{a}_i\left( a_i (-a_i^2+a_j^2+a_k^2 -pq) - (p-q)a_j a_k\right)}.
\end{equation}

\section{Local solutions in a neighbourhood of the singular orbit}
\label{sec:sing:extension}

We are interested in $7$-manifolds $M$ with a cohomogeneity one action of $G=\sunitary{2}\times\sunitary{2}$ and one singular orbit $Q=G/K$, where up to automorphisms of $G$ the stabiliser $K$ can be taken as one of the following:
\[
\triangle\sunitary{2}, \qquad \{ 1\} \times \sunitary{2}, \qquad K_{m,n}
\]
for two coprime integers $m,n$, with $K_{m,n}$ as defined in \eqref{eq:K:m:n}. Note that $Q\simeq S^3$ in the first two cases and $Q\simeq S^2 \times S^3$ when $K=K_{m,n}$. The only reason we restrict to coprime $m,n$ is that the resulting $7$-manifold $M$ is simply connected in this case. We recover the general case by taking finite quotients of the simply connected manifolds.

A tubular neighbourhood of the singular orbit $Q$ in $M$ is equivariantly diffeomorphic to a neighbourhood of the zero section in the vector bundle
\[
G\times_K V\ra Q,
\]
where $V$ is an orthogonal representation of $K$ of dimension $d=1+\dim (K)$: $V$ is the standard representation of $\sunitary{2}$ when $K=\triangle\sunitary{2}$ or $\{ 1\} \times \sunitary{2}$, while $V$ is the irreducible $2$-dimensional real representation with weight $2|m+n|$ when $K=K_{m,n}\simeq\sorth{2}$. Fix a vector $v_0$ in the unit sphere in $V$ and denote by $K_0$ the stabiliser of $v_0$ in $K$. Then $K_0$ is the principal orbit stabiliser: $K_0$ is trivial when $K=\triangle\sunitary{2}$ and $K=\{ 1\}\times\sunitary{2}$ and $K_0\simeq \Z_{2|m+n|}$ when $K=K_{m,n}$.

In this section we derive explicit conditions on $p, q$ and the $a_i$ so that the invariant $3$-form
\[
\varphi = p\, e_1\wedge e_2\wedge e_3 + q\, e'_1\wedge e'_2\wedge e'_3 + d\left( a_1\, e_1\wedge e'_1 + a_2\, e_2\wedge e'_2 + a_3\, e_3 \wedge e'_3\right),
\]
defines a smooth invariant closed positive $3$-form in a neighbourhood of the zero section in \mbox{$G\times_K V$}. We then use this analysis to set up and study singular initial value problems for the ODE system \eqref{eq:Fundamental:ODE} that correspond to local cohomogeneity one torsion-free \gtstr s defined in a neighbourhood of the three types of singular orbit. The main result of the section is Proposition \ref{prop:Solutions:Singular:Orbit} where we parametrise the space of cohomogeneity one torsion-free \gtstr s defined in a neighbourhood of each of the three types of singular orbit.  

\subsection{Smooth extension over the singular orbit}

An efficient way of understanding conditions for the smooth extension of tensors along a singular orbit $Q=G/K$ in a cohomogeneity one manifold $M=G\times_K V$ has been given by Eschenburg--Wang \cite{Eschenburg:Wang}*{\S 1}. We briefly recall their approach. Write $\Lie{g} = \Lie{k}\oplus\Lie{p}$ where $\Lie{g}$ and $\Lie{k}$ are the Lie algebras of $G$ and $K$ respectively. Given a point $q\in Q$ we can identify $T_q Q$ with $\Lie{p}$ and $T_q M$ with $\Lie{p}\oplus V$. By $G$--invariance, the $3$-form $\varphi$ is uniquely determined by its restriction to the fibre of $G\times _K V$ over $q$. We can therefore interpret $\varphi$ as a $K$--equivariant map $\varphi\co V\ra \Lambda^3 (\Lie{p}\oplus V)^\ast$. Furthermore, if we fix a point $v_0$ in the unit sphere $\Sph^{d-1}$ in $V$ and denote by $K_0$ its stabiliser in $K$, by $K$--equivariance $\varphi\co V\ra \Lambda^3 (\Lie{p}\oplus V)^\ast$ is uniquely determined by the curve $t\mapsto\varphi (tv_0)$, which must lie in the subspace of $K_0$--invariant $3$-forms on $\Lie{p}\oplus V$. Similarly, every $K$--equivariant map from the unit sphere in $V$ to $\Lambda^3 (\Lie{p}\oplus V)^\ast$ is uniquely determined by its value at $v_0$. Eschenburg and Wang show that $\varphi\co V\ra \Lambda^3(V\oplus\Lie{p})^\ast$ defines a smooth $G$--invariant $3$-form if and only if for all $p\geq 0$ there exists a homogeneous degree--$p$ polynomial $\varphi_p\co \Sph^{d-1}\ra\Lambda^3(V\oplus\Lie{p})^\ast$ with values in the subspace of $K_0$--invariant forms such that $t\mapsto \varphi (tv_0)$ has Taylor series $\sum_{p\geq 0}{\varphi_p (v_0) \, t^{p}}$.

\begin{prop}\label{prop:Smooth:extension:singular:orbit}
Set $G=\sunitary{2}\times\sunitary{2}$ and $(K,V)$ to be one of $\left( \triangle\sunitary{2},\C^2 \right)$, $\left( \{ 1\}\times \sunitary{2},\C^2 \right)$ or $\left( K_{m,n},\R^2_{2|m+n|}\right)$, where $\C^2$ denotes the standard representation of $\sunitary{2}$ and $\R^2_{2|m+n|}$ is the irreducible real $2$-dimensional representation of $K_{m+n}\simeq \sorth{2}$ with weight $2|m+n|$. Let
\[
\varphi = p\, e_1\wedge e_2\wedge e_3 + q\, e'_1\wedge e'_2\wedge e'_3 + d\left( a_1\, e_1\wedge e'_1 + a_2\, e_2\wedge e'_2 + a_3\, e_3 \wedge e'_3\right),
\]
be an invariant closed $3$-form defined on the complement of the zero section in a disc sub-bundle of the vector bundle $G\times_K V\ra G/K$.
\begin{enumerate}[leftmargin=*]
\item If $K=\triangle \sunitary{2}$ then $\varphi$ extends as a smooth positive $3$-form over the zero section of $G\times_K V$ if and only if
\begin{enumerate}
\item $p+q=0$;
\item for $i=1,2,3$, $a_i$ is an even function of $t$ with $a_i (t) = p + \frac{1}{2}\alpha t^2 + O(t^4)$ for some $\alpha\in\R$;
\item $8\alpha^3=p>0$.
\end{enumerate}
\item If $K=\{ 1\}\times\sunitary{2}$ then $\varphi$ extends as a smooth positive $3$-form over the zero section of $G\times_K V$ if and only if
\begin{enumerate}
\item $q=0$;
\item for $i=1,2,3$, $a_i$ is an even function of $t$ vanishing at the origin with $a_i = \frac{1}{2}\alpha_i t^2 + O(t^4)$ for some $\alpha_i>0$;
\item $8\alpha_1\alpha_2\alpha_3=-p>0$.
\end{enumerate}
\item If $K=K_{m,n}$ then $\varphi$ extends as a smooth positive $3$-form over the zero section of $G\times_K V$ if and only if
\begin{enumerate}
\item $mn>0$;
\item $p=-m^2r_0^3$ and $q=n^2 r_0^3$ for some $r_0\neq 0$;
\item $a_3$ is an even function of $t$ with $a_3(0)=mnr_0^3$ and $\ddot{a}_3(0)>0$;
\item $a_1+a_2$ is an odd function of $t$ and $\dot{a}_1 (0)+\dot{a}_2(0)>0$;
\item $a_1=a_2$ unless $m=n=\pm 1$; in the latter case $a_1-a_2$ is an even function of $t$ and $\alpha = \frac{1}{2}\left( a_1(0)-a_2(0)\right)$ satisfies $\alpha^2<r_0^6$.
\end{enumerate}
\end{enumerate}
\end{prop}

\begin{remark*}
In Theorem \ref{thm:Classification:U(1)} we will prove that, up to finite quotients, the proposition covers all possible singular orbits of smooth \gtmetric s with enhanced $\sunitary{2}\times\sunitary{2}\times\unitary{1}$ symmetry.
\end{remark*}

We now prove the three cases of the proposition.

\subsubsection*{The case $K=\triangle\sunitary{2}$.}
The singular orbit is $Q=\sunitary{2}\times\sunitary{2}/\triangle \sunitary{2}\simeq S^3$ and the stabiliser of points on principal orbits is trivial. A tubular neighbourhood of the singular orbit $S^3$ is equivariantly diffeomorphic to $S^3\times \R^4$, where $\sunitary{2}\times\sunitary{2}$ acts on $S^3\times \R^4\subset \HH \oplus \HH$ by $(q_1,q_2)\cdot (x,y)=(q_2 xq_1^\ast,q_1 y)$. Therefore along the ray $\gamma (t) = (1,t)\in S^3\times\HH \subset \HH^2$ we have
\[
\begin{gathered}
E_1 = \tfrac{1}{2}(-i,it), \qquad E_2 = \tfrac{1}{2}(-j,jt),\qquad E_3 = \tfrac{1}{2}(-k,kt),\\
E'_1 = \tfrac{1}{2}(i,0), \qquad E'_2 = \tfrac{1}{2}(j,0), \qquad E'_3 = \tfrac{1}{2}(k,0).
\end{gathered}
\]
In particular, if we define $e_i^\pm = \frac{1}{2}(e_i \pm e'_i)$ and let $t,x_1,x_2,x_3$ be Euclidean coordinates on $\R^4$ we have
\[
e_1^+=2t^{-1}dx_1, \qquad e_2^+ = 2t^{-1}dx_2, \qquad e_3^+=2t^{-1}dx_3.
\]
For $t\neq 0$ the closed positive $3$-form $\varphi$ can then be written as
\[
\tfrac{1}{2}\varphi = V e^-_1\wedge e^-_2\wedge e^-_3 + \sum_{i=1}^3 {e^-_i \wedge \omega_i} + 8(p+q)t^{-3} dx_1\wedge dx_2 \wedge dx_3 + 2(p+q)t^{-1} \sum_{i=1}^3{dx_i\wedge e_j^-\wedge e_k^-}.
\]
Here $\omega_i =-2\dot{a}_i t^{-1} dt\wedge dx_i + 4V_i t^{-2}dx_j\wedge dx_k$ and
\begin{equation}\label{eq:V:Vi}
V=a_1 + a_2 + a_3 +\tfrac{p-q}{2}, \qquad V_i = a_i - a_j - a_k +\tfrac{p-q}{2}.
\end{equation}

If $\varphi$ extends smoothly at $t=0$ then clearly $p+q=0$, since otherwise the coefficient of $dx_1\wedge dx_2\wedge dx_3$ would blow-up. In order to detect subtler conditions for the smooth extension along the singular orbit, following Eschenburg--Wang's analysis we regard $\varphi$ as a map $\varphi\co \HH \ra \Lambda^3 (\HH \oplus \Imag\HH)$ identifying vector spaces with their duals using their standard metrics. As an $\sunitary{2}$--representation we have
\begin{align*}
\Lambda^3 (\HH \oplus \Imag\HH ) & = \Lambda ^3\HH \oplus (\Lambda ^+\HH\otimes\Imag\HH ) \oplus (\Lambda^-\HH\otimes\Imag\HH ) \oplus ( \HH\otimes\Lambda^2\Imag\HH ) \oplus ( \Lambda^3\Imag\HH )\\
 & \simeq \HH \oplus ( \Imag\HH\otimes\Imag\HH )  \oplus ( \R^3\otimes\Imag\HH ) \oplus ( \HH\otimes\Imag\HH ) \oplus\R,
\end{align*}
where $\HH$ and $\Imag\HH$ are the standard and adjoint representations of $\sunitary{2}$, respectively, and we used the fact that the induced action of $\sunitary{2}$ on $\Lambda^2\HH$ acts trivially on anti-self-dual forms and acts via the adjoint representation on the space of self-dual forms (with respect to the volume form $dt\wedge dx_1\wedge dx_2\wedge dx_3$). By applying Eschenburg--Wang's analysis we deduce conditions for the extension of $\varphi$ over the singular orbit as a smooth $3$-form:
\begin{enumerate}[leftmargin=*]
\item $e^-_1\wedge e^-_2\wedge e^-_3$ is $\sunitary{2}$--invariant and therefore $V$ must be even.
\item $e^-_i\wedge \omega^-_i \in \R^3\times\Imag\HH$ corresponds to a degree $2$ polynomial of the form $q\mapsto quq^\ast$ for $q\in\HH$ and some fixed $u\in\Imag\HH$. Here $\omega_i^-$ is the anti-self-dual part of $\omega_i$. We deduce that
\[
\frac{t\dot{a}_i + 2V_i}{t^2}
\]
must be even and vanish at $t=0$.
\item The triple $(e^-_1\wedge\omega^+_1, e^-_2\wedge\omega^+_2,e^-_3\wedge\omega^+_3)$ represents a map $\Imag\HH \ra\Imag\HH$ which is diagonal in the standard basis of $\Imag\HH$ and therefore symmetric. We decompose this map into a multiple of the identity and a traceless part. The equivariant polynomial $\HH \ra \text{Sym}(\Imag\HH)$ which corresponds to the identity by evaluation at $1$ is clearly the constant polynomial $q\mapsto [u\mapsto u]$. On the other hand, an equivariant map $\HH \ra \text{Sym}_0(\Imag\HH)$ which has value $A\in \text{Sym}_0(\Imag\HH)$ at $q=1$ must correspond to the degree $4$ polynomial $q\mapsto \text{Ad}_q \circ A \circ \text{Ad}_{q^\ast}$. We therefore conclude that
\[
\frac{t\dot{a}_i - 2V_i}{t^2}
\]
must have an even Taylor series expansion with $0$th and $2$nd  order coefficients independent of~$i$.
\end{enumerate}
In summary we must have
\[
p+q=0, \qquad a_i = p + \tfrac{1}{2}\alpha t^2 + O(t^4)
\]
for some $\alpha\in\R$.

Finally, the requirements that $\varphi$ be a positive $3$-form and $t$ be the arc-length parameter along a geodesic meeting orthogonally all principal orbits impose further constraints. As in the proof of Lemma \ref{lem:Metric}, it is useful to observe that when $q=-p$ we can factor $\Lambda (a_1, a_2, a_3)$ as $VV_1 V_2 V_3$, where $V, V_i$ are defined in \eqref{eq:V:Vi}. Since $V\approx 4p+\tfrac{3}{2}\alpha t^2$ and $V_i \approx -\tfrac{1}{2}\alpha t^2$ as $t\ra 0$, at leading order in $t$ the conditions $\dot{a}_i, -\Lambda (a_1, a_2, a_3)>0$ and $2\dot{a}_1\dot{a}_2\dot{a}_3=\sqrt{-\Lambda (a_1, a_2, a_3)}$ are equivalent to the requirement 
\[
8\alpha^3 = p>0.
\]

\subsubsection*{The case $K=\{ 1\} \times \sunitary{2}$}
In this case we have $Q=S^3$ and $M=S^3\times\R^4$ with action of $\sunitary{2}\times\sunitary{2}$ given by the left multiplication by unit quaternions on $S^3\subset \HH$ and $\R^4\simeq\HH$. Hence as $K$--representations we have $V=\HH$ and $\Lie{p}=\Imag\HH$. The $1$-forms $e_1, e_2, e_3$ define a coframe on $\Lie{p}$. If $t,x_1,x_2,x_3$ are Euclidean coordinates on $V$, then along the ray $t\in\HH\simeq V$ we have
\[
e'_i = 2t^{-1}dx_i, \qquad i=1,2,3.
\]
Hence we can write
\[
\begin{aligned}
\varphi &= p\, e_1\wedge e_2\wedge e_3 + 8qt^{-3}\, dx_1\wedge dx_2\wedge dx_3 +  \sum_{i=1}^3 t^{-1}\left( -\dot{a}_i + 2t^{-1}a_i \right) e_i \wedge \left( dt\wedge dx_i + dx_j \wedge dx_k\right)\\
&- t^{-1}\left( \dot{a}_i + 2t^{-1}a_i \right) e_i \wedge \left( dt\wedge dx_i - dx_j \wedge dx_k\right)- 2t^{-1}a_i\, dx_i \wedge e_j\wedge e_k,
\end{aligned}
\]
where for each $i=1,2,3$ we have $\epsilon_{ijk}=1$. Now, for $\varphi$ to be smooth we must certainly have $q=0$. Moreover, Eschenburg--Wang's analysis implies that necessary and sufficient conditions for the smoothness of $\varphi$ are that for all $i=1,2,3$
\begin{enumerate}[leftmargin=*]
\item $2t^{-1}\left( \dot{a}_i \pm  2t^{-1}a_i \right)$ has Taylor series involving only even powers of $t$ and also $2t^{-1}\left( \dot{a}_i - 2t^{-1}a_i \right)$ vanishes at the origin;
\item $2t^{-1}a_i$ has Taylor series only involving odd powers of $t$.
\end{enumerate}
These conditions are satisfied if and only if $a_1, a_2, a_3$ are smooth even functions vanishing at $t=0$. Finally, at leading order in $t$ the requirement that $\varphi$ be a positive $3$-form and $t$ be the arc-length parameter along a geodesic meeting all orbits imposes the further constraints
\[
\alpha_i>0, \qquad 8 \alpha_1 \alpha_2 \alpha_3 = -p>0,
\]
where we write $a_i (t) = \frac{1}{2}\alpha_i t^2 + O(t^4)$.

\subsubsection*{The case $K=K_{m,n}$}
In this case the singular orbit is a circle bundle over 
\[D=S^2\times S^2=\sunitary{2}\times\sunitary{2}/T^2\]
and the $7$-manifold $M$ is the total space of a $\C\times\Sph^1$--bundle over $D$. More precisely, $Q= G\times_{T^2}S^1$ and $M=G\times_{T^2}(S^1\times \C)$, where $T^2$ acts on $S^1$ and $\C$ with weights $(m,n)$ and $(2,-2)$ respectively. The $K_{m,n}$--representations $V$ and $\Lie{p}$ are naturally induced by the $T^2$--representations $V=\C_{2,-2}$ and $\Lie{p} = \R \oplus \C_{2,0}\oplus\C_{0,2}$. Since $m$ and $n$ are coprime $K_{m,n}\simeq \unitary{1}$ is embedded in $T^2$ via $e^{i\theta}\mapsto (e^{in\theta}, e^{-im\theta})$ and as real $K_{m,n}$--representations we have $V=\R^2_{2|n+m|}$ and $\Lie{p} = \R \oplus \R^2_{2|n|}\oplus\R_{2|m|}$.

It is convenient to introduce a new $\Z$--basis for the Lie algebra of the maximal torus $T^2$ in $\sunitary{2}\times\sunitary{2}$: fix integers $(r,s)$ with $mr+ns=1$ and consider the basis elements
\[
nE_3 - mE'_3,\qquad rE_3 +sE'_3.
\]
The dual basis of left-invariant $1$-forms is $se_3 -re'_3$, $m e_3 +n e'_3$. Note that
\[
e_3 = n(se_3 -re'_3) + r(me_3 +ne'_3), \qquad e'_3 = -m(se_3 -re'_3) + s(me_3 +ne'_3).
\]
Let $t,x$ be coordinates on $V\simeq\R^2$. Since $nE_3-mE'_3$ has period $\frac{2\pi}{|m+n|}$ in $\sunitary{2}\times\sunitary{2}/K_0$, along the ray $t$ in $V$ we must have (up to changing $x$ into $-x$)
\[
(n+m)(se_3 -re'_3) = t^{-1}dx.
\]
On the other hand $e_1, e_2$ and $e'_1, e'_2$ are coframes on $\Lie{n}=\R^2_{2|n|}$ and $\Lie{n}'=\R^2_{2|m|}$ respectively and $me_3 + ne'_3$ generates the trivial real factor in $\Lie{p}$.

Although we work with real representations of $K_{m,n}$, in order to apply Eschenburg--Wang's analysis it is convenient to complexify $V\oplus\Lie{p}$ and work with complex $\unitary{1}$--representations instead. We have
\[
V\otimes\C = \C_{2(n+m)}\oplus \C_{-2(n+m)}, \quad \R\otimes\C = \C_0, \quad \Lie{n}\otimes\C =\C_{2n}\oplus\C_{-2n}, \quad \Lie{n}'=\C_{-2m}\oplus\C_{2m}.
\]
Introduce the following basis of $(V\oplus\Lie{p})\otimes\C$ (identified with its dual) adapted to the decomposition of $(V\oplus\Lie{p})\otimes\C$ into $K_{m,n}$--representations:
\[
\begin{gathered}
dz_V = dt+idx = dt + i(n+m)(se_3-re'_3), \quad d\overline{z}_V = dt-idx = dt - i(n+m)(se_3-re'_3),\\ 
dz_\R = me_3 + ne'_3, \quad dz_{\Lie{n}} = e_1+ie_2, \quad d\overline{z}_{\Lie{n}} = e_1-ie_2,\quad dz_{\Lie{n}'} = e'_1+ie'_2, \quad d\overline{z}_{\Lie{n}'} = e'_1-ie'_2.
\end{gathered}
\]
We can then rewrite $\varphi$ as
\[
\begin{aligned}
\varphi = & \frac{i}{2}(pr-sa_3) dz_\R \wedge dz_{\Lie{n}}\wedge d\overline{z}_{\Lie{n}} + \frac{i}{2}(a_3r+sq) dz_\R \wedge dz_{\Lie{n}'}\wedge d\overline{z}_{\Lie{n}'}\\
&+\frac{i}{2}\frac{pn+ma_3}{(m+n)t} \Imag( dz_V )\wedge dz_{\Lie{n}}\wedge d\overline{z}_{\Lie{n}} + \frac{i}{2}\frac{a_3n-mq}{(m+n)t} \Imag( dz_V ) \wedge dz_{\Lie{n}'}\wedge d\overline{z}_{\Lie{n}'}\\
&+\frac{i}{2} \frac{\dot{a}_3}{(m+n)t}dz_\R\wedge dz_V\wedge d\overline{z}_V \\
&+ \frac{(r+s)}{2}(a_1-a_2)dz_\R\wedge \Imag\left( dz_{\Lie{n}}\wedge dz_{\Lie{n}'}\right) - \frac{(r-s)}{2}(a_1+a_2)dz_\R\wedge \Imag\left( dz_{\Lie{n}}\wedge d\overline{z}_{\Lie{n}'}\right)\\
&+\frac{t(\dot{a}_1+\dot{a}_2)-(a_1+a_2)}{4t} \Real\left(dz_V \wedge dz_{\Lie{n}}\wedge d\overline{z}_{\Lie{n}'}\right) \\
&+ \frac{t(\dot{a}_1+\dot{a}_2)+(a_1+a_2)}{4t} \Real\left(d\overline{z}_V \wedge dz_{\Lie{n}}\wedge d\overline{z}_{\Lie{n}'}\right)\\
&+\frac{t(\dot{a}_1-\dot{a}_2)+\tfrac{m-n}{m+n}(a_1-a_2)}{4t} \Real\left(dz_V \wedge dz_{\Lie{n}}\wedge dz_{\Lie{n}'}\right) \\
& +\frac{t(\dot{a}_1-\dot{a}_2)-\tfrac{m-n}{m+n}(a_1-a_2)}{4t} \Real\left(d\overline{z}_V \wedge dz_{\Lie{n}}\wedge dz_{\Lie{n}'}\right).
\end{aligned}
\]
Working with complex representations makes it very easy to understand which $\unitary{1}$--representation each component of $\varphi$ in this decomposition belongs to. We collect the weights of the corresponding real representations in the following table.
\begin{center}
\begin{tabular}{c|c||c|c}
$dz_\R \wedge dz_{\Lie{n}}\wedge d\overline{z}_{\Lie{n}}$  & $0$ & $dz_\R \wedge dz_{\Lie{n}'}\wedge d\overline{z}_{\Lie{n}'}$ & $0$ \\
 $\Imag( dz_V ) \wedge dz_{\Lie{n}}\wedge d\overline{z}_{\Lie{n}}$ & $2|m+n|$ &
$\Imag( dz_V ) \wedge dz_{\Lie{n}'}\wedge d\overline{z}_{\Lie{n}'}$ & $2|m+n|$ \\
$dz_\R\wedge dz_V\wedge d\overline{z}_V$ & $0$ & $dz_\R\wedge \Imag\left( dz_{\Lie{n}}\wedge dz_{\Lie{n}'}\right)$ & $2|m-n|$\\
$dz_\R\wedge \Imag\left( dz_{\Lie{n}}\wedge d\overline{z}_{\Lie{n}'}\right)$ & $2|m+n|$ & $\Real\left(dz_V \wedge dz_{\Lie{n}}\wedge d\overline{z}_{\Lie{n}'}\right)$ & $4|m+n|$ \\
$\Real\left(d\overline{z}_V \wedge dz_{\Lie{n}}\wedge d\overline{z}_{\Lie{n}'}\right)$ & $0$ & $\Real\left(dz_V \wedge dz_{\Lie{n}}\wedge dz_{\Lie{n}'}\right)$ & $4|n|$ \\
 $\Real\left(d\overline{z}_V \wedge dz_{\Lie{n}}\wedge dz_{\Lie{n}'}\right)$ & $4|m|$ & $$ & $$ \\
\end{tabular}
\end{center}

The principal orbits have non-trivial stabiliser $K_0=\Z_{2|m+n|}$ and all terms that are not $K_0$--invariant must vanish. Note that $\R^2_{2|m-n|}$, $\R^2_{4|m|}$ and $\R^2_{4|n|}$ are trivial $K_0$--representations if and only if there exists $d\in\Z$ such that $(d+1)m+(d-1)n=0$. If this is not the case we must have $a_1= a_2$, as we know already from Proposition \ref{prop:Invariant:Half:Flat}.

From the table above we read off the necessary and sufficient conditions for $\varphi$ to extend smoothly along the principal orbit:
\begin{enumerate}[leftmargin=*]
\item $a_3$ is even and there exists $r_0\in\R$ such that
\[
p=-m^2 r_0^3, \qquad q=n^2 r_0^3, \qquad a_3 (0)=mn r_0^3;
\]
\item $a_1 + a_2$ is odd;
\item if there exists $d\in\Z$ such that $(d+1)m+(d-1)n=0$ then the Taylor series of $a_1 - a_2$ only involves monomials of the form $t^{|d|+2l}$, $l \geq 0$ (in fact $l\geq 1$ if $d\neq 0$); otherwise $a_1=a_2$.
\end{enumerate}
Indeed, note that if $(d+1)m+(d-1)n=0$ then $m-n=-d(m+n)$, $2m = -(d-1)(m+n)$ and $2n=(d+1)(m+n)$. 

Finally, we determine constraints imposed by the requirement that $\varphi$ defines a {\gtstr} for small $t$. Note that the conditions (i), (ii) and (iii) above allow us to write
\[
a_1 = \alpha + \beta t +\gamma t^2 + O(t^3), \qquad a_2 = -\alpha + \beta t -\gamma t^2 + O(t^3), \qquad a_3 = mn r_0^3 + \delta t^2 + O(t^4),
\]
with $\alpha=0=\gamma$ unless $m=n=\pm 1$ (\ie $d=0$). Thus $\beta, \delta>0$ guarantee that $\dot{a}_i>0$ for $i=1,2,3$. Moreover, we calculate that the first three coefficients of the Taylor series of $\Lambda (a_1, a_2, a_3)$ at $t=0$ are
\[
4mnr_0^6 (m-n)^2\alpha^2, \qquad 0, \qquad - 4r_0^6\beta^2 mn(m+n)^2 + 16\alpha^2\beta^2.
\]
Note that the first coefficient, \ie the value of $\Lambda (a_1, a_2, a_3)$ at $t=0$, always vanishes, in agreement with the requirement $2\dot{a}_1\dot{a}_2\dot{a}_3=\sqrt{-\Lambda(a_1,a_2,a_3)}$. Hence $\Lambda (a_1, a_2, a_3)<0$ for small $t>0$ if and only if $mn>0$ and $\alpha^2<r_0^6$ when $n=m=\pm 1$. Note also that once we know that $m$ and $n$ have the same sign, then the only way that $(d+1)m + (d-1)n=0$ for some $d\in\Z$ is if $d=0$ and $m=n$ (and therefore $m=n=\pm 1$ since we assume that $\gcd (m,n)=1$).

\subsection{The singular IVP}

We now aim to construct and parametrise solutions to the fundamental ODE system \eqref{eq:Fundamental:ODE} satisfying the smoothness conditions of Proposition \ref{prop:Smooth:extension:singular:orbit}. The main tool is the following existence result for a special type of singular initial value problems, \cf 
\citelist{\cite{Malgrange}*{Theorem 7.1} \cite{Eschenburg:Wang}*{\S 5} \cite{Ferus:Karcher}*{\S 4}}.

\begin{theorem}
\label{thm:Singular:IVP}
Consider the singular initial value problem
\begin{equation}\label{eq:Singular:IVP}
\dot{y}=\frac{1}{t}M_{-1}(y)+M(t,y), \qquad y(0)=y_0,
\end{equation}
where $y$ takes values in $\R^k$, $M_{-1}\co \R^k\ra \R^k$ is a smooth function of $y$ in a neighbourhood of $y_0$ and $M\co\R\times\R^k\ra\R^k$ is smooth in $t,y$ in a neighbourhood of $(0,y_0)$. Assume that
\begin{enumerate}[leftmargin=*]
\item $M_{-1}(y_0)=0$;
\item $h\text{Id}-d_{y_0}M_{-1}$ is invertible for all $h \in \N$, $h \geq 1$.
\end{enumerate}
Then there exists a unique solution $y(t)$ of \eqref{eq:Singular:IVP}. Furthermore $y$ depends continuously on $y_0$ satisfying \tu{(i)} and \tu{(ii)}.
\end{theorem}

The main results of the section are summarised in the following proposition. In the statement $t$ is the arc-length parameter along a geodesic meeting all principal orbits orthogonally.

\begin{prop}
\label{prop:Solutions:Singular:Orbit}
Let $M$ be one of the cohomogeneity one $7$-manifolds considered in Proposition \ref{prop:Smooth:extension:singular:orbit}, labelled by the stabiliser $K$ of the singular orbit. Consider $\sunitary{2}\times\sunitary{2}$--invariant {\gtstr}s on $M$ of the form
\[
\varphi = p\, e_1\wedge e_2 \wedge e_3 + q\, e'_1\wedge e'_2 \wedge e'_3 + d\left( a_1\, e_1\wedge e'_1 + a_2\, e_2\wedge e'_2 + a_3\, e_3 \wedge e'_3\right).
\]
\begin{enumerate}[leftmargin=*]
\item If $K=\triangle\sunitary{2}$ then there exists a $3$-parameter family of torsion-free \gtstr s $\varphi$ defined in a neighbourhood of the singular orbit. The family is parametrised by $r_0>0$ and $\alpha_1, \alpha_2,\alpha_3\in\R$ such that
\[
64 r_0\left( \alpha_1+\alpha_2+\alpha_3\right)=1, \qquad p=-q = r_0^3, \qquad a_i (t) = r_0^3+ \tfrac{1}{4}r_0 t^2 + \alpha_i t^4 + O(t^6).
\] 
\item If $K=\{ 1\}\times\sunitary{2}$ then there exists a $3$-parameter family of torsion-free \gtstr s $\varphi$ defined in a neighbourhood of the singular orbit. The family is parametrised by $r_0>0$
and $\alpha_1, \alpha_2,\alpha_3\in\R$ such that
\[
\alpha_1\alpha_2\alpha_3=1, \qquad p=-r_0^3,\qquad q=0, \qquad a_i (t) = \tfrac{1}{4}r_0\alpha_i t^2 + O(t^4).
\] 
\item If $K=K_{1,1}$ then there exists a $3$-parameter family of torsion-free \gtstr s $\varphi$ defined in a neighbourhood of the singular orbit. The family is parametrised by $r_0,\alpha,\beta\in\R$ with $\beta>0$ via $p=-q=-r_0^3$ and
\[
a_1 = r_0^3\alpha + r_0^2\beta t + O(t^2), \qquad a_2 = -r_0^3\alpha +  r_0^2\beta t + O(t^2), \qquad a_3 = r_0^3 + O(t^2).
\]
\item If $K=K_{m,n}$ for coprime integers with $mn>1$ then there exists a $2$-parameter family of torsion-free \gtstr s $\varphi$ defined in a neighbourhood of the singular orbit. The family is parametrised by $r_0\in\R$ and $\beta>0$ via $p=-m^2 r_0^3$, $q=n^2r_0^3$ and
\[
a_1 = a_2 = r_0^2\beta t + O(t^3), \qquad a_3 = mn r_0^3 + O(t^2).
\]
\end{enumerate}
\end{prop}
In the rest of the section we show how to apply Theorem \ref{thm:Singular:IVP} to prove the four cases of the proposition.

\subsubsection*{The case $K=\triangle\sunitary{2}$}
Fix $r_0>0$. According to Proposition \ref{prop:Smooth:extension:singular:orbit} (i) we must have $p=-q=r_0^3$ and
\[
x_i = \dot{a}_j\dot{a}_k = \tfrac{r_0^2}{4}t^2+ 2r_0 (\alpha_j + \alpha_k)t^4 + O(t^6), \qquad y_i = a_i = r_0^3 + \tfrac{r_0}{4}t^2+\alpha_i t^4+O(t^6).
\]
The constraint $H(x,y)=0$ is already satisfied up to fourth order at $t=0$. The vanishing of the fifth $t$--derivative of $H(x,y)$ at $t=0$ imposes the further constraint
\[
64r_0(\alpha_1 + \alpha_2 + \alpha_3)=1.
\]
Fix any triple $(\alpha_1,\alpha_2,\alpha_3)$ satisfying this constraint and write
\[
x_i = \tfrac{r_0^2}{4} t^2 + t^4 X_i,\qquad y_i = r_0^3 + \tfrac{r_0}{4}t^2 + t^4 Y_i,
\]
where $X_i(0)=2r_0 (\alpha_j + \alpha_k)$ and $Y_i(0)=\alpha_i$. Writing $O(t^k)$ for any (convergent) series $\sum _{i\geq k}{c_i t^i}$ with coefficients $c_i$ depending on $X_1, \dots, Y_3, r_0$, we find that $(X,Y)$ satisfies an ODE system
\[
\begin{aligned}
\dot{X}_i &= t^{-1}\left( -4 X_i +r_0 (-5Y_i + 3Y_j + 3Y_k)+\tfrac{5}{64}\right) + O(1),\\
\dot{Y}_i &= t^{-1}\left( \tfrac{1}{r_0}(-X_i + X_j + X_k) -4 Y_i\right) + O(1)
\end{aligned}
\]
with the same structure as the one considered in Theorem \ref{thm:Singular:IVP}. The condition $64 r_0 (\alpha_1 + \alpha_2 + \alpha_3)=1$ guarantees that the initial condition $y_0$ for $(X,Y)$ satisfies $M_{-1}(y_0)=0$. Moreover, we have
\[
d_{y_0} M_{-1}=\left( \begin{array}{cccccc}
-4 & 0 & 0 & -5r_0 & 3r_0  & 3r_0\\
0 & -4 & 0 & 3r_0 & -5r_0 & 3r_0 \\
0 & 0 & -4 & 3r_0 & 3r_0  & -5r_0\\
-\frac{1}{r_0} & \frac{1}{r_0} & \frac{1}{r_0} & -4 & 0 & 0\\
\frac{1}{r_0} & -\frac{1}{r_0} & \frac{1}{r_0} & 0 & -4 & 0\\
\frac{1}{r_0} & \frac{1}{r_0} & -\frac{1}{r_0} & 0 & 0 & -4\\
\end{array}\right)
\]
and therefore for all $h\geq 1$
\[
\text{det}\left( h\text{Id}-d_{y_0} M_{-1}\right) = h^2(h+3)(h+5)(h+8)^2>0.
\]

\subsubsection*{The case $K=\{ 1\} \times\sunitary{2}$}
Fix $r_0>0$. According to Proposition \ref{prop:Smooth:extension:singular:orbit} (ii) we must have $p=-r_0^3$, $q=0$ and
\[
x_i = \tfrac{r_0^2}{4}\alpha_j \alpha_k t^2+ O(t^4), \qquad y_i = \tfrac{r_0}{4}\alpha_i t^2+O(t^4)
\]
for $\alpha_1,\alpha_2,\alpha_3\in\R$ satisfying $\alpha_1\alpha_2\alpha_3=1$.

The pair $(X,Y)$ defined by $x_i = t^2 X_i$ and $y_i = t^2 Y_i$ satisfies an ODE system of the form of Theorem \ref{thm:Singular:IVP}:
\[
\dot{X}_i = t^{-1}\left( -2 X_i +\sqrt{\frac{r_0^3 Y_j Y_k}{Y_i}}\right) + O(1), \qquad \dot{Y}_i= t^{-1}\left( -2 Y_i +\sqrt{\frac{ X_j X_k}{X_i}}\right) + O(1).
\]
The initial condition $y_0=\left( \tfrac{r_0^2}{4}\alpha_j \alpha_k, \tfrac{r_0}{4}\alpha_i\right)$ for $(X,Y)$ satisfies $M_{-1}(y_0)=0$ precisely when $\alpha_1\alpha_2\alpha_3=1$. Moreover, for any such triple $(\alpha_1,\alpha_2,\alpha_3)$ we have
\[
d_{y_0} M_{-1}=\left( \begin{array}{cccccc}
-2 & 0 & 0 & -\alpha_2^2\alpha_3^2r_0 & -\alpha_3 r_0  & -\alpha_2 r_0\\
0 & -2 & 0 & -\alpha_3 r_0 & -\alpha_3^2\alpha_1^2r_0 & -\alpha_1 r_0 \\
0 & 0 & -2 & -\alpha_2 r_0 & -\alpha_1 r_0  & -\alpha_1^2\alpha_2^2r_0\\
-\frac{\alpha_1^2}{r_0} & -\frac{\alpha_1\alpha_2}{r_0} & -\frac{\alpha_3\alpha_1}{r_0} & -2 & 0 & 0\\
-\frac{\alpha_1\alpha_2}{r_0} & -\frac{\alpha_2^2}{r_0} & -\frac{\alpha_2\alpha_3}{r_0} & 0 & -2 & 0\\
-\frac{\alpha_3\alpha_1}{r_0} & -\frac{\alpha_2\alpha_3}{r_0} & -\frac{\alpha_3^2}{r_0} & 0 & 0 & -2\\
\end{array}\right)
\]
and therefore for all $h\geq 1$
\[
\text{det}\left( h\text{Id}-d_{y_0} M_{-1}\right) = h^2(h+4)^2(h+1)(h+3)>0.
\]

\begin{remark}\label{rmk:Solutions:Singular:Orbit:D7}
For future use we record higher-order expansions of $a_1, a_2, a_3$ in the special case where there is an additional $\unitary{1}$--symmetry:  
\[
a_1(t)=a_2(t) = \tfrac{1}{4}r_0\alpha_1 t^2 +\frac{8-5\alpha_3^3}{576\alpha_1 r_0}t^4 + O(t^6), \quad a_3(t) = \tfrac{1}{4}r_0\alpha_3 t^2 -\frac{(4-7\alpha_3^3)\alpha_3}{576\alpha_1^2 r_0}t^4 + O(t^6).
\]
\end{remark}

\subsubsection*{The case $K=K_{1,1}$}
Fix $r_0\in\R$. According to Proposition \ref{prop:Smooth:extension:singular:orbit} (iii) we must have $p=-r_0^3$, $q=r_0^3$ and
\[
y_1 = r_0^3\alpha + r_0^2\beta t + O(t^2), \qquad y_2 = -r_0^3\alpha +  r_0^2\beta t + O(t^2), \qquad y_3 = r_0^3 + O(t^2),
\]
for some $\alpha\in\R$ and $\beta>0$ (in particular we can assume that $y_1+y_2>0$ for small $t>0$; we will use this condition freely below to simplify square roots).

We now set
\[
\begin{aligned}
&x_1 =tX_1, \qquad &x_2 &= tX_2, \qquad &x_3 &= r_0^4\beta^2 + t^2 X_3,\\
&y_1 = r_0^3\alpha + tY_1, \qquad &y_2 &= -r_0^3\alpha + tY_2,\qquad &y_3 &= r_0^3 + t^2 Y_3.
\end{aligned}
\]

Then $(X,Y)$ satisfies an ODE system of the form of Theorem \ref{thm:Singular:IVP}:
\[
\begin{aligned}
\dot{X}_1 &=\frac{1}{t}\left( 2\epsilon r_0^3\sqrt{1-\alpha^2}-X_1\right) +O(1),\\
\dot{X}_2 &=\frac{1}{t}\left( 2\epsilon r_0^3\sqrt{1-\alpha^2}-X_2\right) +O(1),\\
\dot{X}_3 &=\frac{1}{t}\left( \frac{(Y_1+Y_2)^2-2Y_3r_0^3(1-\alpha^2)}{(Y_1+Y_2)\epsilon\sqrt{1-\alpha^2}}-2X_3\right) + O(1),\\
\dot{Y_1} &=\frac{1}{t}\left( \frac{r_0^2\beta X_2}{\sqrt{X_1 X_2}} - Y_1 \right) +O(1),\\
\dot{Y_2} &=\frac{1}{t}\left( \frac{r_0^2\beta X_1}{\sqrt{X_1 X_2}} - Y_2 \right) +O(1),\\
\dot{Y_3} &=\frac{1}{t}\left( \frac{X_1 X_2}{r_0^2\beta\sqrt{ X_1 X_2}} -2 Y_3 \right) +O(1),
\end{aligned}
\]
where $\epsilon$ is the sign of $r_0$.

There is a unique solution
\[
y_0=\left( 2\epsilon r_0^3\sqrt{1-\alpha^2}, 2\epsilon r_0^3\sqrt{1-\alpha^2}, \frac{\epsilon \beta r_0^2}{\sqrt{1-\alpha^2}}-\frac{r_0^2(1-\alpha^2)}{2\beta^2},r_0^2\beta, r_0^2\beta, \frac{\epsilon r_0\sqrt{1-\alpha^2}}{\beta}\right).
\]
to the equation $M_{-1}(y_0)=0$. Moreover, $d_{y_0} M_{-1}$ is given by
\[
\left( \begin{array}{cccccc}
-1 & 0 & 0 & 0 & 0 & 0\\
0 & -1 & 0 & 0 & 0 & 0\\
0 & 0 & -2 & \frac{2\epsilon\beta^3+(1-\alpha^2)\sqrt{1-\alpha^2}}{2\beta^3\sqrt{1-\alpha^2}} & \frac{2\epsilon\beta^3+(1-\alpha^2)\sqrt{1-\alpha^2}}{2\beta^3\sqrt{1-\alpha^2}} & -\frac{\epsilon r_0\sqrt{1-\alpha^2}}{\beta}\\
-\frac{\epsilon \beta}{4r_0\sqrt{1-\alpha^2}} & \frac{\epsilon \beta}{4r_0\sqrt{1-\alpha^2}} & 0 & -1 & 0 & 0\\
\frac{\epsilon \beta}{4r_0\sqrt{1-\alpha^2}} & -\frac{\epsilon \beta}{4r_0\sqrt{1-\alpha^2}} & 0 & 0 & -1 & 0\\
\frac{1}{2\beta r_0^2} & \frac{1}{2\beta r_0^2} & 0 & 0 & 0 & -2\\
\end{array}\right).
\]
and
\[
\text{det}\left(h\text{Id}-d_{y_0} M_{-1}\right) = (h+1)^4(h+2)^2>0
\]
for all $h\geq 1$.

\subsubsection*{The case $K=K_{m,n}$, $mn>1$}
This case is very similar to the previous one. Fix $r_0\in\R$. According to Proposition \ref{prop:Smooth:extension:singular:orbit} (iii) we must have $p=-m^2 r_0^3$, $q=n^2 r_0^3$ and
\[
y_1 = y_2 = r_0^2\beta t + O(t^2), \qquad y_3 = mn r_0^3 + O(t^2),
\]
for some $\beta>0$. Using $y_1=y_2$ and $x_1=x_2$, we now set
\[
x_1=tX_1,\qquad x_3 = r_0^4\beta^2 + t^2 X_3, \qquad y_1 = tY_1, \qquad y_3 = mn r_0^3 + t^2 Y_3.
\]

Then $(X_1,X_3,Y_1,Y_3)$ satisfies an ODE system of the form of Theorem \ref{thm:Singular:IVP}:
\[
\begin{aligned}
\dot{X}_1 &=\frac{1}{t}\left( \sqrt{mn}(m+n)\epsilon r_0^3-X_1\right) +O(1),\\
\dot{X}_3 &=\frac{1}{t}\left( \frac{(m+n)^2Y_1^2-2m^2 n^2 Y_3r_0^3}{\sqrt{mn}(m+n)\epsilon Y_1}-2X_3\right) + O(1),\\
\dot{Y_1} &=\frac{1}{t}\left( r_0^2\beta X_1 - Y_1 \right) +O(1),\\
\dot{Y_3} &=\frac{1}{t}\left( \frac{X_1}{r_0^2\beta} -2 Y_3 \right) +O(1),
\end{aligned}
\]
where $\epsilon$ is the sign of $r_0$.

There is a unique solution
\[
y_0=\left( \epsilon \sqrt{mn}(m+n)r_0^3, \frac{\epsilon (m+n)\beta r_0^2}{2\sqrt{mn}}-\frac{r_0^2m^2n^2}{2\beta^2},r_0^2\beta, \frac{\epsilon \sqrt{mn}(m+n)r_0}{2\beta}\right).
\]
to the equation $M_{-1}(y_0)=0$. Moreover,
\[
d_{y_0} M_{-1}=\left( \begin{array}{cccc}
-1 & 0 & 0 & 0\\
0 & -2 & \frac{\epsilon(m+n)\beta^3+m^2n^2\sqrt{mn}}{\sqrt{mn}\beta^3} & -\frac{2\epsilon r_0 mn \sqrt{mn}}{(m+n)\beta}\\
0 & 0 & -1 & 0\\
\frac{1}{2\beta r_0^2} & 0 & 0 & -2\\
\end{array}\right).
\]
and
\[
\text{det}\left(h\text{Id}-d_{y_0} M_{-1}\right) = (h+1)^2(h+2)^2>0
\]
for all $h\geq 1$.

\begin{remark}\label{rmk:Solutions:Singular:Orbit:C7}
For future use note that the proof yields
\[
y_1=r_0^2 \beta t + O(t^3), \qquad y_2=mnr_0^3 + \frac{\sqrt{mn}(m+n)|r_0|}{2\beta}t^2 + O(t^4)
\]
as $t\ra 0$.
\end{remark}

\begin{remark*}
A more geometric interpretation of the parameters in Proposition \ref{prop:Solutions:Singular:Orbit} is given by studying the leading order behaviour of the induced metric \eqref{eq:Metric} as $t\ra 0$. For example, in the case of Proposition \ref{prop:Solutions:Singular:Orbit} (ii) with $\alpha_1=\alpha_2$ (so that we have the enhanced $\unitary{1}$--symmetry) we have 
\[
g_t \approx r_0^2\alpha_1^2 (e_1\otimes e_1 + e_2\otimes e_2) + r_0^2 \alpha_3^2\, e_3\otimes e_3 + \tfrac{1}{4}t^2 \left( e'_1\otimes e'_1 + e'_2\otimes e'_2 + e'_3\otimes e'_3\right).
\]
The metric on the singular orbit is a Berger metric on the $3$-sphere. If we fix $r_0=1$ by scaling, the parameters $\alpha_1=\alpha_2$ and $\alpha_3$ determine, respectively, the size of the base $S^2$ and of the Hopf fibres. Similarly, in the case of Proposition \ref{prop:Solutions:Singular:Orbit} (iv) we calculate
\begin{multline*}
g_t \approx \frac{r_0^2\beta}{\sqrt{mn}} \left( m\left( e_1\otimes e_1 + e_2\otimes e_2 \right) +  \left( e'_1\otimes e'_1 + e'_2\otimes e'_2\right) \right) +\frac{mnr_0^2}{\beta^2} (m e_3 + ne'_3)^2  + t^2 (m+n)^2 (se_3-re'_3)^2
\end{multline*}
as $t\ra 0$. Fixing $r_0=1$ by scaling, we see that the metric on the singular orbit is a squashed metric on the principal circle bundle $S^2\times S^3\ra S^2 \times S^2$: the metric on the base $S^2\times S^2$ is a product of the round metrics with ratio $\frac{m}{n}$ between the areas of the two factors; the parameter $\beta$ determines the respective sizes of the base and the circle fibres.
\end{remark*}

\begin{remark*}
As an aside, we note that Alekseevsky--Dotti--Ferraris have classified invariant Einstein metrics on $\sunitary{2} \times \sunitary{2}/K_{m,n}$ \cite{Alekseevsky:Dotti:Ferraris}*{Theorem 4.1}. When $mn>1$ there exists a unique invariant Einstein metric, which coincides with the Einstein metric obtained by Wang--Ziller on any circle bundle over a product of K\"ahler--Einstein manifolds \cite{Wang:Ziller}*{Theorem 1.4}; when  $m=n=1$ there are two Einstein metrics, the product of round metrics and the Sasaki--Einstein metric. The restriction of $g_t$ to the singular orbit is never Einstein unless $m=n=1$ and $2\beta^6 =1$. 
\end{remark*}

\section{Conically singular and asymptotically conical ends}
\label{sec:cs:ac:ends}

In addition to the solutions of \eqref{eq:Fundamental:ODE} just constructed in Proposition \ref{prop:Solutions:Singular:Orbit} and which are defined in the neighbourhood of various classes of singular orbit we will need two further classes of local solutions to \eqref{eq:Fundamental:ODE}: (i) solutions with an isolated conical singularity modelled on the cone over the homogeneous nearly K\"ahler structure on $S^3\times S^3$ and (ii) solutions with an asymptotically conical end asymptotic to the same cone. Describing solutions with either type of end behaviour 
leads to a class of singular initial value problems not widely studied in the previous extensive
work on cohomogeneity one Einstein metrics. The closest work 
we are aware of is Dancer--Wang's work \cite{Dancer:Wang:Painleve} on 
Painlev\'e-type expansions of singular solutions to the cohomogeneity one 
Ricci-flat equations. In many cases their expansions construct continuous families of either AC or ALC Ricci-flat ends, subsets of which often have special or exceptional holonomy. However, because they consider only ``rational resonances'',
their approach often does not capture the full space of ends, \eg in \cite{Dancer:Wang:Painleve}*{Theorem 5.6} the only \gthol AC ends obtained arise from the 
original Bryant--Salamon AC metric.

In order to construct solutions with prescribed singular/asymptotic behaviour we will use the following extension of Theorem \ref{thm:Singular:IVP}. The theorem is proved in Chapter I of Volume 3 of Picard's treatise \cite{Picard} using the method of majorants, \cf in particular \cite{Picard}*{Chapter I, \S 13}.

\begin{theorem}\label{thm:Singular:IVP:Extended}
Consider the singular initial value problem
\begin{equation}\label{eq:Singular:IVP:Extended}
t\dot{y}=\Phi (y,t), \qquad y(0)=y_0,
\end{equation}
where $y$ takes values in $\R^k$ and $\Phi\co \R^k\times\R\ra \R^k$ is a real analytic function in a neighbourhood of $(y_0,0)$ with $\Phi(y_0,0)=0$. Assume also that $\partial_t\Phi (y_0,0)$ lies in the image of $\tu{Id}-d_{y_0}\Phi (\,\cdot\, ,0)$ and fix a preimage $y_1\in \R^k$. After possibly a change of basis, assume that $d_{y_0}\Phi (\,\cdot\, ,0)$ contains a diagonal block $\tu{diag}(\lambda_1,\dots,\lambda_m)$ in the upper-left corner. Furthermore assume that the eigenvalues $\lambda_1,\dots,\lambda_m$ satisfy:
\begin{enumerate}[leftmargin=*]
\item $\lambda_1,\dots, \lambda_m>0$;
\item for every $\triple{h}=(h_0,\dots, h_m)\in\Z^{m+1}_{\geq 0}$ with $|\triple{h}| = h_0+\dots +h_m\geq 2$ the matrix
\[
(\triple{h}\cdot\triple{\lambda})\, \tu{Id} - d_{y_0}\Phi (\,\cdot\, ,0)
\]
is invertible. Here $\triple{\lambda} = (1,\lambda_1,\dots, \lambda_m)$ and $\triple{h}\cdot\triple{\lambda} = \sum_{i=0}^m{h_i\lambda_i}$.
\end{enumerate}
Then for every $(u_1,\dots, u_m)\in\R^m$ there exists a unique solution $y(t)$ of \eqref{eq:Singular:IVP:Extended} given as a convergent generalised power series
\[
y(t) =y_0 + y_1 t+ \left( u_1 t^{\lambda_1}, \dots, u_m t^{\lambda_m},0,\dots,0\right)+\sum_{|\triple{h}|\geq 2}{y_{\triple{h}}\,t^{\triple{h}\cdot\triple{\lambda}}}.
\]
Furthermore, the solutions depend real analytically on $u_1,\dots, u_m$.
\end{theorem}

\begin{remark*}
It is clear that if $\Phi$ does not depend on $t$ then the solution $y$ has a generalised power series expansion in powers of $t^{\lambda_1}, \dots, t^{\lambda_m}$ only and we should take $h_0=0$ in condition (ii).
\end{remark*}

\begin{remark*}
The existence theorem can be extended to the case of smooth (rather than real analytic) $\Phi$ by truncation of a formal generalised power series solution to sufficiently high order and a contraction mapping argument.
\end{remark*}

We now use this existence result to construct $1$-parameter families of conically singular and asymptotically conical ends.

\begin{prop}\label{prop:CS:AC:ends}
Let $\tu{C}$ be the {\gtwo}--holonomy cone over the homogeneous nearly K\"ahler structure on $S^3\times S^3$ and set $\nu_0 = \frac{\sqrt{145}-7}{2}\approx 2.5$ and $\nu_\infty = \frac{\sqrt{145}+7}{2}\approx 9.5$.
\begin{enumerate}[leftmargin=*]
\item For every $c\in\R$ there exists a unique $\sunitary{2}\times\sunitary{2}\times\unitary{1}$--invariant torsion-free \gtstr
\[
\varphi = d\left( a\, (e_1\wedge e'_1 + e_2\wedge e'_2) + b\, e_3\wedge e'_3\right)
\]
defined on $(0,\epsilon)\times S^3\times S^3$ for some $\epsilon>0$, where the functions $t^{-3}a$ and $t^{-3}b$ admit convergent generalised power series expansions in powers of $t^{\nu_0}$ satisfying
\[
\tfrac{54}{\sqrt{3}}t^{-3}a(t) = 1+\tfrac{1}{2}ct^{\nu_0} + O( t^{2\nu_0}), \qquad \tfrac{54}{\sqrt{3}}t^{-3} b(t) = 1-ct^{\nu_0} +O( t^{2\nu_0}).
\]
In particular the associated metric $g_{\varphi}$ has a conical singularity as $t\ra 0$ asymptotic to the cone $\tu{C}$ with rate $\nu_0$.
\item Fix $p,q\in\R$. For every $c\in\R$ there exists a unique $\sunitary{2}\times\sunitary{2}\times\unitary{1}$--invariant torsion-free \gtstr
\[
\varphi = p\, e_1\wedge e_2\wedge e_3 + q\, e'_1\wedge e'_2\wedge e'_3 + d\left( a\, (e_1\wedge e'_1 + e_2\wedge e'_2) + b\, e_3\wedge e'_3\right)
\]
on $(T,\infty)\times S^3\times S^3$ for some $T>0$, where the functions $t^{-3}a$ and $t^{-3}b$ admit convergent generalised power series expansions in powers of $t^{-3}$ and $t^{-\nu_\infty}$ satisfying
\[
\tfrac{54}{\sqrt{3}}t^{-3}a = 1+O(t^{-3}), \quad \tfrac{54}{\sqrt{3}}t^{-3}b = 1+O(t^{-3}), \quad \tfrac{54}{\sqrt{3}}t^{-3}\left( b-a\right) = ct^{-\nu_\infty} + O(t^{-12}).
\]
In fact when $p=0=q$ then $t^{-3}a$ and $t^{-3}b$ admit convergent generalised power series expansions in powers of $t^{-\nu_\infty}$ only. In particular the associated metric $g_{\varphi}$ has a complete asymptotically conical end as $t\ra \infty$ asymptotic to the cone $\tu{C}$ with rate $-\nu_\infty$ if $p=0=q$ and rate $-3$ otherwise.
\end{enumerate}
\proof
First of all, note that, because of its invariance under $\sunitary{2}^3$ and scaling, the solution to the fundamental ODE system \eqref{eq:Fundamental:ODE} corresponding to the \gtwo--cone over the homogeneous nearly K\"ahler structure on $S^3\times S^3$ must satisfy $a=b=Ct^3$. It is then immediate to check that $C=\frac{\sqrt{3}}{54}$.

The proof of the existence of CS solutions is a straightforward application of Theorem \ref{thm:Singular:IVP:Extended}. We write $a=\frac{\sqrt{3}}{54}t^3(1+Y_1)$ and $b=\frac{\sqrt{3}}{54}t^3(1+Y_2)$. Define functions $X_1$ and $X_2$ by $\dot{a}\dot{b} = \frac{1}{108}t^4(1+X_1)$ and $\dot{a}^2 = \tfrac{1}{108}t^4(1+X_2)$. Then the $\unitary{1}$--enhanced ODE system \eqref{eq:Fundamental:ODE:U(1)} for $(\dot{a}\dot{b},\dot{a}^2,a,b)$ becomes an ODE system for $(X_1, X_2, Y_1, Y_2)$ of the form \eqref{eq:Singular:IVP:Extended}:
\[
\begin{aligned}
t\dot{X}_1 &=- 4X_1 + \frac{ 4\sqrt{3}(1+Y_1)(1+Y_2)^2}{\sqrt{4(1+Y_1)^2(1+Y_2)^2-(1+Y_2)^4}} -4,\\
t\dot{X}_2 &= - 4X_2 + \frac{ 4\sqrt{3}(1+Y_2) \left( 2(1+Y_1)^2-(1+Y_2)^2\right) }{\sqrt{4(1+Y_1)^2(1+Y_2)^2-(1+Y_2)^4}}-4,\\
t\dot{Y}_1 &= - 3Y_1 + \frac{ 3(1+X_1)(1+X_2)}{\sqrt{(1+X_1)^2(1+X_2)}}-3,\\
t\dot{Y}_2  &= - 3Y_2 + \frac{ 3(1+X_1)^2}{\sqrt{(1+X_1)^2(1+X_2)}}-3.
\end{aligned}
\]
The linearisation of $\Phi$ at $y_0=(0,0,0,0)$ is
\begin{equation}\label{eq:Linearisation:CS:IVP}
L=\left( \begin{array}{cccc}
-4 & 0 & -\tfrac{4}{3} & \tfrac{16}{3}\\
0 & -4 & \tfrac{32}{3} & -\tfrac{20}{3}\\
0 & \tfrac{3}{2} & -3 & 0\\
3 & -\tfrac{3}{2} & 0 & -3\\
\end{array}\right).
\end{equation}
$L$ has $4$ distinct eigenvalues $-1,-6,-\nu_\infty=-\nu_0 -7,\nu_0$. The corresponding eigenvectors are, respectively,
\[
(4,4,3,3), \quad (2,2,-1,-1), \quad \left( 3+\nu_0,-6-2\nu_0,-3,6\right), \quad \left( 4+\nu_0,-8-2\nu_0,3,-6\right).
\]
The proof of part (i) now follows immediately from Theorem \ref{thm:Singular:IVP:Extended}.

Constructing AC ends is more involved: Theorem \ref{thm:Singular:IVP:Extended} cannot be applied immediately because the nonresonance condition (ii) fails. As above we work with the system \eqref{eq:Fundamental:ODE:U(1)} for the $4$-tuple $(x_1=x_2=\dot{a}\dot{b}, x_3=\dot{a}^2, y_1=y_2=a,y_3=b)$ but it is now convenient to change variable $s=\tfrac{1}{t}$. Define functions $(X_1,X_2,Y_1,Y_2)$ by $s^3 y_1=\frac{\sqrt{3}}{54}(1+Y_1)$, $s^3 y_3=\frac{\sqrt{3}}{54}(1+Y_2)$, $s^4 x_1= \tfrac{1}{108}(1+X_1)$ and $s^4 x_3 = \tfrac{1}{108}(1+X_2)$. The $4$-tuple $(X_1,X_2,Y_1,Y_2)$ satisfies a system of the form $s\dot{y}=\Phi (y,s^3)$. $\Phi$ does not depend on $s^3$ if $p=0=q$. The linearisation of $\Phi(\,\cdot\, , 0)$ at $0$ is $-L$, with $L$ as in \eqref{eq:Linearisation:CS:IVP}.

Now, the presence of $1$ in the spectrum of $-L$ is explained by the fact that the original system \eqref{eq:Fundamental:ODE:U(1)} is $t$--invariant. By a translation $t\mapsto t+t_0$ we can therefore always reduce to the case when the eigenvalue $1$ is never excited. The eigenvalue $6$ should also not be excited. Indeed, we are interested in solutions of \eqref{eq:Fundamental:ODE:U(1)} satisfying the conservation law $H(x,y)=0$. Rewriting $H=H(s^3,X_1,X_2,Y_1,Y_2)$, a calculation shows that $H$ vanishes at leading order if and only if $3(2X_1+X_2) -4(2Y_1+Y_2)=0$.

The discussion above motivates us to look for a $1$-parameter family of AC solutions to the ODE system \eqref{eq:Fundamental:ODE:U(1)} given by convergent generalised power series in powers of $s^3, s^{\nu_\infty}$ and parametrised by the coefficient of $s^{\nu_\infty}$. A further change of variable $s\mapsto s^3$ justifies the fact that powers of $s^3$ and not $s$ should be considered.

Given that $\partial_s\Phi (0,0)=0$, the conditions to be satisfied in order to apply Theorem \ref{thm:Singular:IVP:Extended} are
\[
3 h_0 + \nu_\infty h_1 \neq 1,6,\nu_\infty, -\nu_0
\]
for every $h_0,h_1\in\Z_{\geq 0}$ with $h_0 + h_1\geq 2$. The nonresonance condition fails only when $(h_0,h_1)=(2,0)$. There is therefore a potential obstruction to solve the system: the coefficient $y_{2,0}$ of $s^6$ must satisfy an equation
\begin{equation}\label{eq:AC:end:obstruction:6}
(L+6)\, y_{2,0} = Q_{2,0}(y_0, y_{1,0}),
\end{equation}
where $Q_{2,0}$ is a real analytic function of the initial condition $y_0=0$ for $(X,Y)$ and the coefficient $y_{1,0}$ of $s^3$, which is uniquely determined by the equation. We know that $L+6$ has a one-dimensional kernel and cokernel. Only if $Q_{2,0}(y_0, y_{1,0})$ lies in the hyperplane $\tu{im}\, (L+6)$ can we solve \eqref{eq:AC:end:obstruction:6}. Assuming this is the case, we fix the choice of a solution $y_{2,0}$ to \eqref{eq:AC:end:obstruction:6} by imposing the vanishing of the coefficient of order $s^6$ of the Hamiltonian $H$. Since the linearisation of $H$ does not annihilate the eigenvector of $-L$ of eigenvalue $6$, this requirement fixes a unique choice for $y_{2,0}$. Once $y_{2,0}$ is uniquely determined, the iteration procedure to find a formal generalised power series solution to \eqref{eq:Fundamental:ODE:U(1)} can be continued without further obstructions and the majorisation argument in the proof of Theorem \ref{thm:Singular:IVP:Extended} still guarantees that the generalised power series converges.

The key observation now is that the potential obstruction to solve \eqref{eq:AC:end:obstruction:6} does in fact vanish. Instead of showing this by computation, we observe that for every $p,q\in\R$ there exists an AC solution of the ODE system \eqref{eq:Fundamental:ODE:U(1)} with an enhanced $\sunitary{2}$--symmetry, \ie with $a=b$. Using the conserved quantity $H=0$ we can describe such a solution by the curve $(x,y)\in\R^2$ defined by the equation $4x^3=3y^4-4(p-q)y^3-6pqy^2-p^2 q^2$. Recalling that $y=a=b$ and $x=\dot{a}^2$ we can rewrite this equation as the ODE 
\[
4\dot{a}^6 =3a^4-4(p-q)a^3-6pqa^2-p^2 q^2
\]
for the function $a$. Taking the sixth root of both sides of the equation ($\dot{a}>0$ with our conventions), changing variable $s=1/t$ and writing $a=\frac{\sqrt{3}}{54}s^{-3}(1+w)$, the equation can be written in the form $s\dot{w} = w + O(w^2 +s^3)$. Theorem \ref{thm:Singular:IVP:Extended} then guarantees that the solution admits a convergent power series expansion in $s^3=t^{-3}$, unique up to a time translation $t\mapsto t- t_0$. We conclude that all the coefficients $y_{h,0}$ in a formal generalised power series expansion of $y$ (in particular the coefficient $y_{2,0}$) are uniquely determined by the equation up to a time translation. 
\endproof
\end{prop}

\begin{remark*}
In order to relate the Proposition to the general deformation theory of CS and AC \gtmfd s developed by Karigiannis--Lotay in \cite{Karigiannis:Lotay}, one can check that $\sigma = e_1\wedge e'_1 + e_2 \wedge e'_2 - 2 e_3 \wedge e'_3$ is a coclosed primitive $(1,1)$--form with
\[
\triangle\sigma = (\nu_0 +3)(\nu_0+4)\sigma = (-\nu_\infty +3)(-\nu_\infty+4)\sigma = 36\sigma
\]
on $S^3\times S^3$ endowed with its homogeneous nearly K\"ahler structure. It follows that $d(t^{\nu_0+3}\sigma)$ and $d(t^{-\nu_\infty+3}\sigma)$ are closed and coclosed $3$-forms of type $27$ on the \gtwo--cone $\tu{C}$ and therefore infinitesimal deformations of $\tu{C}$ as a \gtmfd. On the other hand, differentiating our solutions with respect to the parameter $c$ in the construction yields solutions $d(t^{\nu_0+3}\sigma)$ and $d(t^{-\nu_\infty+3}\sigma)$ of the linearisation of \eqref{eq:Fundamental:ODE:U(1)} at the conical solution $a=b=\frac{\sqrt{3}}{54}t^3$.
\end{remark*}

\section{Existence of ALC metrics}
\label{sec:ALC}

In this section we obtain our first main (global) existence results, Theorems \ref{mthm:CS} and \ref{mthm:B7:D7} in the Introduction: we prove the existence of two $1$-parameter families of complete ALC \gtmetric s and the existence of an ALC \gtwo--space with an isolated conical singularity.
One of the two families of ALC \gtmetric s we obtain, named $\mathbb{B}_7$ in the physics literature, is already known to exist thanks to work of Bogoyavlenskaya \cite{Bogoyavlenskaya} (one member of this family is explicit and was found earlier by Brandhuber--Gomis--Gubser--Gukov \cite{BGGG}); the other family, named $\mathbb{D}_7$ in the physics literature, is currently only known numerically or in the collapsed limit of \cite{FHN:ALC:G2:from:AC:CY3}. The existence of the CS ALC space seems not to have been anticipated in the physics literature: as explained in the Introduction it can be used, together with the Bryant--Salamon AC metric on $\tu{S}^3\times\R^4$, to explain the existence of both the  $\mathbb{B}_7$ and $\mathbb{D}_7$ families of ALC manifolds by a gluing construction.   

\subsection{A criterion for forward completeness}

In the following proposition we relate the forward completeness of a
cohomogeneity one torsion-free {\gtstr} to the sign of the mean curvature $l$
of the principal orbits. One direction is very closely related to a result of
B\"ohm \cite{Bohm}*{Proposition~3.2}, while the other implication, which plays
a crucial role in our analysis, was suggested to us by Wilking.

\begin{prop}\label{prop:Mean:Curvature:Blow:Up}
Let $(M,\varphi)$ be a cohomogeneity one $\sunitary{2}\times\sunitary{2}$--invariant {\gtmfd} (not necessarily complete). Assume that $l(t_0)>0$ for some initial time $t_0$ corresponding to a principal orbit. Then the solution blows up in finite time if and only if there exists $t_\ast>t_0$ such that $l(t_\ast)=0$, \ie if and only if there exists a principal orbit that is a minimal hypersurface.
\proof
Since the $7$-dimensional metric $dt^2 + g_t$ induced by $\varphi$ is Ricci-flat, the mean curvature $l$ of the principal orbits satisfies $0=l'+|L|^2=l' + \frac{1}{6}l^2 + |\mathring{L}|^2$, where $\mathring{L}$ denotes the traceless part of the second fundamental form $L$ of the principal orbits, \cf for example \cite{Eschenburg:Wang}*{Proposition 2.1}.

Note that $M$ cannot contain a totally geodesic principal orbit $\mathcal{O}$. To see this, let $\psi$ denote the parallel spinor on $M$ and recall that using Clifford multiplication the covariant derivative of $\psi|_\mathcal{O}$ can be identified with the second fundamental form of $\mathcal{O}$. If $\mathcal{O}$ were totally geodesic, it would therefore carry a parallel spinor and hence a homogeneous Calabi--Yau (and therefore Ricci-flat) metric. This is impossible since homogeneous Ricci-flat metrics must be flat, while the principal orbits of $M$, which are finite quotients of $S^3\times S^3$, do not carry flat metrics. In particular, $l(t_\ast)=0$ implies $l(t)<0$ for $t>t_\ast$. Then comparison with the solution of $u'+\frac{1}{6}u^2=0$ shows that $l$ must blow-up in finite time.

Conversely, assume that the solution exists only on a finite interval $[t_0,T)$. By Lemma \ref{lem:Metric} we can regard Hitchin's flow as a \emph{first-order} ODE system for the metric $g$. Since it satisfies a first-order equation, the metric must degenerate as we approach the maximal existence time $T$ (otherwise we could use local existence to extend the solution beyond $t=T$). Hence as $t\ra T$ the norm of a Jacobi field $J$, \ie a vector field satisfying $J' = L J$, either converges to zero or infinity. In either case, we deduce that $\int_{t_0}^{T}{|L|} \geq \lim_{t\ra T} \log{|J|(t)}-\log{|J|(t_0)}=\infty$. Since $T<\infty$, by H\"older's inequality we also have $\int_{t_0}^T {|L|^2}=\infty$. Integration of $l'+|L|^2=0$ shows that $l(t)\ra -\infty$ as $t\ra T$. Since $l(t_0)>0$ we deduce there must exist $t_\ast\in (t_0,T)$ such that $l(t_\ast)=0$.
\endproof
\end{prop}

\begin{remark*}
If $M$ closes smoothly on a singular orbit or has an isolated conical singularity at $t=0$ then $\lim_{t\ra 0}{l(t)}=+\infty$ so the assumption about the positivity of the mean curvature at an initial time is certainly satisfied in these cases.
\end{remark*}

\begin{remark*}
Finite time blow-up when there exists a minimal principal orbit can be deduced as in the Proposition from any condition that would exclude the existence of a totally geodesic principal orbit. For example, in \cite{Bohm}*{Proposition 3.2} B\"ohm shows that in a complete cohomogeneity one Ricci-flat manifold that does not contain a line and is not flat principal orbits cannot be minimal.
\end{remark*}

\subsection{The \texorpdfstring{$\unitary{1}$--enhanced}{U(1) enhanced} symmetric system}

Let us now specialise to the setting where we assume an additional $\unitary{1}$--symmetry. In the rest of the paper we will give a detailed qualitative analysis of the single second-order equation \eqref{eq:Fundamental:ODE:Brandhuber:U(1)} arising from the Lagrangian formulation of the problem. Hence in the rest of the paper $\,\dot{\,}$ denotes differentiation with respect to an arbitrary parameter $s$. We reserve the freedom to change parametrisation (compatible with a fixed orientation) and will specify the choice of a parametrisation when additional properties are needed. We consider pairs of functions $(a,b)$ satisfying
\[
2F \left( \dot{a} \ddot{b} - \dot{b} \ddot{a} \right) = -\dot{a}\dot{b} \left( 2\dot{b} F_b - \dot{a} F_a \right),
\]
where $F=4a^2 (b-p)(b+q) - (b^2+pq)^2$ and $F_a, F_b$ denote its partial derivatives, and the constraints
\[
\dot{a},\dot{b}>0, \qquad F(a,b)>0.
\] 
Note that the latter condition forces $a$ to have a definite sign and $b-p, b+q$ to have the same definite sign.

The following formulas will play an important role in our analysis:
\begin{subequations}\label{eq:Identities:F}
\begin{equation}
2F_b - F_a = 8 (a-b)\left( a(2b+q-p)+b^2+pq\right),
\end{equation}
\begin{equation}
2bF_b - aF_a = 8(a-b)(a+b)(b^2+pq).
\end{equation}
\end{subequations}
Furthermore the formula \eqref{eq:Mean:Curvature} for the mean curvature $l$ of the principal orbits reads
\begin{equation}\label{eq:Mean:Curvature:U(1)}
l = \frac{\dot{a} F_a + \dot{b} F_b}{2F}.
\end{equation}

We also observe that in all cases we are interested in $p$ and $q$ satisfy $pq\leq 0$. By Theorem \ref{thm:Classification:U(1)} below this is no accident: every $\sunitary{2}\times\sunitary{2}\times\unitary{1}$--invariant torsion-free {\gtstr} closing smoothly on a singular orbit must satisfy this condition.

\begin{lemma}\label{lem:ALC:chamber}
Assume that $(a,b)$ is a solution to \eqref{eq:Fundamental:ODE:Brandhuber:U(1)} with $pq\leq 0$. Then the conditions
\begin{equation}\label{eq:ALC:Chamber}
\dot{a}>\dot{b}, \qquad a>b>\max{\left(p,-q,\sqrt{-pq}\right)}\geq 0
\end{equation}
are preserved as long as the solution exists with $\dot{a},\dot{b},F>0$.
\proof
Since $\dot{a},\dot{b}>0$ the inequalities $a>0$ and $b>\max{(p,-q,\sqrt{-pq})}\geq 0$ are certainly preserved. If $\dot{a}>\dot{b}$ then also $a>b$ is preserved. We therefore must show that the condition $\dot{a}-\dot{b}>0$ is preserved as long as the solution satisfies the open constraints to define a \gtstr. Now, at a point where $\dot{a}=\dot{b}$ \eqref{eq:Fundamental:ODE:Brandhuber:U(1)} yields
\[
2F \left( \ddot{a}-\ddot{b} \right)  = \dot{a}^2 (2F_b - F_a).
\]  
Hence by (\ref{eq:Identities:F}a) $\ddot{a}-\ddot{b}>0$ as long as $a>b>\max{(p,-q,\sqrt{-pq})}\geq 0$.
\endproof
\end{lemma}

\begin{prop}\label{eq:ALC:Forward:Completeness}
If the conditions \eqref{eq:ALC:Chamber} are satisfied then the solution $(a,b)$ is forward complete. In particular, $a\ra \infty$ as we approach the complete end.
\proof
By \eqref{eq:Mean:Curvature:U(1)} the mean curvature of the principal orbit has the same sign as $2\left( \dot{a}F_a+\dot{b}F_b\right)= \dot{b}(2F_b-F_a) + (2\dot{a}+\dot{b})F_a>0$ if the conditions \eqref{eq:ALC:Chamber} are satisfied (note that $F_a = 8a (b-p)(b+q)$ is strictly positive). By Proposition \ref{prop:Mean:Curvature:Blow:Up} the solution cannot blow up in finite time.

In order to prove that $a\ra \infty$, parametrise with respect to the arc-length parameter $t$ along a geodesic meeting all principal orbits orthogonally. Note that since $\dot{a}>\dot{b}$ we have $2\dot{a}^3 > 2\dot{a}^2 \dot{b}$. Up to a positive multiplicative constant, the right-hand side is the orbital volume function $\tu{Vol}(t)$. Since the mean curvature $l$ is positive, $\tu{Vol}(t) \geq \tu{Vol}(t_0)>0$ for all $t\geq t_0$. Hence $\dot{a}$ is bounded below by a positive constant. Since $t$ is unbounded along the complete end, so is $a$.
\endproof
\end{prop}

\begin{remark*}
In the course of proving ALC asymptotics in Proposition \ref{prop:ALC:growth:a:b} we will show that the assumptions \eqref{eq:ALC:Chamber} also force $b\ra \infty$ along the complete end.
\end{remark*}

\subsection{ALC asymptotics}

We are now going to show that under an additional assumption the complete ends of Proposition \ref{eq:ALC:Forward:Completeness} are ALC. First note that the conical Calabi--Yau structure $(\omega_\tu{C},\Omega_C)$ on the conifold $\tu{C}=\tu{C}(\Sigma)$, $\Sigma = \sunitary{2}\times\sunitary{2}/K_{1,-1}$, is given by
\[
\Real\Omega_\tu{C} = d\left(  \tfrac{1}{18}t^3 (e_1\wedge e'_1 + e_2\wedge e'_2) \right),\qquad \omega_\tu{C} = -d\left( \tfrac{1}{6}t^2(e_3-e'_3)\right).
\]
Moreover, $\theta = \frac{1}{2}(e_3+e'_3)$ is the (unique up to gauge transformations) Hermitian Yang--Mills connection on the circle bundle $\R_+ \times \sunitary{2}\times\sunitary{2}\ra \tu{C}$, \ie $d\theta\wedge\omega_\tu{C}^2=0=d\theta\wedge\Omega_\tu{C}$. Fix $\ell>0$ and consider the closed {\gtstr} $\varphi_\infty$ on the total space of this circle bundle
\begin{equation}\label{eq:Model:Coho1:ALC}
\varphi_\infty = d\left( \tfrac{1}{18}t^3 (e_1\wedge e'_1 + e_2\wedge e'_2) +\tfrac{1}{6}\ell t^2e_3\wedge e'_3\right).
\end{equation}
Since $\theta$ is Hermitian Yang--Mills, $d\ast_{\varphi_\infty}\varphi_\infty = O(t^{-2})$. Moreover, up to terms that decay as $t^{-1}$ the metric induced by $\varphi_\infty$ is
\[
g_{\varphi_\infty} = dt^2 + t^2 g_{\tu{se}} + \ell^2 \theta^2 +O(t^{-1}),
\] 
where $g_{\tu{se}}$ is the (pull-back to the total space of the circle bundle of the) Sasaki--Einstein metric on $\Sigma=\sunitary{2}\times\sunitary{2}/\triangle\unitary{1}$. (In particular here $t$ is the arc-length parameter along a geodesic meeting all principal orbits orthogonally.) We therefore regard $\varphi_\infty$ as the asymptotic model for an $\sunitary{2}\times\sunitary{2}$--invariant torsion-free {\gtstr} inducing an ALC metric. 

\begin{lemma}\label{lem:ALC:growth:a:b}
Let 
\[
\varphi = p\, e_1\wedge e_2\wedge e_3 + q\, e'_1 \wedge e'_2 \wedge e'_3 + d\left( a \, (e_1 \wedge e'_1 + e_2 \wedge e'_2) + b\, e_3 \wedge e'_3\right)
\]
be an $\sunitary{2}\times\sunitary{2}\times\unitary{1}$--invariant torsion-free {\gtstr} defined on $(t_0,\infty)\times \sunitary{2}\times\sunitary{2}$. Assume that $a,b$ are positive increasing functions with $\lim_{t\ra\infty}a=\infty$, $\lim_{t\ra\infty} \frac{a^2}{b^3}=\frac{2}{3\ell^3}$ for some $\ell>0$ and $\lim_{t\ra\infty}{\frac{a}{b}\frac{db}{da}}=\frac{2}{3}$. Then the functions $\tilde{a}(t) = 18 t^{-3}a(t)-1$ and $\tilde{b}(t)=6\ell^{-1}t^{-2}b(t)-1$ satisfy $\tilde{a}^{(k)}(t)=O(t^{-k-1})$ and $\tilde{b}^{(k)}(t) = O(t^{-k-1})$ as $t\ra\infty$ for all $k\geq 0$. In particular, $\varphi$ converges to the closed {\gtstr} $\varphi_\infty$ of \eqref{eq:Model:Coho1:ALC} in $C^\infty$ as $t\ra\infty$. 
\proof
Since $a$ is a positive increasing function of $t$ we can introduce a parameter $s$ such that $a=\tfrac{1}{18}s^3$. Then by our assumptions on $a$ and $b$, $b\approx \frac{1}{6}\ell s^2$ in $C^0$. Hence, $\frac{2}{\ell s}\frac{db}{ds}\approx\frac{a}{b}\frac{db}{da}\approx\frac{2}{3}$ and therefore $b-\frac{1}{6}\ell s^2$ converges to zero in $C^1$. Now, the arc-length parameter $t$ is related to $s$ by
\[
2\left( \frac{ds}{dt}\right)^3 \left( \frac{da}{ds}\right)^2\frac{db}{ds}=\sqrt{F(a,b)}.
\]
As $s\ra \infty$ we therefore have $\frac{ds}{dt}\approx 1$ and by integration $s\approx t$ in $C^1$. Thus $a(t)-\tfrac{1}{18}t^3$ and $b(t)-\frac{1}{6}\ell t^2$ converge to zero in $C^1$.

Now consider the system \eqref{eq:Fundamental:ODE:U(1)} for the $4$ functions $x_1=\dot{a}\dot{b}$, $x_2=\dot{a}^2$, $y_1=a$, $y_2=b$. Write $x_1 = \tfrac{\ell}{18}t^3(1+X_1)$, $x_2 = \tfrac{1}{36}t^4(1+X_2)$, $y_1 = \tfrac{1}{18}t^3(1+Y_1)$ and $y_2 = \frac{\ell}{6}t^2(1+Y_2)$. After changing variable $e^\tau = t$, one can check that $(X_1,X_2,Y_1,Y_2)$ is a solution to an initial value problem of the form
\begin{equation}\label{eq:ODE:ALC:end}
\begin{aligned}
\dot{X}_1 &= -3(1+X_1) + 3\frac{(1+Y_1)(1+Y_2)^2 + O(e^{-2\tau})}{\sqrt{(1+Y_1)^2(1+Y_2)^2 + O(e^{-2\tau})}},\\
\dot{X}_2 &=  -4(1+X_2) + 4\frac{(1+Y_1)^2(1+Y_2) + O(e^{-2\tau})}{\sqrt{(1+Y_1)^2(1+Y_2)^2 + O(e^{-2\tau})}},\\
\dot{Y}_1 &= -3(1+Y_1) +3\sqrt{1+X_2},\\
\dot{Y}_2 &= -2(1+Y_2) +2\frac{1+X_1}{\sqrt{1+X_2}},
\end{aligned}
\end{equation}
where $O(e^{-2\tau})$ indicates a real analytic function of $(X_1,X_2,Y_1,Y_2)$ with coefficients depending real analytically on $e^{-\tau}$ and vanishing at $e^{-\tau}=0$ at least with order $2$.

Define $Z=(X_1,X_2,Y_1,Y_2,e^{-\tau})$ so that \eqref{eq:ODE:ALC:end} can be rewritten as an autonomous system $\frac{dZ}{ds}=\Phi(Z)$ for a real analytic map $\Phi$. The linearisation $d_0\Phi$ has a $1$-dimensional kernel and $4$ negative eigenvalues $-1$ (with multiplicity $2$), $-5$ and $-6$. The presence of a $1$-dimensional kernel is explained by the freedom to change $\ell$. In fact $\{ (c,0,0,c,0)\, |\,  c\in\R\}$ is the center manifold of the system $\frac{dZ}{ds}=\Phi(Z)$. Hence standard centre manifold theory \cite{Carr}*{\S 2.4, Theorem 2(b)} implies that any solution $Z$ that stays in a neighbourhood of $0$ must satisfy $Z = (c,0,0,c,0)+O(e^{-\tau})$ as $\tau\ra \infty$ for some $c\in\R$. In particular, if $(x_1,x_2,y_1,y_2)$ is a solution to \eqref{eq:Fundamental:ODE:U(1)} asymptotic to $\left( \tfrac{\ell}{18}t^3,\tfrac{1}{36}t^4, \tfrac{1}{18}t^3, \frac{\ell}{6}t^2\right)$ for some $\ell>0$ then
\[
\left( t^{-3}x_1,t^{-4}x_2, t^{-3}y_1,t^{-2}y_2\right) = \left( \tfrac{\ell}{18},\tfrac{1}{36}, \tfrac{1}{18}, \tfrac{\ell}{6}\right) + O(t^{-1}).
\]
The statement about the decay of derivatives of $t^{-3}a$ and $t^{-2}b$ then follows from a bootstrap argument.
\endproof
\end{lemma}

\begin{remark*}
In terms of a new independent variable $s=\frac{1}{\log{\tau}}$, the system \eqref{eq:ODE:ALC:end} takes the form \mbox{$s\dot{z}=\Phi(z,s^2)$} of Theorem \ref{thm:Singular:IVP:Extended}. The linearisation of $\Phi(\, \cdot\, , 0)$ at the origin has $4$ distinct eigenvalues $0,1,5,6$. We have already observed that the $1$-dimensional kernel is due to the freedom of choosing $\ell>0$. The eigenvector with eigenvalue $1$ is due to the $t$--invariance of the original system. The eigenvectors with eigenvalue $5,6$ are, respectively, $\left( -3,0,0,2\right)$ and $ \left( -1,2,-1,1\right)$. Moreover, the Hamiltonian constraint $H(x,y)=0$ is satisfied at leading order as $(s,z)\ra 0$ if and only if $2X_1 + X_2 =2(Y_1+Y_2)$. The eigenvector with eigenvalue $5$ does not satisfy this constraint. One could then try to use Theorem \ref{thm:Singular:IVP:Extended} to show that, up to translations in $t$ and the scaling freedom to fix the asymptotic length $\ell$ of the circle fibre, there exists a $1$-parameter family of $\sunitary{2}\times\sunitary{2}\times\unitary{1}$--invariant ALC $\gtwo$--holonomy ends. However, since the non-resonance condition (ii) in Theorem \ref{thm:Singular:IVP:Extended} is not satisfied, 
this does not follow immediately from that theorem. We do not pursue the matter further here since it is not necessary in our analysis.
\end{remark*}

Thanks to Lemma \ref{lem:ALC:growth:a:b}, in order to prove that a complete end is ALC it is enough to control the quantities $\frac{a^2}{b^3}$ and $\frac{a}{b}\frac{db}{da}$.

Let us fix the notation that will be used throughout the rest of the section. Let $\lambda$ denote the ratio $\lambda= \frac{\dot{a}}{\dot{b}}$. Then \eqref{eq:Fundamental:ODE:Brandhuber:U(1)} can be rewritten as
\[
2F\dot{\lambda} = \lambda\dot{b} \left( 2F_b - \lambda F_a\right).
\]
We fix a parameter $s$ that satisfies $s\ra \infty$ along the complete end (for example we can take $s=a$ or $s=t$, the arc-length parameter along a geodesic meeting all principal orbits orthogonally).

For $\alpha\in\R$ introduce the following pair of (positive) ratios
\[
P_\alpha = \frac{b^{1+\alpha}}{a}, \qquad Q_\alpha = \frac{b^\alpha}{\lambda}.
\]
Using \eqref{eq:Fundamental:ODE:Brandhuber:U(1)} as rewritten above we find that these two quantities satisfy
\begin{subequations}\label{eq:ALC:growth:a:b}
\begin{equation}
a^2\dot{P}_\alpha= b^\alpha\dot{b} R_\alpha, \qquad 2F\dot{Q}_\alpha = b^{-1}\dot{b}Q_\alpha \left( S_\alpha - F_a R_\alpha\right),
\end{equation}
where
\[
R_\alpha = (1+\alpha) a - b\lambda, \qquad S_\alpha = \alpha (2F + aF_a) - (2b F_b -aF_a).
\]
Moreover,
\begin{equation}
2F\dot{R}_\alpha = \lambda \dot{b} \left( S_\alpha - F_a R_\alpha\right).
\end{equation}
\end{subequations}

\begin{prop}\label{prop:ALC:growth:a:b}
Assume that either $q\geq p$ or $q=-p\leq 0$. If $(a,b)$ is a solution to \eqref{eq:Fundamental:ODE:Brandhuber:U(1)} satisfying \eqref{eq:ALC:Chamber} and additionally
\begin{equation}\label{eq:ALC:Chamber:additional:constraint}
a\dot{b}-\dot{a}b<0
\end{equation}
holds at some initial time, then along the complete end the ratio $\frac{a^2}{b^3}$ converges to a constant and $\frac{a}{b}\frac{db}{da}$ converges to $\frac{2}{3}$.
\proof
The proof is based on the study of the behaviour of the positive ratios $P_\alpha$ and $Q_\alpha$ introduced above for $\alpha\geq 0$.

We begin by analysing (\ref{eq:ALC:growth:a:b}b) for $\alpha$ sufficiently large. Since $a>0$, for all $\alpha$ sufficiently large $R_\alpha$ is positive at the initial time $s_0$. We claim that, assuming $\alpha$ even larger if necessary, $R_\alpha$ remains positive for all $s\geq s_0$. Indeed, using $2\alpha F>0$ we find
\[
\begin{aligned}
S_\alpha &> (1+\alpha) a F_a - 2b F_b \\& = 8a^2 \left( (1+\alpha) (b-p)(b+q) - (b-p)-(b+q)\right) + 8b^2 (b^2 +pq)\\
&> 8a^2 \left( (1+\alpha) (b-p)(b+q) - (b-p)-(b+q)\right).
\end{aligned}
\] 
If $\alpha > \frac{1}{2b-p+q}$ then $(1+\alpha) (b-p)(b+q) - (b-p)-(b+q)$ is an increasing function of $b$. Since $\dot{b}>0$ we conclude that $S_\alpha>0$ provided we choose $\alpha$ so that 
\[(1+\alpha) (b-p)(b+q) - (b-p)-(b+q)>0, \quad \alpha > \frac{1}{2b-p+q}
\]
are satisfied at the initial time $s_0$. Then (\ref{eq:ALC:growth:a:b}b) shows that for large enough $\alpha$  $\dot{R}_\alpha>0$ whenever $R_\alpha=0$. We conclude that $R_\alpha (s)>0$ for all $s\geq s_0$ and all sufficiently large $\alpha$ as claimed. Hence by (\ref{eq:ALC:growth:a:b}a), for all sufficiently large $\alpha$, $P_\alpha$ is an increasing function of $s$, and in particular it is bounded away from $0$. In particular $b$, as well as $a$, is unbounded along the complete end. For otherwise $P_\alpha = \frac{b^{1+\alpha}}{a}\leq \frac{c}{a}\ra 0$. Note that we have not yet made use of assumption \eqref{eq:ALC:Chamber:additional:constraint}, only of assumptions~\eqref{eq:ALC:Chamber}.

We now consider the equation (\ref{eq:ALC:growth:a:b}b) for small $\alpha$. Note that our assumption \eqref{eq:ALC:Chamber:additional:constraint} is equivalent to the assumption that $R_0$ is negative at $s_0$. Hence by continuity $R_\alpha$ is also negative at $s_0$ for $\alpha>0$ sufficiently small. We want to show that for all $\alpha\geq 0$ sufficiently small $R_\alpha$ remains negative for all $s\geq s_0$. First consider the case $\alpha=0$. We have $S_0 = -(2b F_b - aF_a)<0$ by (\ref{eq:Identities:F}b) and \eqref{eq:ALC:Chamber}. Hence (\ref{eq:ALC:growth:a:b}b) shows that at a point where $R_0=0$ we must have $\dot{R}_0<0$. We therefore conclude that $R_0 (s)<0$ for all $s\geq s_0$. In particular, $P_0$ is strictly decreasing by (\ref{eq:ALC:growth:a:b}a).

Now, since $S_0<0$, for all $s\geq s_0$ there exists $\alpha_s$ such that $S_\alpha (s)<0$ for all $0\leq \alpha<\alpha_s$. In order to show that we can choose $\alpha_s$ independent of $s$, using $a,b\ra\infty$ we calculate
\[
\lim_{s\ra\infty} {S_\alpha} = \lim_{s\ra \infty}{\alpha (16 a^2 b^2 -2b^4)-8a^2b^2 + 8b^4}. 
\]
Thus $\lim_{s\ra\infty} {S_\alpha}<0$ if and only if
\[
\alpha < \lim_{s\ra \infty}{\frac{4(a^2-b^2)}{(8a^2-b^2)}} = \lim_{s\ra \infty}{\frac{4(1-P_0^2)}{8-P_0^2}}.
\]
Note that this rational function of $P_0$ is monotone decreasing for $P_0$ in $(0,1)$ with range $(0,\frac{1}{2})$. Since $P_0$ takes values in $(0,1)$ and is decreasing in $s$ we conclude that $\lim_{s\ra\infty}{\alpha_s}>0$ and therefore $\alpha_s$ is bounded below.

As in the case $\alpha=0$, we now conclude that $R_\alpha<0$ for all $s\geq s_0$ and any $\alpha\geq 0$ sufficiently small. Then $P_\alpha$ is decreasing and therefore bounded above. We conclude that $\frac{b}{a}\ra 0$ as $s\ra \infty$ since otherwise $P_\alpha = \frac{b}{a}b^\alpha \geq c b^\alpha$ could not be bounded.   

We can now study the behaviour of $R_\alpha$ for all $\alpha\in\R$. Using the fact that $a,b\ra \infty$ and $a$ dominates $b$ to show that $2F\approx 8a^2 b^2$ and $S_\alpha \approx 8(2\alpha-1)a^2 b^2$ along the complete end, we now conclude that for every $\epsilon>0$ there exists $s_\epsilon>0$ such that
\[
\frac{dR_\alpha}{da} \leq (2\alpha-1+\epsilon) \mbox{ if } R_\alpha\geq 0, \qquad \frac{dR_\alpha}{da} \geq (2\alpha-1-\epsilon) \mbox{ if } R_\alpha\leq 0
\]
for all $s>s_\epsilon$. We conclude that $R_\alpha$ eventually becomes strictly negative if $\alpha<\frac{1}{2}$ and strictly positive if $\alpha>\frac{1}{2}$. In other words, for all $\delta>0$ there exists $s_\delta>0$ such that
\[
\pm \left( \tfrac{3}{2}a-\lambda b \pm \delta a\right)>0
\]
for all $s>s_\delta$. Rearranging and writing $\lambda^{-1}=\frac{db}{da}$, we therefore have $\lim_{t\ra\infty}{\frac{a}{b}\frac{db}{da}}=\frac{2}{3}$. Note that so far we have not used the assumption that either $q \geq p$ or $q=-p \leq 0$.

In order to show that $P_\frac{1}{2}$ converges we now consider the two quantities $P_{\frac{1}{2}}$ and $Q_\frac{1}{2}$ at the same time. First note that along the end $S_\frac{1}{2}\approx 7b^4 + 8(q-p)a^2 b$.

Assume now that $q\geq p$ so that $S_\frac{1}{2}$ is eventually positive. In particular, $R_\frac{1}{2}$ has a definite sign along the complete end, since if it becomes positive for large enough $s$ then it must remain positive by (\ref{eq:ALC:growth:a:b}b).

If $R_\frac{1}{2}$ is eventually negative, then by (\ref{eq:ALC:growth:a:b}a) $P_\frac{1}{2}$ is eventually decreasing, and in particular bounded above, and $Q_\frac{1}{2}$ is eventually increasing and therefore bounded below. Moreover, since $R_\frac{1}{2}<0$ can be rewritten as $3Q_\frac{1}{2} \leq 2 P_\frac{1}{2}$, we conclude that $P_\frac{1}{2}$ and $Q_\frac{1}{2}$ are monotone and bounded and therefore convergent. If $R_\frac{1}{2}$ is eventually positive, then eventually $3Q_\frac{1}{2} \geq 2 P_\frac{1}{2}$, $P_\frac{1}{2}$ is increasing, while $Q_\frac{1}{2}$ satisfies
\[
\frac{d\log{Q_\frac{1}{2}}}{db} \leq \frac{S_\frac{1}{2}}{2bF}\approx \frac{7b}{8a^2} + \frac{(q-p)}{b^2}.
\]
Now, since we know that $b^\gamma/a$ converges to zero along the end for all $\gamma<\frac{3}{2}$, we conclude that for $b$ sufficiently large
\[
\frac{d\log{Q_\frac{1}{2}}}{db} \leq c_\gamma \left( b^{1-2\gamma} + b^{-2}\right)
\]
for some constant $c_\gamma>0$. Choosing $2\gamma \in (2,3)$ makes the right-hand side integrable in $b$ as $b\ra \infty$. Thus $Q_\frac{1}{2}$ is bounded along the end. We conclude that $P_\frac{1}{2}$ is bounded above and increasing and therefore convergent. The proof in the case $q\geq p$ is now complete.

When $q=-p<0$ the final part of the argument, \ie the convergence of $P_{\frac{1}{2}}$, breaks down because $S_\frac{1}{2}$ does not necessarily have a definite sign. We modify the argument as follows. For $\alpha\geq 0$ we now consider functions
\[
P_\alpha = \frac{(b-p)^{1+\alpha}}{a-p}, \qquad Q_\alpha = \frac{(b-p)^\alpha}{\lambda}.
\]
The analogues of \eqref{eq:ALC:growth:a:b} are
\[
(a-p)^2\dot{P}_\alpha= (b-p)^\alpha\dot{b} R_\alpha, \qquad 2F\dot{Q}_\alpha = (b-p)^{-1}\dot{b}Q_\alpha \left( S_\alpha - F_a R_\alpha\right),
\]
but now we must define
\[
R_\alpha = (1+\alpha) (a-p) - (b-p)\lambda, \qquad S_\alpha = \alpha (2F + aF_a-pF_a) - (2b F_b -aF_a) + p(2F_b-Fa).
\]
With these new definitions we also have
\[
2F\dot{R}_\alpha = \lambda \dot{b} \left( S_\alpha - F_a R_\alpha\right).
\]
Using $q=-p$ we now calculate
\[
\frac{S_\alpha}{2(b-p)^2}  = 4(2\alpha-1)a^2 -4(1+\alpha) p a +(4-\alpha) (b+p)^2 -4p(b+p).
\]
Indeed, when $q=-p$ then $F=(b-p)^2(2a - b-p)(2a+b+p)$ and the formulas in \eqref{eq:Identities:F} become
\[
2F_b-F_a = 8(a-b)(2a+b+p)(b-p),\qquad 
2bF_b -a F_a = 8(a-b)(a+b)(b-p)(b+p).
\]
Since $a,b, a/b\ra\infty$ as $s\ra \infty$, we see that $S_\alpha \approx 8(2\alpha-1)a^2 b^2 + 2(4-\alpha)b^4$ for large $s$. In particular $S_\alpha$ is eventually positive for all $\alpha\geq\frac{1}{2}$. The proof now proceeds exactly as before.
\endproof
\end{prop}

\subsection{Incompleteness}

The tools we have developed to prove forward completeness and ALC asymptotics can also be used to prove that certain solutions yield incomplete metrics.

\begin{prop}\label{prop:death:quadrant}
Let $(a,b)$ be a solution to \eqref{eq:Fundamental:ODE:Brandhuber:U(1)} with $a>0$, $b>\max{\left(p,-q,\sqrt{-pq}\right)}$ and $\dot{a}, \dot{b},F(a,b)>0$. If there exists a time such that
\begin{equation}\label{eq:death:quadrant}
0<\frac{\dot{a}}{\dot{b}}<\frac{a}{b}<1
\end{equation}
then the solution cannot be forward complete.
\proof
Recall the functions $P_\alpha, R_\alpha$ and $S_\alpha$ introduced just before Proposition \ref{prop:ALC:growth:a:b} and the evolution equations \eqref{eq:ALC:growth:a:b} for $P_\alpha$ and $R_\alpha$. We will assume that the solution is complete and derive a contradiction by studying the behaviour of the function $P_\alpha$ for small enough $\alpha<0$.

We first show that the conditions \eqref{eq:death:quadrant} are preserved for all time the solution exists and satisfies $\dot{a}, \dot{b},F(a,b)>0$. The fact that the conditions $a<b$ and $\dot{a}<\dot{b}$ persist is proved as in Lemma \ref{lem:ALC:chamber} exploiting the fact that $2F_b-F_a<0$ whenever $0<a<b$ and $b>\max{\left(p,-q,\sqrt{-pq}\right)}$ by (\ref{eq:Identities:F}a). In order to prove the persistence of the condition $\frac{\dot{a}}{\dot{b}}<\frac{a}{b}$ consider the quantity $R_0$ and note that $\frac{\dot{a}}{\dot{b}}<\frac{a}{b}$ is equivalent to $R_0>0$. It is enough to observe that $\dot{R}_0>0$ at any point where $R_0=0$. Indeed, by (\ref{eq:ALC:growth:a:b}b) at a point where $R_0=0$, $\dot{R}_0$ has the same sign as $S_0 = aF_a - 2bF_b$. By (\ref{eq:Identities:F}b), $S_0>0$ whenever $0<a<b$ and $b>\max{\left(p,-q,\sqrt{-pq}\right)}$. We conclude that $R_0>0$ for all time. Note that in particular $P_0=\frac{b}{a}$ is strictly increasing by (\ref{eq:ALC:growth:a:b}a).

Assume now for a contradiction that the solution is complete. Using $\dot{a}<\dot{b}$ we conclude that $b\ra \infty$ along the complete end in the same way that we proved that $a$ was unbounded in Proposition \ref{eq:ALC:Forward:Completeness}. Since $F(a,b)>0$ for all time we must also have
\begin{equation}\label{eq:a:unbounded}
4\left( \frac{a}{b}\right)^2 > \frac{(b^2+pq)^2}{b^2 (b-p)(b+q)}\longrightarrow 1
\end{equation}
as $b\ra \infty$. Hence $a$ is also unbounded along the complete end. Moreover, since $P_0=\frac{b}{a}$ is increasing and initially $P_0>1$ by \eqref{eq:death:quadrant}, $a\approx k b$ for some $\frac{1}{2}\leq k < 1$ as $b\ra\infty$.

We are now going to derive a contradiction to the completeness assumption by studying the behaviour of the function $P_\alpha = \frac{b^\alpha}{a}$ for small enough $\alpha<0$. Since $S_0>0$ as observed earlier, there exists $\alpha_s>0$ such that $S_\alpha (s)>0$ for all $| \alpha|<\alpha_s$. In order to show that $\alpha_s$ can be chosen independent of $s$, using $a\approx kb\ra\infty$ we calculate
\[
\lim_{s\ra\infty} {\frac{S_\alpha}{b^4}} = \lim_{s\ra \infty}{\frac{\alpha (16 a^2 b^2 -2b^4)-8a^2b^2 + 8b^4}{b^4}}= 2\alpha (8 k^2 -1)+ 8(1-k^2). 
\]
Since $\frac{1}{2}\leq k<1$, both of these coefficients are positive and we conclude that $\lim_{s\ra\infty} {S_\alpha}>0$ for all $\alpha > -\frac{4(1-k^2)}{8k^2-1}$. We conclude that $S_\alpha>0$ for all $\alpha$ sufficiently small as claimed.

Now, the condition $R_0(s_0)>0$ for some initial time $s_0$, which holds by assumption \eqref{eq:death:quadrant}, forces $R_\alpha (s_0)>0$ for all small enough $\alpha$ by continuity. Since $S_\alpha>0$ we can now conclude that $R_\alpha>0$ for all time and any $\alpha$ sufficiently small as we did above in the case $\alpha=0$. In particular $P_\alpha$ is positive and increasing for $\alpha$ sufficiently small. If $\alpha<0$ we now reach a contradiction since $P_\alpha = \frac{b}{a}b^\alpha \approx \frac{1}{k}b^\alpha\ra 0$.
\endproof
\end{prop}

\subsection{Global behaviour}
We now use Propositions \ref{prop:ALC:growth:a:b} and \ref{prop:death:quadrant} to describe the global behaviour of the local cohomogeneity one torsion-free \gtstr s of Proposition \ref{prop:Solutions:Singular:Orbit} (i) and (ii) and of Proposition \ref{prop:CS:AC:ends} (i) under an enhanced $\unitary{1}$--symmetry assumption.

\begin{theorem}\label{thm:ALC:B7}
Let $(ijk)$ be a cyclic permutation of $(123)$. Consider the local cohomogeneity one torsion-free \gtstr s of Proposition \ref{prop:Solutions:Singular:Orbit} (i) and the subfamily defined by $\alpha_j=\alpha_k$.
\begin{enumerate}[leftmargin=*]
\item If $\alpha_i<\alpha_j=\alpha_k$ then the torsion-free {\gtstr} extends to a complete $\sunitary{2}\times\sunitary{2}\times\unitary{1}$--invariant ALC {\gtmetric} on $S^3\times\R^4$. Here the $\unitary{1}$--action is the Hopf circle action on the second factor.
\item If $\alpha_i=\alpha_j=\alpha_k$ then the torsion-free {\gtstr} is the Bryant--Salamon's complete $\sunitary{2}^3$--invariant AC {\gtmetric} on $S^3\times\R^4$.
\item If $\alpha_i>\alpha_j=\alpha_k$ then the torsion-free {\gtstr} is incomplete.
\end{enumerate}  
\end{theorem}

\begin{theorem}\label{thm:ALC:D7}
Let $(ijk)$ be a cyclic permutation of $(123)$. Consider the local cohomogeneity one torsion-free \gtstr s of Proposition \ref{prop:Solutions:Singular:Orbit} (ii) and the subfamily defined by $\alpha_j=\alpha_k$.
\begin{enumerate}[leftmargin=*]
\item If $\alpha_i<\alpha_j=\alpha_k$ then the torsion-free {\gtstr} extends to a complete $\sunitary{2}\times\sunitary{2}\times\unitary{1}$--invariant ALC {\gtmetric} on $S^3\times\R^4$. Here the $\unitary{1}$--action is the Hopf circle action on the first factor.
\item If $\alpha_i=\alpha_j=\alpha_k$ then the torsion-free {\gtstr} is the Bryant--Salamon's complete $\sunitary{2}^3$--invariant AC {\gtmetric} on $S^3\times\R^4$.
\item If $\alpha_i>\alpha_j=\alpha_k$ then the torsion-free {\gtstr} is incomplete.
\end{enumerate}  
\end{theorem}

\begin{theorem}\label{thm:ALC:CS}
Consider the conically singular $\sunitary{2}\times\sunitary{2}\times\unitary{1}$--invariant torsion-free {\gtstr} of Proposition \ref{prop:CS:AC:ends} (i) parametrised by $c\in\R$.
\begin{enumerate}[leftmargin=*]
\item If $c>0$ then the solution extends to a torsion-free {\gtstr} on $(0,\infty)\times S^3\times S^3$ with a CS end as $t\ra 0$ and an ALC end as $t\ra \infty$.
\item If $c=0$ then the solution is the $\gtwo$--cone over the $\sunitary{2}^3$--invariant nearly K\"ahler structure over $S^3\times S^3$.
\item If $c<0$ then the solution is incomplete.
\end{enumerate}  
\end{theorem}

\begin{remark*}
In all three cases one parameter can be fixed by scaling, hence we have found two $1$-parameter families of complete ALC metrics and a unique CS ALC manifold up to scale. The ALC family of Theorem \ref{thm:ALC:B7} is the $\mathbb{B} _7$ family of Brandhuber--Gomis--Gubser--Gukov \cite{BGGG} and Bogoyavlenskaya \cite{Bogoyavlenskaya}. The ALC family of Theorem \ref{thm:ALC:D7} is the conjectured $\mathbb{D} _7$ family in the physics literature.
\end{remark*}

Case (ii) in Theorems \ref{thm:ALC:B7}, \ref{thm:ALC:D7} and \ref{thm:ALC:CS} is characterised by an enhanced $\sunitary{2}$--symmetry and is therefore dealt with in Example \ref{ex:Bryant:Salamon}. We therefore concentrate on proving cases (i) and (iii) in each theorem. Up to a change of basis, we can assume that the additional $\unitary{1}$--action is generated by $E_3+E'_3$ in all cases, \ie $(ijk)=(312)$ in Theorems \ref{thm:ALC:B7} and \ref{thm:ALC:D7}.

Proposition \ref{prop:Solutions:Singular:Orbit} (i) and (ii) and Proposition \ref{prop:CS:AC:ends} (i) provide the leading-order behaviour of $a=a_1=a_2$ and $b=a_3$ as $t\ra 0$. We need to check that the hypotheses of Propositions \ref{prop:ALC:growth:a:b} and \ref{prop:death:quadrant} are satisfied, \ie the conditions on $p$ and $q$ together with the inequalities \eqref{eq:ALC:Chamber} and \eqref{eq:ALC:Chamber:additional:constraint} or \eqref{eq:death:quadrant}.
\begin{itemize}[leftmargin=*]
\item In the case of Theorem \ref{thm:ALC:B7}, $q=-p<0$ so $b-p, b+q, b-\sqrt{-pq}>0$. Moreover for small $t$ the signs of $a-b$, $\dot{a}-\dot{b}$ and $\dot{a}b-a\dot{b}$ are all the same as the sign of $\alpha_1-\alpha_3$.
\item In the case of Theorem \ref{thm:ALC:D7}, $q=0$ and $p=-r_0^3<0$. Moreover, $\alpha_3>0$ since $\alpha_1\alpha_2\alpha_3=1$. Then $b-p, b+q, b-\sqrt{-pq}>0$. The signs of $a-b, \dot{a}-\dot{b}>0$ are the same as the sign of $\alpha_1-\alpha_3$. In order to control the sign of $\dot{a}b-a\dot{b}$ for small $t$ we have to use the higher-order expansions of $a,b$ in Remark \ref{rmk:Solutions:Singular:Orbit:D7}. We find
\[
\dot{a}b-a\dot{b} = \frac{(1-\alpha_3^3)\alpha_3}{96\alpha_1}t^5 + O(t^7).
\]
Since $\alpha_1^2\alpha_3=1$, if $\alpha_3<\alpha_1$ then $0<\alpha_3<1<\alpha_1$, while $0<\alpha_1<1<\alpha_3$ if $\alpha_3>\alpha_1$.
\item In the case of Theorem \ref{thm:ALC:CS}, $p=0=q$ and for small $t>0$ we have $b>0$. The signs of $a-b,\dot{a}-\dot{b}>0$ are the same as the sign of $c$. In order to control the sign of $\dot{a}b-a\dot{b}$ we calculate
\[
\frac{54^2\left(\dot{a}b-a\dot{b}\right)}{3t^5} \approx \tfrac{3}{2}c\, \nu_0 t^{\nu_0}
\] 
as $t\ra 0$.
\end{itemize}

\begin{remark*}
Theorems \ref{thm:ALC:B7} and \ref{thm:ALC:D7} do not address the issue of what happens to solutions that do not enjoy an enhanced $\unitary{1}$--symmetry. Our expectation is that all of these local cohomogeneity one metrics are incomplete.
\end{remark*}

\section{Existence of AC metrics}\label{sec:AC}

In Proposition \ref{prop:Solutions:Singular:Orbit} we constructed four different families of local solutions to \eqref{eq:Fundamental:ODE} that close smoothly over various different singular orbits. In the previous section we have shown that a subset of the local solutions closing smoothly on a singular orbit $S^3$ constructed in Proposition  \ref{prop:Solutions:Singular:Orbit} (i) and (ii) extend to complete ALC metrics. None of the local solutions constructed in Proposition  \ref{prop:Solutions:Singular:Orbit} (iii) and (iv) is covered by these results: no choice of parameters there allows us to satisfy all the hypotheses of Propositions \ref{prop:ALC:growth:a:b} or \ref{prop:death:quadrant} for small positive $t$. In this section we use a different approach to study $\sunitary{2}\times\sunitary{2}\times\unitary{1}$--invariant \gtmfd s with singular orbit $\sunitary{2}\times\sunitary{2}/K_{m,n}$ and thus prove the following theorem.

\begin{theorem}\label{thm:ALC:C7:m:n}
Fix coprime positive integers $m,n$ and a real number $r_0>0$. For $\beta>0$, let
\[
\varphi_\beta = -m^2 r_0^3\, e_1\wedge e_2\wedge e_3 + n^2 r_0^3\, e'_1\wedge e'_2\wedge e'_3 + d\left( a\, (e_1\wedge e'_1 + e_2\wedge e'_2) + b\, e_3\wedge e'_3\right)
\]
be the (locally defined) $\sunitary{2}\times\sunitary{2}\times\unitary{1}$--invariant torsion-free {\gtstr} closing smoothly on $\sunitary{2}\times\sunitary{2}/K_{m,n}$ defined in Proposition  \ref{prop:Solutions:Singular:Orbit} (iii) (when $m=n=1$) or (iv) satisfying
\[
a = r_0^2\beta t + O(t^3), \qquad b = mn r_0^3 + O(t^2)
\]
as $t\ra 0$. There exists $\beta_\tu{ac}>0$ such that the following holds.
\begin{enumerate}[leftmargin=*]
\item If $\beta>\beta_\tu{ac}$ then $\varphi_\beta$ extends to a complete torsion-free ALC {\gtstr} asymptotic to a circle bundle over a $\Z_2$--quotient conifold.
\item If $\beta=\beta_\tu{ac}$ then $\varphi_\beta$ extends to a complete torsion-free AC {\gtstr} asymptotic to the cone over the $\Z_{2(m+n)}$--quotient of the homogeneous nearly K\"ahler structure on $S^3\times S^3$ with rate $-3$.  
\item If $\beta<\beta_\tu{ac}$ then $\varphi_\beta$ does not extend to a complete torsion-free {\gtstr}.
\end{enumerate}
\end{theorem}
Figure \ref{fig:m1n2} illustrates a small number of solution curves $(a,b)$ to the ODE system \eqref{eq:Fundamental:ODE:Brandhuber:U(1)}
illustrating all three cases of Theorem~\ref{thm:ALC:C7:m:n} in the case where $m=1$ and $n=2$.
\begin{figure}
\begin{center} 
\includegraphics[scale=0.75]{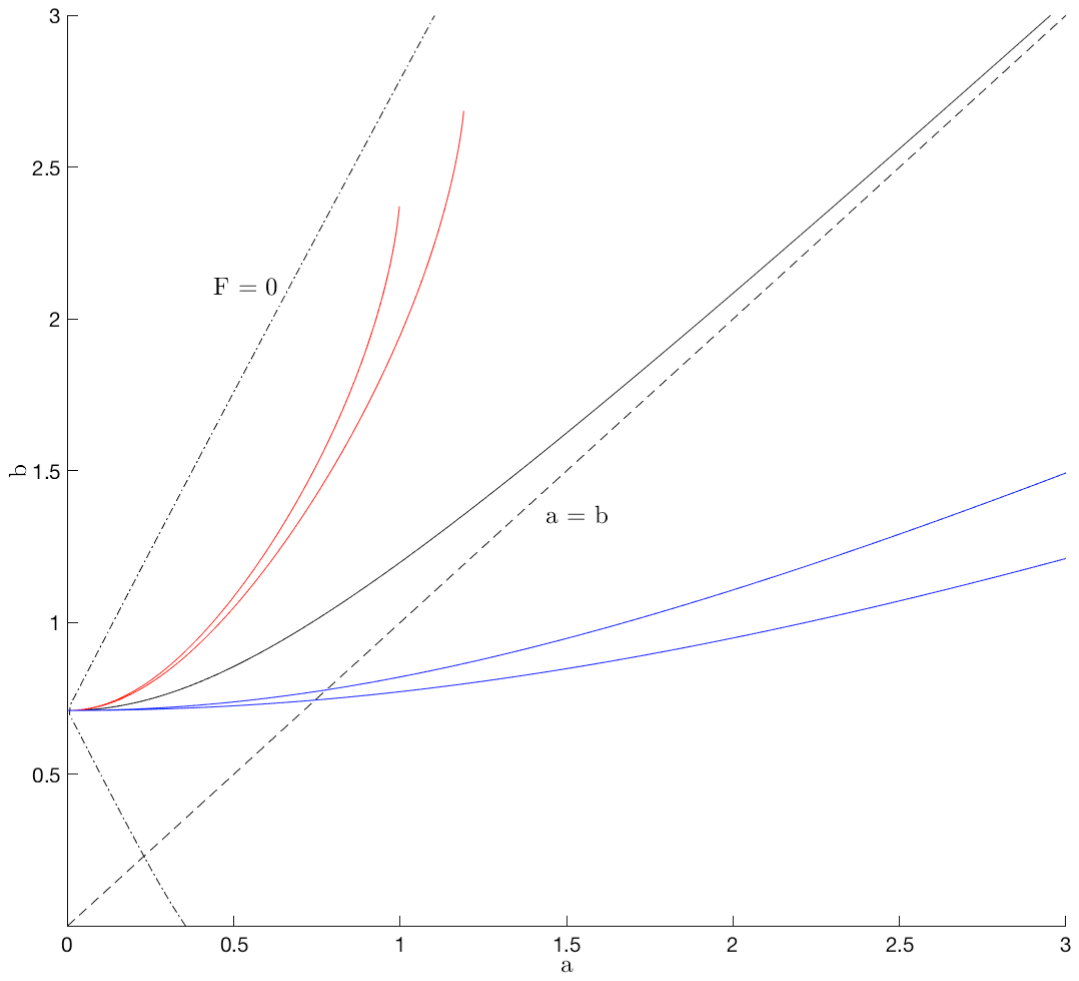}
\end{center}
\vspace{-6.7cm}
\caption{\small \sl Numerical solutions of \eqref{eq:Fundamental:ODE:Brandhuber:U(1)} with $m=1$, $n=2$ satisfying the initial conditions of Proposition \ref{prop:Solutions:Singular:Orbit} (iv): 
There are two incomplete solutions approaching the set F=0, the unique AC solution asymptotic to the diagonal $a=b$ and two complete solutions crossing the diagonal transversally.}
\label{fig:m1n2}
\end{figure}

\begin{remark*}
The existence of the ALC metrics in part (i) of the Theorem when $\beta$ is sufficiently large is guaranteed by our analytic construction of highly collapsed ALC \gtmetric s on circle bundles over AC Calabi--Yau $3$-folds \cite{FHN:ALC:G2:from:AC:CY3}*{Theorem 9.7}.
\end{remark*}

\begin{remark*}
The case $m=n=1$ has also been considered by Baza\u{\i}kin--Bogoyavlenskaya \cite{Bazaikin:Bogoyavlenskaya} and Cveti\v{c}--Gibbons--L\"u--Pope \cite{CGLP:C7}*{\S 3}.
In \cite{Bazaikin:Bogoyavlenskaya} the existence of ALC metrics for every positive value of $\beta$ was claimed. However, there appear to be mistakes in the proof of \cite{Bazaikin:Bogoyavlenskaya}*{Lemma 9}. In \cite{CGLP:C7}*{\S 3} numerical experiments suggested the existence of a full $3$-parameter family of ALC \gtmetric s closing smoothly on the singular orbit $\sunitary{2}\times\sunitary{2}/K_{1,1}$. 
In the collapsed limit this contradicts our analysis in \cite{FHN:ALC:G2:from:AC:CY3}*{Theorem 9.7}. 
\end{remark*}

The most interesting part of Theorem \ref{thm:ALC:C7:m:n} is part (ii). Only three simply connected AC \gtmetric s (up to symmetries and scaling) are currently known \cite{Bryant:Salamon}. Part (ii) of the theorem provides infinitely many new AC \gtmetric s. It is essential that we consider AC manifolds asymptotic to a non-trivial quotient of the \gtwo--cone over $S^3\times S^3$ since, by Karigiannis--Lotay \cite{Karigiannis:Lotay}*{Corollary 6.10}, the Bryant--Salamon metric on the spinor bundle of $S^3$ is the unique (up to scale) AC {\gtmetric} asymptotic to the cone over $S^3\times S^3$.

\begin{remark}\label{rmk:Transitions:AC}
The smooth $7$--manifold $M_{m,n}$ underlying the \gtmetric s constructed in the Theorem depends only on the sum $m+n$. In fact, $M_{m,n}$ can be identified with $H^{2(n+m)}\times S^3$, where $H^{2(n+m)}$ is the total space of the $\R^2$--bundle on $S^2$ with Euler class $2(n+m)$. However, metrics for different choices of $(m,n)$ can never be isometric. Indeed, an isometry would have to respect the $\sunitary{2}\times\sunitary{2}\times\unitary{1}$--orbit structure: otherwise the tangent space at a point would be spanned by Killing vectors and the metric would be homogeneous; the latter is impossible since the metric is Ricci-flat but cannot be flat. The group of $\sunitary{2}\times\sunitary{2}\times\unitary{1}$--equivariant diffeomorphisms of the principal orbit $\sunitary{2}\times\sunitary{2}/\Z_{2(m+n)}$ is $\sunitary{2}\times\sunitary{2}\times \tu{N} \rtimes \Z_2$, where $\tu{N}$ is the normaliser of $\unitary{1}$ in $\sunitary{2}$ acting on the right on $\sunitary{2}\times\sunitary{2}$ and $\Z_2$ is generated by the outer automorphism of $\sunitary{2}\times\sunitary{2}$ that exchanges the two factors. The induced action of $\sunitary{2}\times\sunitary{2}\times \tu{N} \rtimes \Z_2$ on cohomology is generated by the involution that exhanges $e_1\wedge e_2\wedge e_3$ and $e'_1\wedge e'_2\wedge e'_3$. Since the image of the cohomology class of $\varphi$ in the cohomology of the principal orbits depends on the pair $(m,n)$, we conclude that different choices of $(m,n)$ with $0<m\leq n$ and $\gcd (m,n)=1$ give rise to non-isometric metrics. In fact this argument also shows that, modulo the outer automorphism of $\sunitary{2}\times\sunitary{2}$, there is no diffeomorphism between $M_{m,n}$ and $M_{m',n'}$ asymptotic to an isometry of the asymptotic cone. In particular, considering pairs $(m,n)$ with $0<m\leq n$ and fixed (sufficiently large) $m+n$, part (ii) of the Theorem yields different AC \mbox{\gtmetric s} asymptotic to the same $\gtwo$--cone and therefore gives rise to infinitely many new geometric transitions in $\gtwo$--geometry.     
\end{remark}

Before proving Theorem \ref{thm:ALC:C7:m:n} we establish that the complete solutions obtained in Theorems \ref{thm:ALC:B7}, \ref{thm:ALC:D7} and \ref{thm:ALC:C7:m:n} (i) and (ii) are the only complete simply connected $\sunitary{2}\times\sunitary{2}\times\unitary{1}$--invariant \gtmfd s.

\begin{theorem}
\label{thm:Classification:U(1)}
Let $(M,g)$ be a complete $\sunitary{2}\times\sunitary{2}\times\unitary{1}$--invariant {\gtmetric} with $M$ simply connected. Then up to symmetries $(M,g)$ is isometric to one of the complete metrics of Theorems \ref{thm:ALC:B7}, \ref{thm:ALC:D7} and \ref{thm:ALC:C7:m:n}.
\proof
Given the completeness and incompleteness statements in Theorems \ref{thm:ALC:B7}, \ref{thm:ALC:D7} and \ref{thm:ALC:C7:m:n}, it only remains to prove that if $(M,g)$ is an $\sunitary{2}\times\sunitary{2}\times\unitary{1}$--invariant {\gtmetric} closing smoothly on a singular orbit $Q$, then, up to a finite cover and the action of the outer automorphism of $\sunitary{2}^2$, $M$ is described by one of the group diagrams \eqref{eq:group:diagram:D7}, \eqref{eq:group:diagram:C7(m,n)} and \eqref{eq:group:diagram:B7}.

Let
\[
\varphi = p\, e_1\wedge e_2\wedge e_3 + q\, e'_1\wedge e'_2\wedge e'_3 + d\left( a \, (e_1\wedge e'_1 + e_2\wedge e'_2) + b \, e_3\wedge e'_3\right)
\]
be an $\sunitary{2}\times\sunitary{2}\times\unitary{1}$--invariant torsion-free {\gtstr} defined in a neighbourhood of a singular orbit $Q$. Let $t$ be the arc-length parameter along a geodesic meeting all orbits orthogonally and assume that the point $t=0$ lies on the singular orbit $Q$. In particular, $a,b$ are smooth functions defined on $[0,t_0)$ for some $t_0>0$.

In order to determine the behaviour of the functions $a$ and $b$ as $t\ra 0$, observe that $F\ra 0$ as $t\ra 0$ since $\sqrt{F(a,b)}=2\dot{a}^2\dot{b}$ is the orbital volume function. Moreover, the evolution equations for $x_1=\dot{a}\dot{b}$ and $x_2=\dot{a}^2$ in \eqref{eq:Fundamental:ODE:U(1)} show that $(F_a,F_b)\ra 0$ as $t\ra 0$, \ie the point $(a_0,b_0)=(a,b)|_{t=0}$ must be a critical point of $F$ on the level set $F=0$. As an aside, note that we must have $pq\leq 0$ since $F=0=F_a$ force $b^2+pq=0$. Since the coefficients of the autonomous ODE system \eqref{eq:Fundamental:ODE:U(1)} depend real analytically on $a$ and $b$, we conclude that there exist a critical point $(a_0,b_0)$ of $F$ with $F(a_0,b_0)=0$, positive integers $h,k$ and $a_1,b_1\neq 0$ such that $a = a_0 + a_1 t^h +O(t^{h+1})$ and $b = b_0 + b_1 t^k +O(t^{k+1})$. In fact, since with our conventions $\dot{a},\dot{b}>0$ for $t>0$, we must have $a_1,b_1>0$.  

We now consider the metric $g_{\varphi}=dt^2 + g_t$ induced by $\varphi$. Regard $M$ as a cohomogeneity one manifold with group diagram
\[
K_0 \subset K \subset \sunitary{2}\times\sunitary{2},
\] 
where $Q=\sunitary{2}^2/K$, $K_0$ is a finite subgroup of $\sunitary{2}^2$ and $K/K_0$ is a sphere. As $t\ra 0$, $g_t$ converges to a smooth metric $g_0$ on the singular orbit $Q$. Moreover, thinking of $g_0$ as a symmetric endomorphism of $\Lie{su}_2\oplus\Lie{su}_2$, the kernel of $g_0$ coincides with the Lie algebra $\Lie{k}$ of $K$. By studying the behaviour of $g_t$ as $t\ra 0$ we can therefore determine the possibilities for $\Lie{k}$ and therefore the group diagram of $M$ up to finite quotients.

Denote by $\Lie{n}$, $\Lie{n}'$ and $\Lie{t}$ the subspaces of $\Lie{su}_2\oplus\Lie{su}_2$ defined by $\tu{span}(E_1,E_2)$, $\tu{span}(E'_1,E'_2)$ and $\tu{span}(E_3,E'_3)$ respectively. Since $K/K_0$ is a sphere and $K_0$ is finite, we deduce that $\Lie{k}$ cannot contain $\Lie{n}\oplus\Lie{n}'$ nor $\Lie{t}$. Indeed, if $\Lie{n}\oplus\Lie{n}'\subseteq\Lie{k}$ then $K=\sunitary{2}\times\sunitary{2}$ (since $\Lie{n}\oplus\Lie{n}'$ generates $\Lie{su}_2\oplus\Lie{su}_2$ as a Lie algebra) and $K/K_0$ cannot be a sphere; similarly, if $\Lie{t}\subseteq\Lie{k}$ then $K$ is diffeomorphic to a $2$-torus, $S^3\times S^1$ or $S^3\times S^3$, none of which finitely covers a sphere.

Now, in order to study the behaviour of $g_t$ for small $t\geq 0$, regard it as a symmetric endomorphism of $\Lie{su}_2\oplus\Lie{su}_2$ and note that the decomposition $\Lie{su}_2\oplus\Lie{su}_2 = (\Lie{n}\oplus\Lie{n}')  \oplus \Lie{t}$ is $g_t$--orthogonal. By \eqref{eq:Metric} the restriction of $g_t$ to the first factor $\Lie{n}\oplus\Lie{n}'$ is the block matrix 
\[
\frac{2\dot{a}}{\sqrt{F}} \left( \begin{array}{cc}
a(b-p)  & -\tfrac{1}{2} (b^2 +pq) \\
-\tfrac{1}{2} (b^2 +pq) & a(b+q)\end{array}
\right),
\]
and the restriction of $g_t$ to $\Lie{t}$ is
\[
\frac{2\dot{b}}{\sqrt{F}} \left( \begin{array}{cc}
a^2 -pb  & -\tfrac{1}{2} (2a^2 -b^2 +pq) \\
-\tfrac{1}{2} (2a^2 -b^2 +pq) & a^2 + qb\end{array}
\right).
\]

Now, consider first the case where $pq(p+q)\neq 0$. In this case $F$ has only two critical points contained in the level set $F=0$, $(0,\pm \sqrt{-pq})$. In fact, since $\dot{a},\dot{b}>0$ for $t>0$ we must have $b_0=\sqrt{-pq}$, $p<0$ and $q>0$. Indeed, since $a(0)=0$, the sign constraint $\dot{a}>0$ for $t>0$, forces the same sign constraint for $a$. Then the positive definiteness of $g_t$ for $t>0$ forces $-pb_0, qb_0, b_0-p, b_0+q>0$. 

Using $a = a_1 t^h +O(t^{h+1})$ and $b = \sqrt{-pq} + b_1 t^k +O(t^{k+1})$, we now calculate $F=O(t^{m})$ where $m\geq 2h$. Indeed, $h\leq k$ since $F>0$ for $t>0$. Moreover, $m>2h$ if and only if $h=k$ and $a_1, b_1$ are appropriately chosen so that the coefficient of $t^{2h}$ in $F$ vanishes.

At leading order in $t$ as $t\ra 0$, the restriction of $g_t$ to $\Lie{t}$ takes the form
\[
c\, t^{k-1-\frac{m}{2}} \left( \begin{array}{cc} |p| & \sqrt{-pq} \\ \sqrt{-pq}  & |q| \end{array}\right) 
\]
for some $c>0$, while the restriction of $g_t$ to $\Lie{n}\oplus\Lie{n}'$ is of the form
\[
\left( \begin{array}{cc} O(t^{2h-1-\frac{m}{2}}) & O(t^{h+k-1-\frac{m}{2}}) \\ O(t^{h+k-1-\frac{m}{2}})  & O(t^{2h-1-\frac{m}{2}}) \end{array}\right).
\]
Since the kernel $\Lie{k}$ of $g_0$ cannot contain $\Lie{t}$ nor $\Lie{n}\oplus\Lie{n}'$, we deduce that $k-1-\frac{m}{2} = 2h-1-\frac{m}{2} =0$, \ie $h=1, k=2$. Then $\Lie{k}$ is one dimensional, spanned by $\sqrt{|q|}E_3 - \sqrt{|p|}E'_3$. The orbit of this vector field in $\sunitary{2}^2$ is closed if and only if $\sqrt{-\frac{p}{q}}\in\Q$. If this is the case there exist relatively prime positive integers $m,n$ and $r_0>0$ such that $K=K_{m,n}$ up to finite quotients and $p=-m^2 r_0^3$, $q=n^2 r_0^3$.

Finally, since $K=K_{m,n}$ is a circle we must also argue that the principal orbit stabiliser $K_0$ is $K_{m,n}\cap K_{2,-2}$. This is a consequence of the proof of Proposition \ref{prop:Solutions:Singular:Orbit} (iv). Indeed, the proposition parametrises all smooth solutions $(x_1,x_2,y_1,y_2)$ to \eqref{eq:Fundamental:ODE:U(1)} with $p=-m^2 r_0^3$ and $q=n^2 r_0^3$ satisfying $x_1=O(t)$, $x_2=r_0^4\beta +O(t^2)$, $y_1 =O(t)$ and $y_2 = mnr_0^3+O(t^2)$. We have already established that $a=a_1 t + O(t^2)$ and $b=mnr_0^3 + b_1 t^2 +O(t^3)$ for some $a_1,b_1>0$. Hence $(x_1=\dot{a}\dot{b}, x_2=\dot{a}^2, y_1=a, y_2=b)$ coincides with one of the solutions of Proposition \ref{prop:Solutions:Singular:Orbit} (iv).

We now briefly indicate the changes to the proof in the case where $pq(p+q)=0$. If $p+q=0$ the critical locus of $F$ contained in the zero-level set is $\{ (a_0,p), a_0\in\R\} \cup \{ (0,-p) \}$. Consider first the $1$-dimensional component. Boundedness of the restriction of $g_t$ to $\Lie{t}$ as $t\ra 0$ forces $a_0=\pm p$. If $p\neq 0$, consideration of the behaviour of the restriction of $g_t$ to $\Lie{n}\oplus\Lie{n}'$ as $t\ra 0$ then implies that $p>0$ and $\Lie{k}=\triangle\Lie{su}_2$. If $p=0$, one shows instead that it is impossible to find $h,k\geq 1$ so that $g_t$ remains bounded as $t\ra 0$ and $\Lie{k}$ contains neither $\Lie{n}\oplus\Lie{n}'$ nor $\Lie{t}$. The case $p\neq 0$ and $(a_0,b_0)=(0,-p)$ is analysed exactly as in the case $pq(p+q)\neq 0$. When $p\neq 0$ and $q=0$ (the case $p=0, q\neq 0$ can be reduced to this by acting with the outer automorphism of $\sunitary{2}^2$), $F$ has a unique critical point, $(0,0)$, on its zero-level set. Analysis of the behaviour of the restriction of $g_t$ to $\Lie{n}\oplus\Lie{n}'$ as $t\ra 0$ forces $\Lie{n}\oplus\{ 0\}\subseteq \Lie{k}$. The only possibility for $\Lie{k}$ is then $\Lie{su}_2\oplus\{ 0\}$. Finally, since $K$ is $3$-dimensional in all these cases, $K_0$ is automatically trivial. 
\endproof
\end{theorem}

\begin{remark}\label{rmk:Rigidity:AC}
In particular, we deduce a strong rigidity and uniqueness result for the AC \gtmetric s of Theorem \ref{thm:ALC:C7:m:n} (ii). Indeed, by \cite{Karigiannis:Lotay}*{Propositions 6.3 and 6.8} any complete AC \gtmfd~$(M,g)$ asymptotic to the cone over $S^3\times S^3/\Z_{2(n+m)}$ must be $\sunitary{2}\times\sunitary{2}\times\unitary{1}$--invariant. By Theorem \ref{thm:Classification:U(1)}, up to a finite quotient, any such metric has group diagram either \eqref{eq:group:diagram:D7}, \eqref{eq:group:diagram:C7(m,n)} and \eqref{eq:group:diagram:B7}. The proof of Theorems \ref{thm:ALC:B7}, \ref{thm:ALC:D7} and \ref{thm:ALC:C7:m:n} then shows that $(M,g)$ is either a finite quotient of the Bryant--Salamon AC metric on $S^3\times\R^4$ (if $\Z_{2(m+n)}$ acts freely on $S^3\times\R^4$) or one of the AC metrics of Theorem \ref{thm:ALC:C7:m:n} (ii). 
\end{remark}

In the rest of the section we prove Theorem \ref{thm:ALC:C7:m:n}. We first establish part (ii) of the theorem and the existence of the critical value $\beta_\tu{ac}$, which is not explicit. Our strategy is to consider the AC ends constructed in Proposition \ref{prop:CS:AC:ends} (ii) and study which of these extend \emph{backward} to close smoothly on the singular orbit $\sunitary{2}\times\sunitary{2}/K_{m,n}$. Once part (ii) of Theorem \ref{thm:ALC:C7:m:n} is established, a comparison argument with the AC solution will let us obtain the existence of the ALC \gtmetric s in part (i) and the incompleteness result in part (iii).

\subsection{Extending AC ends backwards}\label{sec:AC:back}

Fix a pair of positive coprime integers $m,n$ and $r_0\in\R$ and set $p=-m^2 r_0^3$, $q=n^2 r_0^3$. We consider pairs of functions $(a,b)$ satisfying the ODE \eqref{eq:Fundamental:ODE:Brandhuber:U(1)}, \ie
\[
2F \left( \dot{a} \ddot{b} - \dot{b} \ddot{a} \right) = -\dot{a}\dot{b} \left( 2\dot{b} F_b - \dot{a} F_a \right),
\]
where
\begin{equation}\label{eq:mn:equation}
F=4a^2 (b+m^2 r_0^3)(b+n^2r_0^3) - (b^2-m^2 n^2 r_0^6)^2.
\end{equation}

By Proposition \ref{prop:CS:AC:ends} (ii) for each $c\in\R$ there exists a solution $(a,b)$ corresponding to an AC end asymptotic to the cone over the homogeneous nearly K\"ahler structure on $\tu{S}^3\times\tu{S}^3$. We now consider the problem of extending these AC ends backwards, \ie to decreasing values of the parameter $t$.

\begin{prop}
\label{prop:evolve_back}
Suppose that $r_0 \geq 0$ and $(a, b)$ is a solution to \eqref{eq:Fundamental:ODE:Brandhuber:U(1)} satisfying
\begin{subequations}\label{eq:acdomain}
\begin{equation}\label{eq:acdomain:1}
\dot{a},\dot{b}>0, \qquad a>0, \qquad b>\max\left( -m^2r_0^3, -n^2 r_0^3\right), \qquad F>0
\end{equation}
and
\begin{equation}
\label{eq:acdomain:2}
b > a, \quad \dot{a} > \dot{b} > 0
\end{equation}
\end{subequations}
at some time $t_0$ (or equivalently, since \eqref{eq:acdomain} are open conditions, on an open interval $(t_1,t_2)$ of existence). Then the solution extends backwards in time, with the conditions
\eqref{eq:acdomain} persisting, until $F(a,b)\ra 0$, \ie until
\[
2a - \frac{|b^2 - m^2n^2r_0^6|}{\sqrt{(b+m^2r_0^3)(b+n^2r_0^3)}} \longrightarrow 0 .
\]
\end{prop}

\begin{proof}
Set $\mu := \frac{\dot{b}}{\dot{a}}$. As at the beginning of the proof of Proposition \ref{prop:ALC:growth:a:b}, we rewrite \eqref{eq:Fundamental:ODE:Brandhuber:U(1)} in the form
\begin{equation}
\label{eq:ac_ode}
2F \dot{\mu} = \mu \dot{a} ( F_a - 2\mu F_b).
\end{equation}
The coefficient of $\dot{\mu}$ on the left-hand side of the equation is positive. On the other hand, thanks to our hypotheses \eqref{eq:acdomain:2}, $\mu$ takes values in the interval $(0,1)$ and therefore $F_a - 2\mu F_b$ is greater than the minimum of $F_a$ and $F_a - 2F_b$. Using \eqref{eq:acdomain:1} one can check that the former is always positive and the latter is a concave function
of $a$ which is non-negative both for $a = 0$ and $a = b$.
Thus $F_a - 2\mu F_b$ is positive whenever $b>a>0$.
We conclude that $\dot{\mu} > 0$.

Now, since $\dot{\mu} > 0$, the inequality $\mu < 1$ is preserved as we evolve backwards, and hence $b > a$ is also preserved. It remains to prove that we can extend backwards until $F(a,b)\ra 0$.

Since \eqref{eq:Fundamental:ODE:Brandhuber:U(1)} is equivalent to Hitchin's flow \eqref{eq:Fundamental:ODE:U(1)} for the $4$-tuple $(\dot{a}\dot{b},\dot{a}^2,a,b)$, it is clear that solutions fail to extend only when one of $a,b,\dot{a},\dot{b}$ diverge or when one of the inequalities $\dot{a},\dot{b},F(a,b)>0$ fails to be satisfied. Now, for any $M>0$, the curves $\{ a=b\}$, $\{ F(a,b)=0\}$ and $\{ a^2 + b^2 =M^2\}$ bound a compact region $\mathcal{R}\subset \{ F\geq 0\}$ in the first quadrant in the $(a,b)$--plane. By assumption, the curve $(a,b)$ lies in $\mathcal{R}$ for some $M>0$ and therefore the solution can be extended backward until $F(a,b)\ra 0$ provided we control $\dot{a},\dot{b}$. Since $2\dot{a}^2\dot{b}=\sqrt{F(a,b)}$ and $\dot{b} < \dot{a}$ it is enough to prove that $\dot{b}$ is bounded away from zero until $F(a,b)\ra 0$.

Now, since $\dot{a}>0$, we can reparametrise so that $\dot{a} = 1$. Then \eqref{eq:ac_ode} becomes
\[ \frac{d \log \dot{b}}{da}
= \frac{F_a - 2\mu F_b}{2F} \]
The right-hand side can only blow up as $F(a,b)\ra 0$, so until then
$\log \dot{b}$ remains bounded. Thus $\dot{b}$ is bounded away from zero until $F(a,b)\ra 0$.
\end{proof}

The solutions $(a,b)$ of \eqref{eq:Fundamental:ODE:Brandhuber:U(1)} constructed in Proposition \ref{prop:CS:AC:ends} (ii) always satisfy \eqref{eq:acdomain:1} in the interior of a maximal interval of existence. Moreover, the solutions of Proposition \ref{prop:CS:AC:ends} (ii) satisfy
\[
b-a \approx \tfrac{\sqrt{3}}{54} c\, t^{3-\nu_\infty} 
\]
as $t\ra\infty$ for some $c\in\R$. When $c$ is positive, \eqref{eq:acdomain:2} will then also hold for $t$ sufficiently large. Hence we can apply Proposition \ref{prop:evolve_back} to conclude that the solutions constructed in Proposition \ref{prop:CS:AC:ends} (ii) with $c>0$ extend backward until $F(a,b)\ra 0$. In the limiting case $c=0$, the uniqueness statement in Proposition \ref{prop:CS:AC:ends} (ii) implies that $a=b$.

We will prove that there exists $c_\tu{ac}>0$ such that, after taking the quotient of the principal orbits by $\Z_{2|m+n|}$, the AC {\gtmetric} corresponding to the solution $(a_\tu{ac},b_\tu{ac})$ constructed in Proposition \ref{prop:CS:AC:ends} (ii) with $c=c_\tu{ac}$ extends smoothly over a singular orbit $\sunitary{2}\times\sunitary{2}/K_{m,n}$. The following lemma will be used to show that there exists a unique such value $c_\tu{ac}$.

\begin{lemma}\label{lem:nocross}
Fix a pair of positive coprime integers $m$ and $n$ and suppose that $r_0 \geq 0$. Let $(a_1,b_1)$ and $(a_2,b_2)$ be solutions of \eqref{eq:Fundamental:ODE:Brandhuber:U(1)} satisfying
\begin{equation}\label{eq:nocross:boundary}
b>\max\left( a, mnr_0^3\right), \qquad F(a,b)>0, \qquad \dot{a},\dot{b}>0.
\end{equation}
Parametrise the curves $(a_1,b_1)$ and $(a_2,b_2)$ so that $a_1(s)=s=a_2(s)$.
\begin{enumerate}[leftmargin=*]
\item The inequalities $b_1>b_2$, $\dot{b}_1<\dot{b}_2$ are preserved evolving $(a_1,b_1)$ and $(a_2,b_2)$ backwards until either solution hits the boundary of the region defined by the inequalities \eqref{eq:nocross:boundary}.
\item The inequalities $b_1<b_2$, $\dot{b}_1<\dot{b}_2$ are preserved evolving $(a_1,b_1)$ and $(a_2,b_2)$ forward until either solution hits the the boundary of the region defined by the inequalities \eqref{eq:nocross:boundary}.
\end{enumerate}
\proof
Let $(a,b)$ be a solution \eqref{eq:Fundamental:ODE:Brandhuber:U(1)} parametrised so that $\dot{a}=1$. Then \eqref{eq:Fundamental:ODE:Brandhuber:U(1)} can be rewritten as
\begin{equation}\label{eq:nocross:ODE}
2\ddot{b} = \left( \frac{F_a}{F}-2\dot{b}\frac{F_b}{F} \right) b\dot{b} .
\end{equation}

In order to compare two solutions $(a_1,b_1)$ and $(a_2,b_2)$ we now observe that, for each fixed $a>0$, $\frac{F_a}{F}$ and $-\frac{F_b}{F}$ are strictly increasing functions of $b$ on the range defined by the inequalities $b > \max\left( a, mnr_0^3\right)$ and $F(a,b)>0$. Indeed, the function
\[ \frac{F_a}{F} = \frac{8a}{4a^2 - \frac{(b^2-m^2n^2r_0^6)^2}{(b+m^2r_0^3)(b+n^2r_0^3)}} \]
is increasing in $b$ if and only if
\[ \frac{(b^2-m^2n^2r_0^6)^2}{(b+m^2r_0^3)(b+n^2r_0^3)}
= (b-mnr_0^3)^2
\frac{(b+mnr_0^3)^2}{(b+mnr_0^3)^2 + (m-n)^2 r_0^3 b}
\]
is. Each factor on the right-hand side is increasing precisely when $b > mnr_0^3$. Meanwhile $F_{bb} = 4(2a^2 - 3b^2 + m^2n^2r_0^6) < 0$ in the given range, while $F$ is a priori positive. Thus
\[
-\frac{d}{db}\left( \frac{F_b}{F}\right) = \frac{-F_{bb}F + F_b^2}{F^2} > 0.
\]

Going back to \eqref{eq:nocross:ODE}, we now conclude that at every point where $\dot{b}_1 = \dot{b}_2$, $\ddot{b}_1 - \ddot{b}_2$ has the same sign as $b_1-b_2$. Hence the inequalities $b_1 > b_2$ and
$\dot{b}_1 < \dot{b}_2$ ($b_1 < b_2$ and
$\dot{b}_1 < \dot{b}_2$) are preserved as we evolve backwards (forwards) as long as both solutions remain in the region defined by the inequalities \eqref{eq:nocross:boundary}.
\endproof
\end{lemma}

\begin{prop}
\label{prop:uniquehit} 
For each $r_0>0$ there exists a unique $c_\tu{ac}>0$  such that the solution of \eqref{eq:Fundamental:ODE:Brandhuber:U(1)} constructed in Proposition \ref{prop:CS:AC:ends} (ii) with $c=c_\tu{ac}$ extends backwards until $a \ra 0$ and $b \ra mnr_0^3$.
Moreover, for any $k \in (1,2)$ the solution satisfies
$ka > \frac{|b^2 - m^2n^2r_0^6|}{\sqrt{(b+m^2r_0^3)(b+n^2r_0^3)}}$ and $b>mnr_0^3$ whenever $(a,b)\neq (0,mnr_0^3)$.
\end{prop}

\begin{proof}
Fix $k\in (1,2)$. For any $c \geq 0$, evolve
the AC end solution of Proposition \ref{prop:CS:AC:ends} (ii) backwards until we hit the curve $\gamma$ which is the union of the curves 
\[\gamma_1=\{ b = mnr_0^3\} \quad  \text{and} \quad 
\gamma_2 = \left\{ ka = \frac{|b^2 - m^2n^2r_0^6|}{\sqrt{(b+m^2r_0^3)(b+n^2r_0^3)}}\right\}.
\]
Since $\gamma$ is smooth except at the intersection point $(0,mnr_0^3)$ of $\gamma_1$ and $\gamma_2$, the intersection point of the trajectory $(a,b)$ with $\gamma$ depends continuously on $c$ unless $(a,b)$ approaches $(0, mnr_0^3)$.

The uniqueness statement in Proposition \ref{prop:CS:AC:ends} (ii) implies that the solution with $c = 0$ satisfies $a = b$ and therefore it must hit $\gamma$ along the segment $\gamma_1$. In order to analyse the behaviour of the intersection point as $c \to \infty$, we can rescale so that $r_0 \to 0$ and $c>0$ is fixed. This solution must hit the half-line $ka = b, b>0$ away from $(0,0)$ since $a-b$ is strictly increasing for all time and strictly negative along the AC end.

We conclude that for each $r_0>0$ the set $S$ of $c\geq 0$ for which we hit the segment $\gamma_1$
is non-empty and bounded. Set $c_\tu{ac}=\sup{S}$. The solution corresponding to $c=c_\tu{ac}$ can hit neither $\gamma_1$ nor $\gamma_2$, so by Proposition \ref{prop:evolve_back} the
only possibility is that it approaches $(0, mnr_0^3)$.

The uniqueness part of the claim follows from Lemma \ref{lem:nocross} (i). Indeed, if $(a,b)$ is a solution of \eqref{eq:Fundamental:ODE:Brandhuber:U(1)} given by Proposition \ref{prop:CS:AC:ends} (ii) then
\[
a \approx \tfrac{\sqrt{3}}{54}t^3, \qquad b-a \approx \tfrac{\sqrt{3}}{54} c\, t^{3-\nu_\infty}
\]
for some $c\in\R$. Hence if we parametrise $(a,b)$ so that $a(s)=s$ we have
\[
b\approx s + \left( \frac{\sqrt{3}}{54}\right)^{\frac{\nu_\infty}{3}} c\, s^{\frac{3-\nu_\infty}{3}}
\]
as $s\ra \infty$. Since $3-\nu_\infty\approx -6.5$, if $(a_1,b_1)$ and $(a_2,b_2)$ are two solutions corresponding to $c_1>c_2>0$ then
\[
b_1 > b_2, \qquad \dot{b}_1<\dot{b}_2
\]
for large $s$.
\end{proof}

In order to conclude the proof of Theorem \ref{thm:ALC:C7:m:n} (ii) we must show that after a $\Z_{2(m+n)}$--quotient the AC solution with $c=c_\tu{ac}$ singled out by
Proposition \ref{prop:uniquehit} extends smoothly across the singular orbit $\sunitary{2}\times\sunitary{2}/K_{m,n}$.

\begin{prop}
\label{prop:even}
Fix $k\in (1,2)$. Consider a solution $(a,b)$ satisfying \eqref{eq:acdomain}, $b>mnr_0^3$ and
\begin{equation}\label{eq:acdomain:k}
ka > \frac{(b^2 - m^2n^2r_0^6)}{\sqrt{(b+m^2r_0^3)(b+n^2r_0^3)}}.
\end{equation}
Assume that $(a,b)\ra (0,mnr_0^3)$. Then $(a,b)$ defines a cohomogeneity one torsion-free {\gtstr} with principal orbits $\sunitary{2}\times\sunitary{2}/\Z_{2(m+n)}$ and extending smoothly on the singular orbit $\sunitary{2}\times\sunitary{2}/K_{m,n}$. 
\end{prop}
\begin{proof}
We must show that $(a,b)$ satisfies the boundary conditions of Proposition \ref{prop:Smooth:extension:singular:orbit} (iii).

First of all note that \eqref{eq:acdomain:k} implies that the function $F$ in \eqref{eq:mn:equation} satisfies
\begin{subequations}\label{eq:AC:bound:RHS}
\begin{equation}
(4-k^2) (b+m^2r_0^3)(b+n^2r_0^3) a^2 \leq F(a,b) \leq 4 (b+m^2r_0^3)(b+n^2r_0^3) a^2.
\end{equation}
Differentiating the expression in \eqref{eq:mn:equation} and using \eqref{eq:acdomain:k} and (\ref{eq:AC:bound:RHS}a) we also obtain
\begin{equation}
\sqrt{(b+m^2r_0^3)(b+n^2r_0^3)} \leq \frac{F_a}{4\sqrt{F}} \leq \tfrac{2}{\sqrt{4-k^2}}\sqrt{(b+m^2r_0^3)(b+n^2r_0^3)}
\end{equation}
and
\begin{equation}
-2k b (b+mnr_0^2)  \leq \frac{F_b}{2\sqrt{F}} \leq \frac{2}{\sqrt{4-k^2}}\frac{\left( 2b+\left( m^2+n^2\right) r_0^3\right) a }{\sqrt{(b+m^2r_0^3)(b+n^2r_0^3)}}.
\end{equation}
\end{subequations}

Reparametrise so that $\dot{a} = 1$ and consider the first-order equation \eqref{eq:ac_ode} for $\mu = \frac{db}{da}$. Since $\frac{dF}{da}=F_a + \mu F_b$, we find
\[
\frac{d}{da}\left( \frac{\sqrt{F}}{\mu}\right) = \frac{\sqrt{F}}{\mu} \frac{3\mu F_b}{2F} \leq C\frac{\sqrt{F}}{\mu}
\]
for some $C>0$. Indeed, $\frac{F_b}{F}$ is bounded above by (\ref{eq:AC:bound:RHS}a) and (\ref{eq:AC:bound:RHS}c) and $0<\mu<1$. Thus the logarithm of $\frac{\sqrt{F}}{\mu}$ is a function of $a$ whose derivative is bounded above. It follows that $\frac{\sqrt{F}}{\mu}$ is bounded below away from zero as $a\ra 0$.

Now change parametrisation and introduce the arc-length parameter $t$ along a geodesic meeting all principal orbits orthogonally. Recall that $t$ is defined by the normalisation $2\dot{a}^2\dot{b}=\sqrt{F(a,b)}$. By (\ref{eq:AC:bound:RHS}a) and the hypothesis $\dot{b}<\dot{a}$ in \eqref{eq:acdomain:2} we have
\[
Ca\leq \sqrt{F(a,b)}=2\dot{a}^2\dot{b} < 2\dot{a}^3
\]
for some $C>0$. Since the function $a\mapsto a^{-\frac{1}{3}}$ is integrable near $a=0$, $t$ has a finite limit as $(a,b)\ra (0,mnr_0^3)$. By a time translation we can assume that $t\ra 0$ as $(a,b)\ra (0,mnr_0^3)$. Moreover, since
\[
\frac{\sqrt{F}}{\mu}=\frac{2\dot{a}^2\dot{b}}{\mu} = 2 \dot{a}^3
\]
and $\frac{\sqrt{F}}{\mu}$ is bounded below away from zero, so is $\dot{a}$.

Consider now the ODE system \eqref{eq:Fundamental:ODE:U(1)} for $x_1=\dot{a}\dot{b}$, $x_2 = \dot{a}^2$, $y_1=a$, $y_2=b$. By (\ref{eq:AC:bound:RHS}b) and (\ref{eq:AC:bound:RHS}c), $\dot{x}_1$ and $\dot{x}_2$ remain bounded and therefore $x_1$ and $x_2$ have a well-defined limit as $t\ra 0$. Since $2\dot{b}^3\leq 2\dot{a}^2 \dot{b}=\sqrt{F}\ra 0$ as $t\ra 0$ and $\dot{a}$ is bounded away from zero, $(x_1, x_2)\ra (0,r_0^4 \beta_\tu{ac} ^2)$ as $t\ra 0$ for some $\beta_\tu{ac}>0$. Then the right-hand side of \eqref{eq:Fundamental:ODE:U(1)} is bounded as $t\ra 0$ and a bootstrap argument shows that $x_1,x_2,y_1,y_2$ are smooth functions of $t$ up to $t=0$. Uniqueness of the solutions in Proposition \ref{prop:Solutions:Singular:Orbit} (iii) and (iv) then forces $(a,b)$ to satisfy the conditions of Proposition \ref{prop:Smooth:extension:singular:orbit} (iii).
\end{proof}

\subsection{Existence of ALC metrics}\label{sec:ALC:mn}

We now prove Theorem \ref{thm:ALC:C7:m:n} (i). Theorem \ref{thm:ALC:C7:m:n} (ii) guarantees the existence of $\beta_\tu{ac}>0$ such that the $\sunitary{2}\times\sunitary{2}\times\unitary{1}$--invariant local solution $(a_\tu{ac},b_\tu{ac})$ of Proposition  \ref{prop:Solutions:Singular:Orbit} (iii) (when $m=n=1$) or (iv) with
\[
a_\tu{ac} = r_0^2\beta_\tu{ac} t + O(t^3), \qquad b_\tu{ac} = mn r_0^3 + O(t^2)
\]
exists for all time $t\geq 0$ and gives rise to a complete AC metric. From the proof of Theorem \ref{thm:ALC:C7:m:n} (ii) we also know that
\[
b_\tu{ac} >\max \left( a_\tu{ac}, mnr_0^3\right), \qquad \dot{a}_\tu{ac} > \dot{b}_\tu{ac}>0
\] 
for all $t>0$.

Consider now one of the local solutions $(a,b)$ of Proposition \ref{prop:Solutions:Singular:Orbit} (iii) and (iv) with $\beta>\beta_\tu{ac}$. We want to show that $(a,b)$ eventually satisfies the constraints \eqref{eq:ALC:Chamber} and \eqref{eq:ALC:Chamber:additional:constraint}, \ie
\[
a>b>mnr_0^3, \qquad \dot{a}>\dot{b}, \qquad a\dot{b}-\dot{a}b<0,
\]
so that Propositions \ref{eq:ALC:Forward:Completeness} and \ref{prop:ALC:growth:a:b} can be applied to guarantee that the solution $(a,b)$ is immortal and gives rise to an ALC end as $t\ra\infty$. In fact, it is enough to show that the solution $(a,b)$, initially contained in the region
\[
a<b, \qquad \dot{a}>\dot{b},
\]
will intersect the line $a=b$ with the condition $\dot{a}>\dot{b}$ preserved. For \eqref{eq:ALC:Chamber} and \eqref{eq:ALC:Chamber:additional:constraint} will then be satisfied immediately after the intersection time.

First of all, note that $a>0$ and $b>mnr_0^3$ are preserved as long as the solution exists. From Remark \ref{rmk:Solutions:Singular:Orbit:C7} we know that $a=r_0^2 \beta t + O(t^3)$ and $b=mnr_0^3 + \frac{\sqrt{mn}(m+n)r_0}{2\beta}t^2 + O(t^4)$ as $t\ra 0$. In order to compare $(a,b)$ and $(a_\tu{ac},b_\tu{ac})$, we now reparametrise both solutions so that $a(s)=s=a_\tu{ac}(s)$. Then
\begin{equation}\label{eq:comparison:ac}
b = mnr_0^3 + \frac{\sqrt{mn}(m+n)}{2\beta^3 r_0^3}s^2 + O(s^4), \qquad b_\tu{ac} = mnr_0^3 + \frac{\sqrt{mn}(m+n)}{2\beta_\tu{ac}^3 r_0^3}s^2 + O(s^4)
\end{equation}
as $s\ra 0$. In particular, if $\beta >\beta_\tu{ac}$ we have
\[
b<b_\tu{ac}, \qquad \dot{b}<\dot{b}_\tu{ac}
\]
for $s>0$ sufficiently small. By Lemma \ref{lem:nocross} (ii) these conditions are preserved as long as $b>\max \left( a,mnr_0^3\right)$.

Now, on the one hand the solution $(a,b)$ certainly exists as long as $b>a>0$, because then $(a,b)$ is bounded to stay in the region $\{ b_\tu{ac} > b >a, b>mnr_0^3 \}$ where \eqref{eq:Fundamental:ODE:Brandhuber:U(1)} cannot blow-up. On the other hand, since $b_\tu{ac} \approx a_\tu{ac}=s=a(s)$ for large $s$, for all $\epsilon>0$ there exists $s_0$ such that $b_\tu{ac}<a+\epsilon$ for all $s\geq s_0$. Moreover, as long as $b>\max \left(a,mnr_0^3\right)$, $b-b_\tu{ac}$ is strictly decreasing and therefore bounded above by a definite constant $-\delta_0<0$ depending only on $\beta-\beta_\tu{ac}$. Hence if $b>a$ for all $s\in [0,2s_0]$ we would have
\[
b-a - \epsilon < b-b_\tu{ac} < -\delta_0<0
\]
for all $s\in [s_0,2s_0]$. If $\epsilon$ is chosen small enough we reach a contradiction. We conclude that $(a,b)$ must intersect the boundary of the region $\{ b_\tu{ac} > b >a, b>mnr_0^3 \}$. By Lemma \ref{lem:nocross} (ii) the only possibility is that $(a,b)$ intersects the line $a=b$. Moreover since we must have
\[
\dot{b}\leq \dot{b}_\tu{ac}<1
\]
up to and including the intersection time, at the intersection time we also have $\dot{b}<\dot{a}$.

An application of Propositions \ref{eq:ALC:Forward:Completeness} and \ref{prop:ALC:growth:a:b} now concludes the proof of Theorem \ref{thm:ALC:C7:m:n} (i).

\subsection{Incompleteness results}\label{sec:ALC:mn:in}

We now conclude the proof of Theorem \ref{thm:ALC:C7:m:n} establishing part (iii).

Let $(a,b)$ be one of the local solutions constructed in Proposition \ref{prop:Solutions:Singular:Orbit} (iii) and (iv) with $\beta<\beta_\tu{ac}$. Assume for a contradiction that $(a,b)$ yields a complete \gtmetric. We want to show that the constraints \eqref{eq:death:quadrant}, \ie
\[
0<\frac{\dot{a}}{\dot{b}}<\frac{a}{b}<1,
\]
are eventually satisfied. For Proposition \ref{prop:death:quadrant} would then give a contradiction.

We compare $(a,b)$ and $(a_\tu{ac},b_\tu{ac})$: parametrise both solutions so that $a(s)=s=a_\tu{ac}(s)$ and use \eqref{eq:comparison:ac} to conclude that since $\beta<\beta_{ac}$ we have
\[
b>b_{ac}, \qquad \dot{b}>\dot{b}_{ac}
\]
by Lemma \ref{lem:nocross} (ii). Since $b_{ac}>a_{ac}=s$ we already conclude that $b>a=s$ for all time. 

We will also need to know that $b$ is unbounded and exists for all $s\geq 0$. For this note that $2\dot{b}=\sqrt{F}$ is the orbital volume function, which must be bounded below if $(a,b)$ is complete. Hence $b$ is unbounded. By \eqref{eq:a:unbounded}, the condition $F>0$ forces the ratio $\frac{a}{b}=\frac{s}{b}$ to be bounded below and therefore $s\ra \infty$ as $b\ra\infty$.

Now consider the quantities $R=s\dot{b}-b$ and $R_{ac}=s\dot{b}_{ac}-b_{ac}$ and, taking into account the parametrisation $a(s)=s$, note that $R>0$ is equivalent to
\[
\frac{\dot{a}}{\dot{b}}<\frac{a}{b}.
\]
By \eqref{eq:comparison:ac}, $R$ and $R_{ac}$ are both initially strictly negative, but $R_{ac}\ra 0$ along the AC end. We make two further crucial observations about $R$ and $R_{ac}$:
\begin{enumerate}
\item $R-R_{ac}$ is strictly increasing: indeed, $\dot{R}-\dot{R}_{ac}=s\left( \ddot{b}-\ddot{b}_{ac}\right)$ and $\ddot{b}-\ddot{b}_{ac}>0$ for $b,b_{ac}>s$ is proved in Lemma \ref{lem:nocross};
\item $R-R_{ac}>0$ for small $s>0$, by \eqref{eq:comparison:ac}.
\end{enumerate}
Hence by (i) and (ii), $R - R_{ac}$ has a strictly positive lower bound. Since $R_{ac}$ converges to zero as $s\ra\infty$, we conclude that $R$ must eventually be positive.

\bibliographystyle{amsinitial}
\bibliography{Coho1_ALC}

\end{document}